\documentclass[11pt,oneside,english,final]{amsart} 

\usepackage[notcite,notref,color]{showkeys}  

\usepackage{amstext}
\usepackage{amsthm,amssymb}
\usepackage[T1]{fontenc}
\usepackage[latin1]{inputenc}
\usepackage{geometry}
\usepackage{graphicx}
\usepackage{mathrsfs}

\usepackage[colorlinks,bookmarks,linkcolor=black,citecolor=black]{hyperref} 
\usepackage{enumitem}
\usepackage[english,algoruled,lined,noresetcount,norelsize]{algorithm2e}
\usepackage{bbm}

\newtheorem{theorem}{Theorem}
\newtheorem{lemma}[theorem]{Lemma}
\newtheorem{claim}[theorem]{Claim}
\newtheorem{proposition}[theorem]{Proposition}
\newtheorem{corollary}[theorem]{Corollary}

\newtheorem{observation}[theorem]{Observation}

\newtheorem{setting}[theorem]{Setting}

\newtheorem{definition}[theorem]{Definition}

\theoremstyle{remark}
\newtheorem{remark}[theorem]{Remark}

\newcommand{\oldqed}{}
\newcommand{\qedClaim}{\hfill\scalebox{.6}{$\Box$}}
\newenvironment{claimproof}[1][Proof]{
  \renewcommand{\oldqed}{\qedsymbol}
  \renewcommand{\qedsymbol}{\qedClaim}
  \begin{proof}[#1]
}{
  \end{proof}
  \renewcommand{\qedsymbol}{\oldqed}
}

\setlist{itemsep=2pt,parsep=1pt,topsep=3pt,partopsep=0pt}  
\setenumerate{leftmargin=*,labelindent=\parindent} 

\setdescription{style=sameline,leftmargin=.7cm,itemsep=2pt}

\def\itm#1{\rm ({#1})} 
\def\itmit#1{\itm{\it #1\,}} 
\def\rom{\itmit{\roman{*}}} 
\def\abc{\itmit{\alph{*}}}

\def\itmarab#1{\mbox{\itm{{\it #1\,}\arabic{*}\hspace{.05em}}}}
\def\itmarabp#1{\mbox{\itm{{\it #1\,}\arabic{*}'\hspace{.05em}}}}

\newcommand{\RandomEmbedding}{\emph{RandomEmbedding}}

\newcommand{\PackingProcess}{\emph{PackingProcess}}
\newcommand{\MatchLeaves}{\emph{MatchLeaves}}
\newcommand{\PerfectPacking}{\emph{PerfectPacking}}
\newcommand{\OrientationSwitch}{\emph{OrientationSwitch}}
\newcommand{\AlgMap}{\hookrightarrow}

\newcommand{\By}[2]{\overset{\mbox{\tiny{#1}}}{#2}}
\newcommand{\ByRef}[2]{   \By{\eqref{#1}}{#2} }

\newcommand{\eqByRef}[1]{ \ByRef{#1}{=} }

\newcommand{\leByRef}[1]{ \ByRef{#1}{\le} }

\DeclareMathOperator{\im}{im}
\DeclareMathOperator{\dom}{dom}

\newcommand{\eps}{\varepsilon}
\renewcommand{\rho}{\varrho}
\renewcommand{\subset}{\subseteq}

\newcommand{\cE}{\mathcal{E}}

\newcommand{\cG}{\mathcal{G}}
\newcommand{\cH}{\mathcal{H}}

\newcommand{\tmu}{\tilde{\mu}}
\newcommand{\trho}{\tilde{\rho}}
\newcommand{\tnu}{\tilde{\nu}}
\newcommand{\teta}{\tilde{\eta}}
\newcommand{\tsigma}{\tilde{\sigma}}
\newcommand{\oH}{\vec{H}}

\newcommand{\hist}{\mathscr{H}}
\newcommand{\histens}{\mathscr{L}}

\newcommand{\Exp}{\mathbb{E}}
\newcommand{\Prob}{\mathbb{P}}
\newcommand{\ONE}{\mathbbm{1}}
\DeclareMathOperator{\Var}{Var}

\newcommand{\DietEvent}{\mathsf{DietE}}
\newcommand{\CoverEvent}{\mathsf{CoverE}}
\newcommand{\CodietEvent}{\mathsf{CoDietE}}
\newcommand{\CapEvent}{\mathsf{CapE}}

\newcommand{\LNBH}{N^{-}}
\newcommand{\CANDSET}{C}
\newcommand{\LEFTDEG}{\deg^{-}}

\title{Perfectly packing graphs with bounded degeneracy and many leaves}

\author[P. Allen]{Peter Allen}
\address{(PA) London School of Economics, Department of Mathematics, Houghton
  Street, London WC2A 2AE, UK}
\thanks{PA was partially supported by the EPSRC, grant number EP/P032125/1.}
\email{p.d.allen@lse.ac.uk}

\author[J. B\"ottcher]{Julia B\"ottcher}
\address{(JB) London School of Economics, Department of Mathematics, Houghton Street, London WC2A 2AE, UK}
\email{j.boettcher@lse.ac.uk}
\thanks{JB was partially supported by the EPSRC, grant number EP/R00532X/1.}
\thanks{PA and JB were partially supported by a STICERD grant.}

\author[D. Clemens]{Dennis Clemens}
\address{(DC) Technische Universität Hamburg, Institut f\"ur Mathematik, Am Schwarzenberg-Campus 3, 21073 Hamburg, Germany }
\email{dennis.clemens@tuhh.de}
\author[A. Taraz]{Anusch Taraz}
\address{(AT) Technische Universität Hamburg, Institut f\"ur Mathematik, Am Schwarzenberg-Campus 3, 21073 Hamburg, Germany }
\email{taraz@tuhh.de}

\date{\today}

\begin{document}

\begin{abstract}
 We prove that one can perfectly pack degenerate graphs into complete or dense $n$-vertex quasirandom graphs, provided that all the degenerate graphs have maximum degree $o\big(\tfrac{n}{\log n}\big)$, and in addition $\Omega(n)$ of them have at most $\big(1-\Omega(1)\big)n$ vertices and $\Omega(n)$ leaves. This proves Ringel's conjecture and the Gy\'arf\'as Tree Packing Conjecture for all but an exponentially small fraction of trees (or sequences of trees, respectively).
\end{abstract}

\maketitle

\section{Introduction}

Let $\cG=\{G_1,G_2,\dots,G_s\}$ be a collection of graphs, and $H$ be a graph.
We say that \emph{$\cG$ packs into $H$} if we can find pairwise edge-disjoint
copies in $H$ of the graphs $G_1,\dots,G_s$. If in addition we have
$\sum_{i\in[s]}e(G_i)=e(H)$, we call the packing \emph{perfect}: in this case,
each edge of $H$ is used in a copy of exactly one $G_i$.


The study of perfect packings in graphs has a long history, beginning with
Pl\"ucker~\cite{Pluecker}, who in 1835 showed that for certain values of $n$
there is a perfect packing of copies of $K_3$ into $K_n$. Steiner~\cite{Steiner}
in 1853 asked, more generally, when one can perfectly pack the $n$-vertex
$k$-uniform complete hypergraph with cliques on $r$ vertices. He phrased the
question as a problem in set theory, and gave some obvious divisibility-based
necessary conditions on $n$; today such perfect packings are called
\emph{combinatorial designs}. In 1846 Kirkman~\cite{Kirkman} asked for a
strengthening of Pl\"ucker's ideas: when can one have a perfect packing of
spanning $K_3$-factors (that is, $\tfrac{n}{3}$ vertex disjoint copies of $K_3$)
into $K_n$? Again, he showed that for specific values of $n$ such a thing is
possible. Generalising this in the direction of Steiner one obtains the concept
of a \emph{resolvable design}; again, it is easy to find divisibility-based
necessary conditions on $n$.

Despite their simple statement, these problems turned out to be difficult.
Kirkman gave explicit constructions showing that one can perfectly pack $K_n$
with copies of $K_3$ if and only if $n$ is congruent to $1$ or $3$ modulo $6$.
But it took more than a century until, in 1975 Wilson~\cite{Wilson75} proved the
(much harder) statement that the necessary divisibility conditions are also
sufficient for cliques of any fixed size in large enough ($2$-uniform) graphs.
Ray-Chaudhuri and Wilson~\cite{RCW} in 1971 solved Kirkman's problem. There was
then another pause, till 2014---up to which time, despite significant work, not
a single example of a non-trivial hypergraph perfect packing for uniformity at
least $6$ was discovered---when Keevash~\cite{Kee1}, in a major breakthrough,
proved that the necessary divisibility conditions are also sufficient for any
fixed clique size and hypergraph uniformity, provided $n$ is large enough. The
problem was re-solved, by a rather different method, by Glock, K\"uhn, Lo and
Osthus~\cite{GKLO:Designs,GKLO:Fdesigns}, who also solved the problem of perfect
packings with general fixed hypergraphs in place of cliques. In~\cite{Kee2},
Keevash made the beautiful observation that a resolvable design is equivalent to
a perfect packing in an auxiliary well-structured hypergraph, and established
the existence of such a packing. Hence, he proved that resolvable designs
exist whenever the obvious necessary conditions are satisfied and $n$ is large
enough.

When one moves away from packings with fixed-size objects (or statements which
can be reduced to such packings), the first positive result is due to Walecki in
the 1800s (see~\cite{Lucas}), who proved that $K_n$ can be perfectly packed with
Hamilton cycles whenever $n$ is odd. In the 1960s and 70s, interest in this area
was renewed, in particular due to conjectures of Ringel~\cite{Ringel} and
Gy\'arf\'as~\cite{GyaLeh} on packings of trees. These conjectures state,
respectively, that $K_{2n-1}$ can be packed with $2n-1$ copies of any given
$n$-vertex tree, and that if $T_2,\dots,T_n$ is any sequence of trees such that
$v(T_i)=i$, then $\{T_2,\dots,T_n\}$ packs into $K_n$. In both cases, the
packing is necessarily perfect, which makes these conjectures difficult. It is
not too hard to prove either conjecture for stars or paths, and a considerable
amount of effort was put into solving special cases of both cases (for the
former, see the survey of Gallian~\cite{Gallian}). However until rather
recently, there were no proofs of either conjecture for any reasonably large
family of trees. Then Joos, Kim, K\"uhn and Osthus~\cite{JKKO} proved (among
other things) that both conjectures hold when the trees have constant maximum
degree $\Delta$ and $n$ is large enough. The proof of this result is very hard,
using a variety of powerful techniques from modern extremal graph theory.

Broadly, the recent solutions to perfect packing conjectures (both, in the case
of combinatorial designs and in the case of tree packing) depend on two
advances: \emph{randomised packing methods}, and the \emph{absorbing method}.
The idea is that, rather than
deterministically specifying how to pack, one gives a randomised packing
algorithm and argues that it is likely to succeed. Here `succeed' means packing
almost all (not all) of the graphs, and there will be some edges remaining. This
leftover is dealt with by the absorption method: one should begin by cleverly
choosing an `absorbing packing' of the first few graphs which has the property
that whatever the remaining edges from the randomised algorithm turn out to be,
one can modify the absorbing packing in order to incorporate the leftover to a
perfect packing. In the work of Keevash, roughly this template is followed
(though there are some mild conditions on the leftover), and an intricate algebraic
structure is used to obtain the absorbing packing. In the work of Joos et al.,
the \emph{iterative absorption} method (originating in~\cite{KnoxKO}) is used:
here one packs in a way that uses all the edges adjacent to most vertices and
almost no edges among the remaining few vertices, and then iterates this
process, until all the difficulty has been pushed into a set of vertices so tiny
that a relatively simple absorber suffices.

The idea of randomised packing dates back to R\"odl's celebrated \emph{nibble
  method}~\cite{RodlNibble} in which he solved the Erd\H{o}s-Hanani problem, of
showing that if $n$ is large enough then one can pack most of the edges of the
complete $k$-uniform $n$-vertex hypergraph with cliques of size~$r$, solving
Steiner's problem approximately. Note that for this problem there is no
divisibility restriction on $n$. The nibble method was brought to tree packing
by B\"ottcher, Hladk\'y, Piguet and Taraz~\cite{BHPT}, who showed that one can
pack most of the edges of $K_n$ with bounded degree trees, provided the trees
are not too close to spanning. This was the trigger for a sequence of
generalisations: Messuti, R\"odl and Schacht~\cite{MRS} showed that one can
replace trees with graphs from any non-trivial minor-closed family; Ferber, Lee
and Mousset~\cite{FLM} showed that one can additionally allow spanning graphs;
Kim, K\"uhn, Osthus and Tyomkyn~\cite{KKOT} discarded the structural assumption
entirely, packing most of the edges of $K_n$ with arbitrary bounded degree
graphs. All these results work in more generality than just for packings in~$K_n$.
In particular, we should note that the results of~\cite{KKOT} work in the
Szemer\'edi regularity setting, which was necessary for the proof strategy
of~\cite{JKKO}.

All the results mentioned so far deal with bounded degree graphs; the first
result to handle growing degrees is due to Ferber and Samotij~\cite{FerSam}, who
showed that one can pack most of the edges of $K_n$ with trees of maximum degree
$O\big(\tfrac{n}{\log n}\big)$ (for almost-spanning trees) or
$O\big(\tfrac{n}{\log n}\big)^{1/6}$ (for spanning trees). In~\cite{DegPack} it was proved that one can pack most of the edges of $K_n$
with arbitrary $D$-degenerate graphs with maximum degree $O\big(\tfrac{n}{\log
  n}\big)$. Finally, recently Montgomery, Pokrovskiy and Sudakov~\cite{MPS} were able to deal with trees of unbounded degree, at least in the setting of Ringel's conjecture: they proved an approximate version of Ringel's conjecture, proving that $K_{2n-1}$ can be packed with $2n-1$ copies of any tree $T$ with $n-o(n)$ vertices.

One might be tempted to think that, while a randomised
strategy is very good for packing most of the edges, one cannot hope for a
perfect packing: After all, at some point one has to pack the
last few graphs, or at least somehow use the last few edges; at this point the
packing is very constrained and any mistake will cause the packing to fail (and
there cannot be many choices left, so that one cannot hope for strong
concentration bounds), but a randomised algorithm will probably make a mistake
(at least, unless it does a good deal of `looking ahead' which will be hard to
analyse). In this paper, however, we show that this thinking is flawed and
a rather natural, simple randomised algorithm can succeed
in giving a perfect packing. Using this algorithm, we prove the following.

\subsection{Main result}

A graph~$G$ is \emph{$D$-degenerate} if every subgraph of~$G$ has a vertex of
degree at most~$D$.

\begin{definition}[$(\mu,n)$-sequence]
  We say that a sequence $(G_i)_{i\in[m]}$ of graphs is a
  \emph{$D$-degenerate $(\mu,n)$-graph sequence with maximum degree~$\Delta$} if 
  \begin{enumerate}[label=\itmarab{G}]
  \item $G_i$ is $D$-degenerate and $\Delta(G_i)\le\Delta$ for each $i\in[m]$,
  \item $v(G_i)=n$   for each $1\le i\le m-\lfloor\mu n\rfloor$, and
  \item $v(G_i)=n-\lfloor \mu n\rfloor$ and $G_i$ has at least $\lfloor\mu n\rfloor$
    leaves for each~$i$ with $m-\lfloor\mu n\rfloor<i\le m$.
  \end{enumerate}
  We also call the~$G_i$ with $m-\lfloor\mu n\rfloor<i\le m$ the
  \emph{special} graphs of the sequence.
\end{definition}

An $n$-vertex graph $H$ is \emph{$(\alpha,k)$-quasirandom} with density $p$ if
$e(H)=p\binom{n}{2}$ and for every $\ell\in[k]$ and every set
$\{v_1,\dots,v_\ell\}$ of vertices of $H$ we have
\[\big|N_H(v_1,\dots,v_\ell)\big|=(1\pm\alpha)p^\ell n\,.\]

Our main result states that a 
$D$-degenerate $(\mu,n)$-sequence $(G_i)_{i\in[t]}$ of guest graphs with
maximum degree of order at most $\frac{n}{\log n}$ can be perfectly
packed into a sufficiently quasirandom graph $\widehat{H}$.

\begin{theorem}[main result]\label{thm:main}
  For every $D$ and $\mu,\hat p_0>0$ there are~$n_0$ and $\xi,c>0$ such that for
  every $\hat{p}\ge \hat p_0$, every $n\ge n_0$, and every $m$, the following holds
  for every $n$-vertex graph $\widehat{H}$ which is $(\xi,4D+7)$-quasirandom
  with density $\hat{p}$. Every $D$-degenerate $(\mu,n)$-graph sequence
  $(G_i)_{i\in[m]}$ with maximum degree $\frac{cn}{\log n}$ such that
  $\sum_{i\in m}e(G_i)\le e(\widehat{H})$ packs into~$\widehat{H}$.
\end{theorem}

It is easy to see (and proved for completeness in Proposition~\ref{prop:tree})
that if
$T$ is a uniform random labelled $n$-vertex tree, then for each $c>0$, with probability
$1-e^{-O(n)}$ the tree $T$ will have at least $n/100$ leaves and maximum degree
at most $\tfrac{cn}{\log n}$. In particular, we have the following corollary to
Theorem~\ref{thm:main}, proving almost all cases of Ringel's conjecture and the
Gy\'arf\'as Tree Packing Conjecture.

\begin{corollary}\label{cor:unif}
  Let $T$ be a uniform random $n$-vertex tree. With probability $1-e^{-O(n)}$,
  there is a packing of $2n-1$ copies of $T$ into $K_{2n-1}$.
 
  Let $T_2,\dots,T_n$ be chosen independently and uniformly at random such that
  $T_i$ is an $i$-vertex tree for each $2\le i\le n$. With probability at least 
  $1-e^{-O(n)}$, there is a packing of $\{T_2,\dots,T_n\}$ into $K_n$.\hfill\qed
\end{corollary}

We should briefly compare these results to the earlier result of Joos, Kim,
K\"uhn and Osthus~\cite{JKKO}. On the one hand, we cannot handle trees with few
leaves, so our result does not contain theirs. Furthermore, as far as
Corollary~\ref{cor:unif} goes, packing bounded degree trees `almost' covers a
typical uniform random tree, whose maximum degree is likely to be
$\Theta\big(\tfrac{\log n}{\log\log n}\big)$, and most likely the approach
of~\cite{JKKO} could be pushed to allow for a few vertices of logarithmic
degree: this would prove something like Corollary~\ref{cor:unif}, though the failure probability would be polynomial rather than exponential in $n$ (the probability of having a vertex of at least logarithmic degree is only polynomially small). On the other hand, the method of~\cite{JKKO} heavily
relies on the structure of trees, and in particular that one can embed them
effectively in a Szemer\'edi partition; handling general degenerate graphs with
high maximum degree would be rather challenging with their approach.

Finally, we discuss which conditions in Theorem~\ref{thm:main} are needed.
It is easy to see that a typical graph~$H$ on $n$ vertices with density $\frac12$ will be
quasirandom. However such a graph will typically not
contain any $\tfrac1{10}\log n$-set $S$ of vertices such that each other vertex
has a neighbour in~$S$.
In particular, if $G$ is an $n$-vertex graph which is the vertex
disjoint union of $\tfrac1{10}\log n$ stars, each with the same number of leaves
(up to an error $1$), then $G$ has maximum degree less than $\tfrac{20n}{\log
  n}$ but does not embed into~$H$. Thus the maximum degree
bound in our theorem is optimal up to a constant factor.

One can allow $D$ to grow with $n$. Examination of our proof shows it can grow
roughly as $\log\log\log n$, but this is presumably not optimal. On the other
hand, $D$ cannot be as big as $10\log n$, since a typical random graph is
unlikely to contain any given graph with $9n\log n$ edges.

We cannot allow all graphs to be spanning in Theorem~\ref{thm:main}, as an
example in~\cite[Section~9.1]{BHPT} shows. However we expect one can do better than needing
linearly many graphs to be linearly far from spanning.

We do not believe that it is necessary to have many graphs with many leaves. We
should note that one cannot simply omit this condition, because for example no
collection of cycles can perfectly pack $K_{2n}$, due to a parity obstruction:
cycles use an even number of edges at each vertex, but $K_{2n}$ has odd degree
vertices. However for the case $D=1$ (i.e.\ forests) we believe one can omit the
condition entirely (as the leaves should allow for parity correction). Work on
this problem is work in progress with Hladk\'y and Piguet.

\subsection{Proof outline}

This paper builds on~\cite{DegPack}, so we begin by outlining the randomised algorithm
\PackingProcess{} described there (and repeated here later). In
\PackingProcess{}, we embed graphs one-by-one into $\widehat{H}$. To embed a
given graph, we take the vertices in the degeneracy order, and one by one embed
them: at each step, we choose from the set of all vertices to embed to which do
not immediately break our packing (either by re-using a vertex already used in
the current embedding or by re-using an edge already used for a previous graph)
uniformly at random. We do this until almost the entire graph is embedded; then
we choose arbitrarily a way of completing the embedding to a spanning
embedding. (This is a slight simplification, but the simplification does
  not affect the point.) Note that here we certainly do not `look ahead' in any
way at what we will embed in the future, and the algorithm is essentially purely
random. 

To prove our main theorem, we will begin by removing some
of the leaves from each special graph in a given $(\mu,n)$-graph sequence. We
will then use \PackingProcess{} to pack all the non-special graphs
and all the special graphs minus the removed leaves into
$\widehat{H}$. It
remains to embed these removed leaves into the graph $H$ consisting of the
unused edges of $\widehat{H}$. We say a removed leaf is \emph{dangling} at a
vertex $v\in V(H)$ if its parent is embedded to $v$. We will show that at each
vertex of $H$, it is likely that there are about twice as many edges as dangling
leaves. In order to decide where to embed the dangling leaves, we first orient
the edges of $H$ randomly, then `correct' this orientation (by reversing a few
carefully chosen directed paths of length $2$) such that the out-degree of each
vertex is equal to the number of dangling leaves. This is the only step in our algorithm
where we
`look ahead' and prepare for the future.

We then complete the packing by going through $H$ vertex-by-vertex, and for each
vertex choosing a uniform random assignment of the dangling leaves to the
out-neighbours which preserves having a packing. We should note that this last
step has some similarity to the approach of~\cite{JKKO}, where the
authors also complete their perfect packing by assigning dangling leaves to
out-neighbours, but in a small set of vertices. However in their setting, they
only need to assign one dangling leaf per tree, and no other vertices of that
tree are embedded to the small vertex set. As they already did all the hard work
to reach this point, it is not hard for them to make such an assignment. In our
setting, we need to embed linearly many dangling leaves per tree, which dangle
on many different vertices, and the previously embedded images of these trees
can cover most of the vertices to which we want to embed dangling leaves. It is
already non-trivial that we can even assign the dangling leaves at the first
vertex of $H$, and this assignment affects what we can do at later vertices.

In order to understand how it can be that this random process succeeds in
obtaining a perfect packing, one should note that when we embed the dangling
leaves at the first vertex of $H$, we have no choice over the set of edges we
use (these are fixed as the out-neighbours) but the set of assignments, from
which we choose uniformly, is rather large. This property is preserved right
through to the last vertex of $H$---even in the last step, we have not one but
many possible assignments to choose from, so that even in the last steps we have
quite a lot of randomness.

\subsection{Organisation}

In Section~\ref{sec:prelim} we fix notation and collect some concentration
inequalities and facts about degenerate graphs. In Section~\ref{sec:alg} we
state our main technical theorem, show that it implies Theorem~\ref{thm:main},
and formalise our random packing process. In Section~\ref{sec:const} we fix the
constants we will use throughout our proofs. In Section~\ref{sec:lemmas} we
provide our main lemmas which analyse what happens in the different phases of
our packing process: the almost perfect packing lemma, the orientation
lemma, and the matching lemma. In Section~\ref{sec:maintech} we show that these
lemmas imply our main technical theorem.  Section~\ref{sec:orient} proves the
orientation lemma, Section~\ref{sec:match} the matching lemma, and
Section~\ref{sec:appl} the almost perfect packing lemma. The latter takes up the
(technical) bulk of the paper. Concluding remarks are given in Section~\ref{sec:concl}.

\section{Preliminaries}
\label{sec:prelim}

\subsection{Notation}
\label{subsec:not}

For a graph~$G$ we write $V(G)$ for the vertices of~$G$, and~$E(G)$ for its
edges, $v(G)$ for the number of vertices in~$G$, and $e(G)$ for the number
of edges.
For disjoint vertex sets $X,Y\subset V(G)$ we write $G[X]$ for the subgraph
of~$G$ induced by~$X$, and $G[X,Y]$ for the bipartite subgraph of~$G$ on
vertex set $X\cup Y$ and with all edges of~$G$ with one end in~$X$ and the
other in~$Y$. For a set $S$ of vertices of $G$, we write $N_G(S)$ for the \emph{common neighbourhood} $\{u\in V(G):su\in E(G)\text{ for each }s\in S\}$. We write $\deg_G(S):=\big|N_G(S)\big|$ for the \emph{common degree} of $S$ in $G$. When $S=\{v_1,\dots,v_\ell\}$ we will omit the set braces and simply write $N_G(v_1,\dots,v_\ell)$ and $\deg_G(v_1,\dots,v_\ell)$. We will not use joint neighbourhoods of sets of vertices in this paper.

Given a graph $G$ and a set of vertices $X$, if $X\subset V(G)$ we write $G-X$ for the graph obtained by removing the vertices $X$ from $V(G)$, i.e.\ $G\big[V(G)\setminus X\big]$. If $X$ is disjoint from $V(G)$, we write $G+X$ for the graph obtained by adding $X$ as a set of isolated vertices, i.e.\ the graph on vertex set $V(G)\cup X$ whose edge set is $E(G)$. Given graphs $G_1$ and $G_2$ with $V(G_2)\subset V(G_1)$, we write $G_1-G_2$ for the graph obtained by removing the edges of $G_2$ from $G_1$, i.e.\ the graph on vertex set $V(G_1)$ whose edge set is $E(G_1)\setminus E(G_2)$.

Given an ordering $V(G)=\{v_1,\dots,v_n\}$ of the vertices of a graph $G$, we write $\LNBH_G(v_i)$ for the \emph{left-neighbourhood of $v_i$}, i.e.\ the set
\[\LNBH_G(v_i):=N_G(v_i)\cap \{x_k:~ k\in [i-1]\}\,.\]
We write $\LEFTDEG_G(v_i):=\big|\LNBH_G(v_i)\big|$ for the \emph{left-degree of
  $v_i$}, and say the order is a \emph{$D$-degenerate order} if for each
$i\in[n]$ we have $\LEFTDEG_G(v_i)\le D$.

An \emph{orientation} of a graph $H=(V,E)$ is an oriented graph on~$V$
which contains, for each undirected edge $uv\in
E$, exactly one directed edge, either $\vec{uv}$ or $\vec{vu}$.
The \emph{outdegree} $\deg^+_{\vec{H}}(v)$ of a vertex~$v$ in an oriented
graph~$\vec{H}$ is the number of vertices~$u$ in~$\vec{H}$ such that $\vec{vu}$
is an edge of~$\vec{H}$; the set of these vertices~$u$ is the
\emph{outneighbourhood} $N^+_{\vec{H}}(v)$ of~$v$.

\smallskip

Let
  $\Omega$ be a finite probability space. A \emph{filtration} $\mathcal{F}_0$,
  $\mathcal{F}_1$,\dots, $\mathcal{F}_n$ is a sequence of partitions of~$\Omega$
  such that $\mathcal{F}_i$ refines $\mathcal{F}_{i-1}$ for all $i\in[n]$. In
  our application, the partition $\mathcal{F}_i$ is given by all possible
  histories of the run of one of our algorithms up to time~$i$. (For more
  explanation see~\cite{DegPack}.) We say that a function $f:\Omega\rightarrow
  \mathbb{R}$ is \emph{$\mathcal{F}_i$-measurable} if $f$ is constant on each
  part of $\mathcal{F}_i$. Further, for any random variable
  $Y\colon\Omega\rightarrow\mathbb{R}$ the \emph{conditional expectation}
  $\Exp(Y|\mathcal{F}_i)\colon\Omega\rightarrow\mathbb{R}$ and the
  \emph{conditional variance} $\Var(Y|\mathcal{F}_i)\colon\Omega\rightarrow\mathbb{R}$
  of $Y$ with respect to $\mathcal{F}_i$ are defined
  by
  \begin{equation*}
    \begin{split}
      \Exp(Y|\mathcal{F})(x)&=\Exp(Y|X),\\
      \Var(Y|\mathcal{F})(x)&=\Var(Y|X),
    \end{split}
    \qquad
    \text{where $X\in\mathcal{F}$ is such that $x\in X$\,.}
  \end{equation*}

Suppose that we have an algorithm which proceeds in $m$ rounds using a new
source of randomness $\Omega_i$ in each round $i$. Then the probability space
underlying the run of the algorithm is $\prod_{i=1}^{m}\Omega_i$. By
\emph{history up to time $t$} we mean a set of the form
$\{\omega_1\}\times\cdots\times\{\omega_t\}\times
\Omega_{t+1}\times\cdots\Omega_m$, where $\omega_i\in\Omega_i$. We shall use the
symbol $\hist_{t}$ to denote any particular history of such a form. By a
\emph{history ensemble up to time $t$} we mean any union of histories up to time
$t$; we shall use the symbol $\histens$ to denote any one such. Observe that
there are natural filtrations associated to such a probability space: given
times $t_1<t_2<\dots$ we let $\mathcal{F}_{t_i}$ denote the partition of
$\Omega$ into the histories up to time $t_i$.

\subsection{Probabilistic tools}

  \begin{theorem}[Chernoff bounds,~{\cite[Theorem 2.10]{JLR}}]
   \label{thm:chernoff}
  Suppose
  $X$ is a random variable which is the sum of a collection of independent Bernoulli random variables.
  Then we have for $\delta\in(0,3/2)$
 \[\Prob\big[X>(1+\delta)\Exp  X\big]<e^{-\delta^2\Exp X/3}\quad\text{and}\quad
 \Prob\big[X<(1-\delta)\Exp  X\big]<e^{-\delta^2\Exp X/3}
 \,.\]
  \end{theorem}

  We use the following consequence of Freedman's inequality~\cite{Freedman},
  derived in~\cite{DegPack}, for analysing our random embedding algorithms. 

  \begin{lemma}[Freedman's inequality on a good event]
    \label{lem:freedman}
    Let $\Omega$ be a finite probability space, and $(\mathcal{F}_0,
    \mathcal{F}_{1},\dots,\mathcal{F}_{n})$ be a filtration. Suppose that we
    have $R>0$, and for each $1\le i\le n$ we have an $\mathcal{F}_i$-measurable
    non-negative random variable $Y_{i}$,
    nonnegative real numbers~$\tmu$,~$\tnu$ and
    $\tsigma$, and an event $\cE$. Suppose that almost surely, either $\cE$ does
    not occur or we have
    $\sum_{i=1}^{n}\Exp\big[Y_{i}\big|\mathcal{F}_{i-1}\big]=\tmu\pm\tnu$, and
    $\sum_{i=1}^n\Var\big[Y_i\big|\mathcal{F}_{i-1}\big]\le\tsigma^2$, and $0\le
    Y_i\le R$ for each $1\le i\le n$. Then for each $\trho>0$ we have
    \[
      \Prob\left[\mathcal{\cE}\text{ and }\sum_{i=1}^{n}Y_{i}\neq\tmu\pm(\tnu+\trho)\right]\le2\exp\Big(-\frac{\trho^{2}}{2\tsigma^2+2R\trho}\Big)\,.
    \]
    Furthermore, if we assume only that either $\cE$ does not occur or we have
    $\sum_{i=1}^{n}\Exp\big[Y_{i}\big|\mathcal{F}_{i-1}\big]\le\tmu+\tnu$, and
    $\sum_{i=1}^n\Var\big[Y_i\big|\mathcal{F}_{i-1}\big]\le\tsigma^2$, and $0\le
    Y_i\le R$ for each $1\le i\le n$, then for each $\trho>0$ we have
    \[
      \Prob\left[\mathcal{\cE}\text{ and }\sum_{i=1}^{n}Y_{i}>\tmu+\tnu+\trho\right]\le\exp\Big(-\frac{\trho^{2}}{2\tsigma^2+2R\trho}\Big)\,.
    \]
\end{lemma}

A special case is the following corollary.

\begin{corollary}\label{cor:freedm}
  Let $\Omega$ be a finite probability space, and $(\mathcal{F}_0,
  \mathcal{F}_{1},\dots,\mathcal{F}_{n})$ be a filtration. Suppose that we have
  $R>0$, and for each $1\le i\le n$ we have an $\mathcal{F}_i$-measurable
  non-negative random variable $Y_{i}$, nonnegative real numbers $\tmu,\tnu$ and
  an event $\cE$.
  \begin{enumerate}[label=\abc]
  \item\label{cor:freedm:uppertail} Suppose that either $\cE$ does not occur or
    we have $\sum_{i=1}^{n}\Exp\big[Y_{i}\big|\mathcal{F}_{i-1}\big]\le\tmu$,
    and $0\le Y_i\le R$ for each $1\le i\le n$. Then
    \begin{equation*}
      \Prob\left[\mathcal{\cE}\text{ and
        }\sum_{i=1}^{n}Y_{i} >2\tmu\right] \le\exp\Big(-\frac{\tmu}{4R}\Big)\,.
    \end{equation*}
  \item\label{cor:freedm:tails} Suppose that either $\cE$ does not occur or we
    have $\sum_{i=1}^{n}\Exp\big[Y_{i}\big|\mathcal{F}_{i-1}\big]=\tmu\pm\tnu$,
    and $0\le Y_i\le R$ for each $1\le i\le n$. Then for each $\trho>0$ we have
    \begin{equation*}
      \Prob\left[\mathcal{\cE}\text{ and
        }\sum_{i=1}^{n}Y_{i} \neq\tmu\pm(\tnu+\trho)\right] \le2\exp\Big(-\frac{\trho^2}{2R(\tmu+\tnu+\trho)}\Big)\,.
    \end{equation*}
    In particular, if $\tnu=\trho=\tmu\teta>0$ and $\teta\le\frac12$, then
    \begin{equation*}
      \Prob\left[\mathcal{\cE}\text{ and
        }\sum_{i=1}^{n}Y_{i} \neq\tmu(1\pm 2\teta)\right] \le2\exp\Big(-\frac{\tmu\teta^2}{4R}\Big)\,.
    \end{equation*}
  \end{enumerate}
\end{corollary}
\begin{proof}
Both parts follow from Lemma~\ref{lem:freedman} with
$\tsigma^2=R(\tmu+\tnu)$; for
the first part we also set $\tnu=0$ and
$\trho=\tmu$. Observe that
\[\Var\big[Y_i\big|\mathcal{F}_{i-1}\big]\le\Exp\big[Y_i^2\big|\mathcal{F}_{i-1}\big]\le R\cdot \Exp\big[Y_i\big|\mathcal{F}_{i-1}\big]\,,\]
so that
\[\sum_{i=1}^n\Var\big[Y_i\big|\mathcal{F}_{i-1}\big]\le R\sum_{i=1}^n\Exp\big[Y_i\big|\mathcal{F}_{i-1}\big]\le R(\tmu+\tnu)\]
when $\cE$ holds, justifying the choice of $\tsigma^2$.
\end{proof}

We conclude this subsection by proving maximum degree and leaf statistics for random labelled trees.

\begin{proposition}\label{prop:tree} Let $T_n$ be a tree chosen uniformly at random from the set of $n$-vertex labelled trees. Then
\begin{enumerate}[label=(\roman*)]
 \item\label{prop:tree:i} With probability at most $\exp\big(-\tfrac{n}{500}\big)$ the number of leaves in $T_n$ is less than $\tfrac{n}{100}$.
 \item\label{prop:tree:ii} Given $c>0$, if $n$ is sufficiently large then with probability at most $e^{-cn/2}$ there is a vertex in $T_n$ with degree greater than $\tfrac{cn}{\log n}$.
\end{enumerate}
\end{proposition}
\begin{proof}
 To prove both probability bounds, we use the well-known Pr\"ufer code bijection between labelled $n$-vertex trees and sequences of $n-2$ vertex labels.
 
 For~\ref{prop:tree:i}, we note that a vertex is a leaf if and only if its label does not appear in the corresponding Pr\"ufer code, and hence the Pr\"ufer code of a tree with less than $\tfrac{n}{100}$ leaves has at least $\tfrac{49n}{50}$ distinct labels. Consider generating the first $\tfrac{n}{2}$ terms of a Pr\"ufer code. If there are less than $\tfrac{n}{4}$ distinct labels, then the full code has less than $\tfrac{3n}{4}$ distinct labels and hence corresponds to a tree with at least $\tfrac{n}{100}$ leaves. Otherwise, there are at least $\tfrac{n}{4}$ distinct labels among the first $\tfrac{n}{2}$ terms. We now count the number of times that these labels are used in the subsequent $\tfrac{n}{2}-2$ terms. Each term is chosen uniformly at random from the set of all $n$ vertex labels, hence has probability at least $\tfrac14$ of repeating a label used in the first $\tfrac{n}{2}$ terms. Thus the expected number of repeated labels is at least $\tfrac14(\tfrac{n}{2}-2)\ge\tfrac{n}{16}$. By the Chernoff bound, Theorem~\ref{thm:chernoff}, with $\delta=\tfrac12$ the probability that less than $\tfrac{n}{32}$ repeated labels occur is at most $\exp\big(-\tfrac{n}{500}\big)$.
 
 For~\ref{prop:tree:ii}, we note that a vertex has degree equal to one plus the number of its appearances in the Pr\"ufer code. Thus a vertex has degree exceeding $\tfrac{cn}{\log n}$ only if its label appears $\tfrac{cn}{\log n}$ times in the Pr\"ufer code. For a given vertex label and choice of $\tfrac{cn}{\log n}$ terms of the Pr\"ufer code, the probability that each of the chosen terms is equal to the given label is $n^{-\tfrac{cn}{\log n}}=e^{-cn}$. Taking the union bound over the choices of vertex label and terms of the code, the probability that some vertex label appears at least $\tfrac{cn}{\log n}$ times is at most
 \[n\cdot\binom{n-2}{\tfrac{cn}{\log n}}\cdot e^{-cn}\le n\cdot\big(\tfrac{n\log n}{cn}\big)^{\tfrac{cn}{\log n}}e^{-cn}=\exp\big(\log n+\tfrac{cn\log(c^{-1}\log n)}{\log n}-cn\big)\le e^{-cn/2}\,,\]
 where the final inequality is valid for all sufficiently large $n$.
\end{proof}

We should point out that much more precise statistics are known; we give these rough and simple bounds for completeness.

\subsection{Degenerate graphs}

It is easy to show that degenerate graphs contain large independent sets all of whose vertices have the same degree.

\begin{lemma}[Lemma~8 of~\cite{DegPack}]\label{lem:degindept}
  Let $G$ be a $D$-degenerate $n$-vertex graph. Then there exists an integer
  $0\le d\le 2D$ and a set $I\subseteq V(G)$ with $|I|\ge (2D+1)^{-3}n$ which is
  independent, and all of whose vertices have the same degree $d$ in $G$.
\end{lemma}

In~\cite{DegPack} this was used to show that one can modify a degeneracy order
slightly to move such an independent set to the end of the order while not
increasing the degeneracy by much. We repeat the straightforward argument here
for completeness.

\begin{lemma}
  \label{lem:reorder}
  Let $G$ be a $D$-degenerate $n$-vertex graph. Then there
  exists an integer
  $0\le d\le 2D$
  and a $2D$-degenerate order of~$V(G)$ such that the last $\lceil(2D+1)^{-3}n\rceil$
  vertices in this order form an independent set and all have degree~$d$.
\end{lemma}
\begin{proof}
  By Lemma~\ref{lem:degindept} there is an independent set~$I$ in~$G$ of $\lceil
  (2D+1)^{-3}n\rceil$ vertices, each of which has degree~$d$. Now pick a
  $D$-degenerate order of~$G$ and then modify this order by moving all vertices
  of~$I$ to the end (in an arbitrary order). Since all vertices in~$I$ have
  degree $d\le 2D$ the resulting order is $2D$-degenerate.
\end{proof}
  
Further, we shall use the following auxiliary lemma, which given an arbitrary family of
graphs we want to pack produces a family with at most $\frac32 n$ members and
the same bound on maximum degree and degeneracy by combining graphs with many
isolated vertices in the family. Obtaining such a family with at most $\frac32
n$ members needs some argument; while obtaining a family with at most $2n$
members instead is straightforward. In~\cite{DegPack} we only used the latter, and
the reason why we use the smaller family here is that it allows us to
stay consistent with the constants used in~\cite{DegPack}. More precisely, the
constant $\alpha_{2n}$ that will be defined in~\eqref{eq:defconsts} is not small
enough for our analysis here, while $\alpha_{7n/4}$ is small enough.

\begin{lemma}[compression lemma]
  \label{lem:compress}
  Let $(G_i)_{i\in[m]}$ be a family of $D$-degenerate graphs
  with maximum degree at most~$\Delta$, with $\sum_{i=1}^{m}
  e(G_i)\le\binom{n}{2}$ and $v(G_i)\le n$ for all $i\in[m]$. Then there is a
  family of graphs $(\check G_i)_{i\in[\check m]}$ with $\check m\le\frac32 n$
  such that $\sum_{i=1}^{\check m} e(\check G_i)\le\binom{n}{2}$, such that for
  each $i\in[\check m]$ we have $v(\check G_i)\le n$, $\Delta(\check
  G_i)\le\max\{2,\Delta\}$, and~$\check G_i$ is $\max\{2,D\}$-degenerate, and
  such that $(\check G_i)$ is a packing of $(G_i)$.
\end{lemma}
\begin{proof}
  Given the family $(G_i)$, repeatedly perform the following operation,
  packing two members of the family into one graph.
  If the current family contains two graphs $G$, $G'$ which have at most
  $\frac23 n$ vertices of degree at least~$1$ and at most $\frac13 n$
  vertices of degree at least~$2$ then pack~$G$ and~$G'$ into a graph~$G''$
  as follows and then remove $G$, $G'$ from the family and add~$G''$ instead.

  To define an embedding $\phi$ of~$G'$ into $\bar G$, let $A,B,C\subset
  V(G)$ be a partition of $V(G)$ into sets of size either $\lfloor\tfrac{n}{3}\rfloor$ or $\lceil\tfrac{n}{3}\rceil$, such that $|A|=|C|$, and such that 
  $\deg_{G}(x)=0$ for all $x\in A$ and $\deg_{G}(x)\le 1$ for all $x\in B$,
  which is possible by our assumptions on~$G$.
  Analogously, let $A',B',C'\subset
  V(G')$ be  a partition of $V(G')$ into sets of size either $\lfloor\tfrac{n}{3}\rfloor$ or $\lceil\tfrac{n}{3}\rceil$, such that $|A'|=|C'|$ and such that
  $\deg_{G'}(x)=0$ for all $x\in A'$ and $\deg_{G'}(x)\le 1$ for all $x\in B'$.
  We construct $\phi$ by first finding
  a packing of $G[B]$ and $G'[B']$ (which is easy since these two graphs
  are matchings and $|B|\ge3$) and then extending this by arbitrarily mapping~$A'$ to~$C$
  and~$C'$ to~$A$. Note that by construction $|A|=|A'|=|C|=|C'|$. Clearly, this indeed gives a packing of~$G$ and~$G'$
  since vertices in~$A$ have degree~$0$ in~$G$ and vertices in~$A'$ have
  degree~$0$ in~$G'$. Moreover, $\Delta(G'')\le\max\{2,\Delta\}$, and~$G''$ is
  $\max\{2,D\}$-degenerate by construction.

  We stop combining graphs in this way when at most one graph with at most 
  $\frac23 n$ vertices of degree at least~$1$ and at most $\frac13 n$
  vertices of degree at least~$2$ remains; we call the resulting family $(\check
  G_i)_{i\in[\check m]}$.
  In this family, all graphs~$\check G_i$ but possibly one graph satisfy at least
  one of the following conditions:
  \begin{itemize}
    \item $\check G_i$ has more than $\frac23 n$ vertices of degree at least~$1$,
    \item $\check G_i$ has more than $\frac13 n$ vertices of degree at least~$2$.
  \end{itemize}
  In either case $e(\check G_i)\ge \frac13 n$, and therefore we conclude from
  $\sum_{i=1}^{\check m} e(\check G_i)=\sum_{i=1}^{m} e(G_i)\le\binom{n}{2}$ that
  \[\check m\le 1+\frac{\binom{n}{2}}{\frac13 n}\le\frac32 n\,.\]
\end{proof}

\section{Main technical theorem and the packing algorithm}\label{sec:alg}

In this section we detail our packing algorithm, introduce the definitions
necessary for this algorithm, and outline the proof of why this algorithm succeeds.
We deduce Theorem~\ref{thm:main} from the following technical result.

\begin{theorem}\label{thm:maintech}
  For every $D$ and $\mu,\hat p_0>0$ with $\mu\le\frac14$ there are~$n_0$ and $\xi,c>0$ such that for
  every $\hat{p}\ge \hat p_0$ and every $n\ge n_0$ the following holds. Suppose that~$\widehat H$ is a
  $(\xi,2D+3)$-quasirandom graph with $n$ vertices and density $\hat{p}$.
  Suppose that $s^{*}\le \frac74 n$ and that the graph sequence $(G_{s})_{s\in[s^*]}$
  is a $D$-degenerate $(\mu,n)$-graph sequence with maximum degree
  $\Delta\le\tfrac{cn}{\log n}$, such that for each $s\in[s^*]$ there is a
  $D$-degenerate order of $G_s$ such that
  the last $\lceil(D+1)^{-3}n\rceil$ vertices
  form an independent set in $G_{s}$, and all have the same
  degree $d_s$ in $G_s$. Suppose further that
  $\sum_{s\in[s^*]}e(G_s)=e(\widehat{H})$. Then $\left(G_{s}\right)_{s\in
    [s^{*}]}$ packs into~$\widehat H$.
\end{theorem}

Before sketching the proof of Theorem~\ref{thm:maintech}, we show that it
implies Theorem~\ref{thm:main}.
 
\begin{proof}[Proof of Theorem~\ref{thm:main}]
  Given~$D$, $\mu$, $\hat p_0$, let~$n'_0$, $\xi$, $c$ be as given by
  Theorem~\ref{thm:maintech} for input $D'=2\max\{2,D\}$,
  $\mu'=\min\{\mu,\frac14\}$, and $\hat p_0$. Choose $n_0=\max\{n_0',10c^{-2}\}$.
  Next,
  let~$\hat p$ and~$n$ as well as the graphs $\widehat{H}$ and $(G_i)_{i\in[m]}$ be
  given.

  Now we first add new graphs~$G_i$ with $i>m$ consisting of single edges to our
  graph sequence until $\sum e(G_i)=e(\widehat{H})$. Assume that the resulting
  sequence has $m'$ graphs and reorder the sequence so that the $\lfloor \mu n
  \rfloor$ special graphs come last.
  In a second step we apply the compression lemma, Lemma~\ref{lem:compress}, to
  the non-special graphs $(G_i)_{i\in[m'-\lfloor {\mu' n} \rfloor]}$ to obtain a
  family $(\check G_i)_{i\in[\check m]}$ with $\check m\le\frac32$ that is a
  packing of $(G_i)_{i\in[m']}$. In a third
  step, we add
  the remaining special graphs to this compressed family, that is, for $1\le i\le\lfloor
  {\mu' n} \rfloor$ we let $\check G_{\check m+i}=G_{m'-\lfloor {\mu'
      n}\rfloor+i}$. We obtain a family $(\check G_s)_{s\in[s^*]}$ of
  $\max\{2,D\}$-degenerate graphs with maximum degree at most $cn/\log n$, where
  $s^*\le\frac32 n+\lfloor
  {\mu' n} \rfloor\le\frac74 n$. In a fourth step, we apply
  Lemma~\ref{lem:reorder} to obtain a $D'$-degenerate order of each $\check G_s$
  such that the last $\lceil (D+1)^{-3}n \rceil \geq \lceil (D'+1)^{-3}n \rceil$ vertices
  form an independent set in $\check G_{s}$, and all have the same
  degree $d_s$ in $\check G_s$. Hence the family $(\check G_s)_{s\in[s^*]}$ satisfies
  all conditions required by Theorem~\ref{thm:maintech} with the above chosen
  constants. Since $\widehat{H}$ is $(\xi,4D+7)$-quasirandom, it is also $(\xi,2D'+3)$-quasirandom as required for Theorem~\ref{thm:maintech}. Applying this theorem, we obtain a perfect packing of  $(\check G)_{s\in[s^*]}$ into
  $\widehat H$, which gives a packing of $(G_i)_{i\in[m]}$ since $(\check G)_{s\in[s^*]}$
  is a packing of $(G_i)_{i\in[m]}$ (plus possibly some additional edges).
\end{proof}

We now sketch the proof of Theorem~\ref{thm:maintech}. We start by creating an almost perfect packing, which omits linearly many leaves
in the linearly many special graphs~$G_s$, by packing only the following
subgraph sequence omitting $\ell=\lfloor\nu n\rfloor$ leaves.

\begin{definition}[corresponding subgraph sequence]
  For a $D$-degenerate $(\mu,n)$-graph sequence $(G_s)_{s\in[s^*]}$ with
  maximum degree $\Delta$, we say that $(G'_s)_{s\in[s^*]}$ is a
  \emph{corresponding subgraph sequence omitting $\ell$ leaves} if
  \begin{enumerate}[label=\itmarabp{G}]
    \item\label{def:subgraph:1} for each $s\le s^*-\lfloor\mu n\rfloor$ we have $G'_s=G_s$, and
    \item\label{def:subgraph:2} for each $s>s^*-\lfloor\mu n\rfloor$ we have $G'_s=G_s-V_s+I_s$
      for an independent set $V_s$ of leaves in~$G_s$ with $|V_s|=\ell$, and a set
      $I_s$ of new and independent vertices with $|I_s|=\ell$.
  \end{enumerate}
\end{definition}

We remark that the addition of the independent set~$I_s$
in~\ref{def:subgraph:2} is purely for technical reasons: it guarantees that
the special~$G'_s$ have $n-\lfloor\mu n\rfloor$ vertices, which makes the
statement of some of our later lemmas easier (in particular Lemma~\ref{lem:appl}). The restriction that the set $V_s$ is independent simply says that if there is a component of $G_s$ which contains exactly one edge, both endpoints are leaves but only one may be in $V_s$.

The subgraph sequence $(G'_s)$ is packed with the help of the following
\PackingProcess, which uses an algorithm \RandomEmbedding{} that we shall
describe thereafter. This \PackingProcess{} was introduced and analysed
in~\cite{DegPack}, and it requires $n$-vertex graphs with a degeneracy ordering
whose last vertices form an independent set as input. To that end, for each
$s\in[s^*]$ with $s>s^*-\lfloor \mu n \rfloor$, we do the following. We let $I'_s$ be a set of $n-v(G'_s)$ new
isolated vertices and we obtain $G''_s$ by adding $I'_s$ to $G'_s$. Each
non-special graph $G'_s$, with $s\le s^*-\lfloor\mu n\rfloor$, already has~$n$
vertices and so we simply set $G''_s=G'_s$. For the special graphs $G'_s$, with
$s>s^*-\lfloor\mu n\rfloor$, we fix a $D$-degenerate order of $G''_s$ such that
the $\lfloor\mu n\rfloor>\delta n$ isolated vertices in $I'_s$ come last. We
then relabel vertices, so that again $V(G''_s)=[n]$ and the fixed $D$-degenerate
order is the natural order on $[n]$.

\begin{algorithm}[ht]
    \caption{\PackingProcess{}}\label{alg:pack}
    \SetKwInOut{Input}{Input}
    \SetKwInOut{Output}{Output}
    \Input{graphs $G''_1,\dots,G''_{s^*}$, with $G''_s$ on vertex set $[n]$
      such that  the last $\delta n$ vertices of~$G''_s$ form an independent set; a graph $\widehat H$ on $n$ vertices}
    \Output{a packing $(\phi^*_s)_{s\in[s^*]}$ of $(G''_s)_{s\in[s^*]}$ into
      $\widehat H$ and a left-over graph $H$}
    choose $H^*_0$ by picking edges of $\widehat H$ independently with probability $\gamma\binom{n}{2}/e(\widehat{H})$ \;
    let $H_0=\widehat H-H^*_0$ \;
    \For{$s=1$ \KwTo $s^*$}{
      run \RandomEmbedding($G''_s$,$H_{s-1}$) to get
      an embedding $\phi''_s$ of~$G''_s[{\scriptstyle [n-\delta n]}]$ into~$H_{s-1}$\;   
      let $H_{s}$ be the graph obtained from $H_{s-1}$ by removing the
      edges of $\phi''_{s}\big(G''_s[{\scriptstyle [n-\delta n]}]\big)$\;
      choose an arbitrary extension $\phi^*_s$ of $\phi''_s$ embedding all of $G''_s$ and embedding the edges of $G''_s-G''_s[{\scriptstyle [n-\delta n]}]$ into $H^*_{s-1}$ \;
      let $H^*_s$ be the graph obtained from $H^*_{s-1}$ by removing the edges of $\phi^*_s\big(G''_s-G''_s[{\scriptstyle [n-\delta n]}]\big)$ \;
    }    
    \Return $(\phi^*_s)_{s\in[s^*]}$ and $H=H_{s^*}+H^*_{s^*}$
\end{algorithm}

\medskip

For describing \RandomEmbedding{} we need the following definitions. We shall use the symbol $\AlgMap$ to denote embeddings produced by \RandomEmbedding{}. We write $G\AlgMap H$ to indicate that the graph $G$ is to be embedded into $H$. Also, if $t\in V(G)$, $v\in V(H)$ and $A\subset V(H)$ then $t\AlgMap v$ means that $t$ is embedded on $v$, and $t\AlgMap A$ means that $t$ is embedded on a vertex of $A$.

\begin{definition}[partial embedding, candidate set]
  Let~$G$ be a graph with vertex set $[v(G)]$, and~$H$ be a graph with
  $v(H)\ge v(G)$.
  Further, assume $\psi_{j}\colon[j]\rightarrow V(H)$ is a \emph{partial embedding}
  of $G$ into~$H$ for $j\in[v(G)]$, that is, $\psi_j$ is a graph embedding
  of $G\big[[j]\big]$ into~$H$. Finally, let $t\in[v(G)]$ be such that $\LNBH_G(t)\subset[j]$.
  Then the \emph{candidate set of $t$} (with respect to~$\psi_j$) is the
  common neighbourhood in~$H$ of the already embedded neighbours of~$t$,
  that is,
  \[\CANDSET_{G\AlgMap H}^{j}(t)=N_{H}\Big(\psi_{j}\big(\LNBH_{G}(t)\big)\Big)\,.\]
\end{definition}

\RandomEmbedding{} (see Algorithm~\ref{alg:embed}) randomly embeds most of a guest graph~$G$ into a host
graph~$H$. The algorithm is simple: we
iteratively embed the first $(1-\delta)n$ vertices of~$G$ randomly to one
of the vertices of their candidate set which was not used for embedding
another vertex already.

  \begin{algorithm}[ht]
    \caption{\RandomEmbedding{}}\label{alg:embed}
    \SetKwInOut{Input}{Input}
    \SetKwInOut{Output}{Output}
    \Input{graphs~$G$ and~$H$, with $V(G)=[v(G)]$ and $v(H)=n$}
    \Output{an embedding $\psi_{t^*}$ of~$G\big[[n-\delta n]\big]$ into~$H$}
    $\psi_0:=\emptyset$\;
    $t^*:=(1-\delta)n$\;
    \For{$t=1$ \KwTo $t^*$}{
      \lIf{$\CANDSET_{G\AlgMap H}^{t-1}(t)\setminus\im(\psi_{t-1})=\emptyset$}{
        halt with failure}
      choose $v\in\CANDSET_{G\AlgMap
        H}^{t-1}(t)\setminus\im(\psi_{t-1})$ uniformly at random\;
      $\psi_{t}:=\psi_{t-1}\cup\{t\AlgMap v\}$\;
    }
    \Return $\psi_{t^*}$
  \end{algorithm}

If successful, \PackingProcess{} returns the packing
$(\phi^*_s)_{s\in[s^*]}$ of $(G''_s)_{s\in[s^*]}$ and a leftover graph
$H$. For each $s\in[s^*]$, we obtain an embedding $\phi'_s$ of $G'_s$ into
$\widehat{H}$ from the embedding $\phi^*_s$ of $G''_s$ into $\widehat{H}$
by ignoring the vertices of $G''_s$ which are not in $G'_s$. Recall that
all these vertices are isolated vertices. It follows that the
$(\phi'_s)_{s\in[s^*]}$ give a packing of $(G'_s)_{s\in[s^*]}$ into
$\widehat{H}$ which leaves unused exactly the edges of $H$. 
In~\cite{DegPack} it was shown that \PackingProcess{} is indeed a.a.s.\
successful.
We shall use the techniques developed there
to show in Lemma~\ref{lem:appl} that moreover \PackingProcess{} 
returns a packing of $(G'_s)_{s\in[s^*]}$ and a leftover graph with suitable properties for the
following steps.
  
It remains to pack all the leaves
we omitted from $(G_s)_{s\in[s^*]}$. For this we shall proceed vertex by vertex of the remaining
host graph~$H$, and when
considering $r\in V(H)$ we shall randomly embed all leaves \emph{dangling at $r$}, that is, the leaves of all guest
graphs such that the neighbour of the leaf is already embedded to~$r$.
For describing this process in more detail, we will need the following
definitions.
  
  \begin{definition}[weights]
    Let $(G_s)_{s\in[s^*]}$ be a $(\mu,n)$-graph sequence, and $(G'_s)_{s\in[s^*]}$ be a
    corresponding subgraph sequence, $H$ be an $n$-vertex graph, and
    $\phi'_s\colon V(G'_s)\to V(H)$ be an injection for each $s\in[s^*]$. For
    $s^*-\lfloor\mu n\rfloor<s\le s^*$ we define for each $x\in V(G_s)$ the weight 
    \[ w_s(x)=\big|\{y\in N_{G_s}(x)\colon y \text{ is a leaf of~$G_s$ in~$G_s-G'_s$}\}\big|\,,\]
    and for each $v\in V(H)$ the weight
    \[ w_s(v)=w_s\big({\phi'}_s^{-1}(v)\big)\,.\]
    Further, for each $v\in V(H)$ we define
    \[ w(v)=\sum_{s^*-\lfloor\mu n\rfloor<s\le s^*}w_s(v)\,.\]
  \end{definition}
  Note that since each set $V_s$ of omitted leaves is an independent set in $G_s$, the weight of an omitted leaf is $0$. Thus the entire weight of $G_s$ (which is $\ell$, the number of omitted leaves) is on the vertices in $G'_s$ embedded by $\phi'_s$.
  We next choose an orientation $\oH$ of $H$ such that $N^+_{\oH}(r)=w(r)$
  for each $r\in V(H)$. We shall show in Lemma~\ref{lem:orient} that we can
  choose an orientation with this property which is moreover random-like
  (in the sense that it suitably inherits the properties guaranteed by Lemma~\ref{lem:appl}).
  The idea now is to embed the remaining leaves dangling at~$r$ by  using only edges directed away from~$r$.
  We define the following auxiliary graphs, which encode the ways in which
  we can embed the dangling leaves.
  
  \begin{definition}[leaf matching graphs]\label{def:leafmatch}
   Given $r\in V(\oH)$, we define the \emph{leaves at $r$} to be the set
   \[L_r:=\Big\{x\colon \exists s\text{ such that } x\in V(G_s)\setminus
     V(G'_s) \text{ and } x{\phi'}^{-1}_s(r)\in E(G_s)\Big\}\]
   Let the \emph{leaf matching graph} $F_r$ be the bipartite graph with parts $L_r$ and $N^+_{\oH}(r)$, and edges $xu$ with $x\in L_r$ and $u\in N^+_{\oH}(r)$
   whenever $u\not\in\im\phi'_s$ for the~$s$ such that $x\in V(G'_s)$.
  \end{definition}

  Observe that a perfect matching in $F_r$ defines an assignment of the
  leaves at (all preimages of) $r$ to $N^+_{\oH}(r)$ which extends the packing of $(G'_s)_{s\in[s^*]}$. We will see that each $F_r$ is a graph whose parts have size roughly $\tfrac12pn$ and whose density is roughly $\mu$. If we simply chose a perfect matching in each $F_r$ to embed all the leaves $\bigcup_rL_r$, then we would almost have a perfect packing---each edge of $K_n$ would be used exactly once---but it could be the case that multiple leaves of some $G_s$ (not in the same $L_r$) are embedded to a single $u\in V(H)$. To avoid this, we find perfect matchings in each $F_r$ one at a time and update the leaf matching graphs by removing edges which are no longer useable. In order that not too many edges are removed from any one vertex in any $F_r$, we choose perfect matchings uniformly at random. Making this precise, assume $V(\oH)=\{1,\dots,n\}$, and set $F_r^{(0)}:=F_r$ for each $r\in V(\oH)$. We use the following algorithm.
  \begin{algorithm}[ht]
    \caption{\MatchLeaves{}}\label{alg:MatchLeaves}
    \SetKwInOut{Input}{Input}
    \SetKwInOut{Output}{Output}
    \Input{a
      $(\mu,n)$-graph sequence $(G_s)_{s\in[s^*]}$, a corresponding subgraph
    sequence $(G'_s)_{s\in[s^*]}$ omitting $\lfloor \nu n \rfloor$ leaves, and
    associated leaf matching graphs $F_1^{(0)},\dots,F_n^{(0)}$}
    \Output{matchings $(\sigma_r)_{r\in[n]}$ of the omitted leaves to feasible
      image vertices as given by the leaf matching graphs}
    \For{$r=1$ \KwTo $n$}{
     let $\sigma_r$ be a uniform random perfect matching in $F_{r}^{(r-1)}$ \;
     \For{$k=r+1$ \KwTo $n$}{
     let $B_k:=\big\{xu\in E(F_{k}^{(r-1)})\colon
      \exists s\text{ such that } x\in V(G'_s) \text{ and } \sigma_r^{-1}(u)\in V(G_s)\big\}$\;
      let $F_{k}^{(r)}:=F_{k}^{(r-1)}-B_k$ \;
     }
    }
   \Return $(\sigma_r)_{r\in[n]}$ \;
  \end{algorithm}
  
  We shall show that, throughout, the graphs $F_{k}^{(r)}$ satisfy a certain
  degree-codegree condition. We shall show in
  Lemma~\ref{lem:match} that under this degree-codegree condition we can
  find a perfect matching in $F_r^{(r-1)}$. Further, the same lemma asserts that a perfect matching $\sigma_r$ chosen
  uniformly at random in $F_r^{(r-1)}$ uses edges almost uniformly, which is
  important for maintaining the degree-codegree condition.

  We will then, for each $s\in[s^*]$ and each $x\in V(G_s)$ set
  \begin{equation}\label{eq:phi}
    \phi_s(x)=\begin{cases} \phi'_s(x) & \text{ if }x\in\dom(\phi'_s)\\
      \sigma_{r}(x) & \text{ if }x\in L_{r}\,.
   \end{cases}
 \end{equation}
 This is a perfect packing of $(G_s)_{s\in[s^*]}$ into $\widehat{H}$ since
 $(\phi'_s)_{s\in[s^*]}$ is a packing of the subgraph sequence
 $(G'_s)_{s\in[s^*]}$ into $\widehat{H}$ and we chose matchings in the leaf
 matching graphs $F_r^{(r-1)}$ to embed the remaining leaves and updated
 the subsequent leaf matching graphs accordingly.

Summing up, our packing algorithm proceeds as described in Algorithm~\ref{alg:perfect}.

  \begin{algorithm}[ht]
   \caption{\PerfectPacking{}}\label{alg:perfect}
    \SetKwInOut{Input}{Input}
    \SetKwInOut{Output}{Output}
    \SetKw{KwForEach}{for each}
    \Input{graphs $G_1,\dots,G_{s^*}$ that form a $(\mu,n)$-graph sequence
      such that  the last $(D+1)^{-3}n$ vertices of~$G_s$ form an independent set; a graph $\widehat H$ on $n$ vertices}
    \Output{A packing $(\phi_s)_{s\in[s^*]}$ of $(G_s)_{s\in[s^*]}$
      into~$\widehat H$}
    let $(G'_s)_{s\in[s^*]}$ be a subgraph sequence corresponding to
    $(G_s)_{s\in[s^*]}$ omitting $\lfloor \nu n \rfloor$ leaves\;
    \For{$s=s^*-\lfloor \mu n \rfloor+1$ \KwTo $s^*$}{
      let~$I'_s$ be a set of $n-v(G'_s)$ (new) isolated vertices\;
      $G''_s:=G'_s+I'_s$, where we place~$I'_s$ at the end of the degeneracy
      order\;
    }
    \lForEach{$s\in[s^*]$}{ assume that $V(G''_s)=[n]$, with the natural degeneracy
    order}
    run \PackingProcess{} to obtain embeddings $(\phi^*_s)_{s\in[s^*]}$ of
    $(G_s'')_{s\in[s^*]}$ into $\widehat H$ with leftover~$H$\;
    obtain embeddings $(\phi'_s)_{s\in[s^*]}$ of $(G_s')_{s\in[s^*]}$ into
    $\widehat H$ from $(\phi^*_s)_{s\in[s^*]}$ by ignoring the $I'_s$\;
    construct a random-like orientation $\oH$ of~$H$ with
    $N^+_{\oH}(r)=w(r)$ for all $r\in V(\oH)$\;
    \lForEach{$r\in V(\oH)$}{
      let $F_r^{(0)}$ be the leaf matching graph $F_r$}
    run \MatchLeaves{} to obtain embeddings $(\sigma_r)_{r\in[n]}$ of the
    leaves at~$r$\;
    \lFor{$s=1$ \KwTo $s^*$ and \KwForEach $x\in V(G_s)$}{set $\phi_s(x)$
      as in~\eqref{eq:phi}}
    \Return $(\phi_s)_{s\in[s^*]}$\;
  \end{algorithm}

\subsection{Graphs and maps used in the algorithm}

As described above, a number of different (auxiliary) graphs and maps are used in
our packing procedure. For the convenience of the reader we collect these in the
following table.

\begin{description}[leftmargin=1cm]
\item[$G_s$] are the given $n$-vertex guest graphs, forming a $D$-degenerate
  $(\mu,n)$-graph sequence, whose last $\lfloor(D+1)^{-3}n\rfloor$ vertices form
  an independent set.
\item[$G'_s$] is in the subgraph sequence corresponding to $(G_s)_{s\in[s^*]}$
  omitting $\lfloor \nu n \rfloor$ leaves; the special $G'_s$ have $n-\lfloor
  \mu n \rfloor$ vertices, the others~$n$.
\item[$G''_s$] is obtained from $G'_s$ by adding isolated vertices to the end of
  the degeneracy order until we have~$n$ vertices.
\item[$\widehat H$] is the given $n$-vertex $(\xi,2D+3)$-quasirandom host graph.
\item[$H_{s-1}$] is the part of $\widehat H$ used by \RandomEmbedding{} to embed $G''_s\big[[n-\delta
  n]\big]$.
\item[$H_{s-1}^*$] is the part of $\widehat H$ used in \PackingProcess{} to
  complete the embedding of $G''_s$.
\item[$H$] is the leftover host graph after running \PackingProcess{}.
\item[$\oH$] is a random-like orientation of~$H$ with as many outgoing edges for
  each vertex~$r$ as there are leaves dangling at~$r$.
\item[$F_k^{(r)}$] is what remains of the leaf matching graph~$F_k$ after round~$r$ of \MatchLeaves{}. 
\item[$\phi''_s$] embeds $G''_s\big[[n-\delta
  n]\big]$ into $H_{s-1}$ and is constructed by \RandomEmbedding{}.
  
\item[$\phi_s^*$] is an extension of $\phi''_s$, embedding $G''_s$ into
  $H_{s-1}\cup H^*_{s-1}$ constructed in \PackingProcess{}.
\item[$\phi'_s$] is an embedding of $G'_s$ into $\widehat H$ obtained from
  $\phi_s^*$ by ignoring the added isolated vertices.
\item[$\phi_s$] is an embedding of $G_s$ into $\widehat H$ obtained from
  $\phi'_s$ and the $\sigma_r$ in \PerfectPacking{}.
\item[$\psi_t$] is the partial embedding obtained in round~$t$ of \RandomEmbedding{}.
\item[$\sigma_r$] is the matching in the leaf matching graph $F_r^{(r-1)}$ found
  by \MatchLeaves{}.
\end{description}
  
\section{Constants}
\label{sec:const}

In this section we set values for the various constants we need throughout our proofs (including
those used in the algorithms above), which are the following.

\begin{description}
\item[$\alpha_x$] is the error in the quasirandomness of~$H_x$.
\item[$\alpha$] is a quasirandomness error used in auxiliary lemmas; we
  always assume $\alpha_0\le\alpha\le\alpha_{2n}$.
\item[$\beta_t$] is the error in the diet-condition (see
  Definition~\ref{def:dietcover}) for round~$t$ of \RandomEmbedding{}.
\item[$c$] is the constant in the maximum degree bound of the~$G_s$.
\item[$C$] appears in the error term for the probability of embedding a fixed
  vertex of~$G''_s$ on a fixed vertex of~$H_{s-1}$.
\item[$C'$] appears in the error term for the fraction of vertices of certain
  sets that get covered by embedding one graph~$G''_s$
\item[$D$] is the degeneracy bound of the guest graphs~$G_s$.
\item[$n_0$] is the lower bound on the number~$n$ of vertices.
\item[$\hat{p}$] is the density of the host graph $\widehat{H}$.  
\item[$\hat{p}_0$] is the lower bound on~$\hat{p}$.
\item[$p$] is the density of the leftover host graph~$H$ after running 
  \PackingProcess{} to embed the subgraph sequence.
\item[$p_s$] is the density of the graph~$H_s$.
\item[$\delta$] is the proportion of vertices in $G''_s$ formed by the
  independent set at the end of the degeneracy order as required by \PackingProcess{}.
\item[$\eps$] gives the length~$\eps n$ of intervals in $V(G_s)$ used in
  the cover condition (see Definition~\ref{def:dietcover}); it also appears in
  the error term of the cover condition.
\item[$\eta$] is the error in the quasirandomness of $H_0^*$.
\item[$\gamma$] is the proportion of host graph edges used by
  \PackingProcess{} to complete almost spanning embeddings to spanning embeddings.
\item[$\gamma'$] determines the error bound in our analysis here of \PackingProcess{}.
\item[$\mu$] specifies the fraction of special $G_s$, how far they are from spanning and how many leaves they have.
\item[$\nu$] specifies the fraction of leaves omitted in the subgraph sequence.
\item[$\xi$] is the error in the quasirandomness of $\widehat{H}$.
\end{description}

The constants $D$, $\hat p$ and $\mu$ are provided as input to our main
technical theorem; the other constants are chosen to satisfy
\begin{equation*}
  0\ll\frac{1}{n_0},c\ll\eps\le\xi\ll\alpha_0\le\alpha_{2n}\le\frac{1}{C'}\ll\frac{1}{C}\ll\delta\ll\eta
  \ll\gamma\ll \gamma'\ll \nu\ll \mu,\hat p_0,\frac{1}{D} \,.
\end{equation*}
Here $a\ll b$ means that we choose~$a$ sufficiently small in terms of~$b$, so
that our calculations work. In other words, there is a monotone increasing function
$f\colon\mathbb{R}_{>0}\to\mathbb{R}_{>0}$ with $f(b)\le b$ such that we choose $a=f(b)$.

For the constants $\nu$, $\gamma'$, $\gamma$, and $n_0$ we do not provide
explicit dependencies (mainly because Lemma~\ref{lem:KKOT}, which we take from
elsewhere, does not provide explicit dependencies), but merely state that we can choose them suitably with
the above relations.

The various host graph densities satisfy the following relations. We have $\hat
p\ge\hat p_0$.
Given~$n\ge n_0$, the density~$p$ is determined by 
\begin{equation}\label{eq:p}
  p=\lfloor\mu n\rfloor \lfloor\nu n\rfloor\binom{n}{2}^{-1}\,.
\end{equation}
Moreover, in our later proofs we will assume that
$e(H_0^*)=(1\pm \frac{1}{10})\gamma \binom{n}{2}$,
which can be seen to hold with probability 
larger than $1-e^{-n}$ by an application of Theorem~\ref{thm:chernoff}.
Then, since the density of~$\widehat H$ is~$\hat p$, the density of $H_0$ is
$p_0\ge\hat p-1.1\gamma$, and therefore, because \PackingProcess{} embeds
$\sum_{s\in[s^*]}e(G_s)-\lfloor \mu n \rfloor\lfloor \nu n \rfloor\le\hat p\binom{n}{2}-\lfloor \mu n \rfloor\lfloor \nu n \rfloor$
edges we have
\begin{equation}\label{eq:ps}
  p_s\ge\hat p-1.1\gamma-\frac{\hat p \binom{n}{2}-\lfloor \mu n \rfloor\lfloor  \nu n \rfloor}{\binom{n}{2}}
  \ge \nu\mu-1.1\gamma\ge\gamma\ \qquad\text{for all $s\in[s^*]$}\,.
\end{equation}

The remaining constants are only used in the proof of the most technically involved of our main lemmas,
Lemma~\ref{lem:appl}.
These are defined precisely in the same way as
in~\cite{DegPack} apart from~$C'$, which is added here.\footnote{The density of the
  host graph $\widehat H$ is denoted by~$p$ in~\cite{DegPack}; we
  denote it by~$\hat p$ here.}
This is important, because much of our proof of Lemma~\ref{lem:appl} builds on
tools developed in~\cite{DegPack}, and the relation of the constants involved is
somewhat more intricate.

\begin{setting}\label{set:const}
  Let $D,n\in\mathbb N$ and $\hat p,\gamma>0$ be given. We define
  \begin{equation}\label{eq:defconsts}\begin{split}
      \eta&=\frac{\gamma^D}{200D}\,,\quad
      \delta=\frac{\gamma^{10D}\eta}{10^{6}D^{4}}\,,\quad
      C=40D\exp\big(1000D\delta^{-2}\gamma^{-2D-10}\big)\,,\quad
      C'=10^4C\delta^{-1}\,, \\
      \alpha_x&=\frac{\delta}{10^8CD}\exp\Big(\frac{10^8CD^3\delta^{-1}(x-2n)}{n}\Big) \qquad\text{for each $x\in\mathbb{R}$},\\
      \eps&= \alpha_0\delta^2\gamma^{10D}/1000CD\,,\quad c= D^{-4}\eps^{4}/100\,\quad\text{and}\quad \xi=\alpha_0/100\,.
  \end{split}\end{equation}
  Moreover, given $\alpha>0$ we use the following constants
  $\beta_{t}(\alpha)$, which are
  chosen such that $\beta_{0}(\alpha)=\alpha$ and such that
  $\beta_{v(G)}(\alpha)/\beta_{0}(\alpha)$ is bounded by a constant which does not depend on
  $\alpha$ (though it does depend on $D$, $\gamma$ and $\delta$). We define
  \begin{equation}\label{eq:defnbeta}
    \beta_{t}(\alpha) =2\alpha\exp\left(\tfrac{1000D\delta^{-2}\gamma^{-2D-10}t}{n}\right)\;.
  \end{equation}
\end{setting}

\begin{remark}
  When using the constants $\alpha_x$, $\beta_t$, we will mainly take~$x$ and~$t$ integer in
  the range $[0,2n]$, but it is convenient to allow them to be any real
  number.

  Note that we call~$\alpha_x$ and~$\beta_t$ `constant' even
  though $n$ appears in their definition. 
  It is easy to check though that
  $\alpha_x$ is strictly increasing in $x$ and $\beta_t$ is strictly
  increasing in~$t$ and
  that neither $\alpha_0,\beta_0$ nor $\alpha_{2n},\beta_{2n}$ depends on
  $n$. Further, for each $t\ge 0$, we have
  \begin{equation}\label{eq:betabound}\begin{split}
      &\tfrac{1}{n}\int_{i=0}^t 1000D\delta^{-2}\gamma^{-2D-10}\beta_i \,\mathrm{d}i\\
      \le &2\alpha\int_{i=-\infty}^t
      \frac{1000D\delta^{-2}\gamma^{-2D-10}}{n}\exp\left(\tfrac{1000D\delta^{-2}\gamma^{-2D-10}i}{n}\right)
      \,\mathrm{d}i
      =\beta_t\,.
    \end{split}\end{equation}
\end{remark}

\section{Main lemmas}
\label{sec:lemmas}

In this section we collect the main lemmas we need for the proof of our main technical
theorem.
Our first lemma states that the randomised algorithm
\PackingProcess{} generates an almost perfect packing of the
corresponding subgraph sequence of our guest graphs such that this packing
and the leftover~$H$ of the host graph satisfy certain properties. We prove this
lemma in Section~\ref{sec:appl}. The fact that \PackingProcess{} produces a
packing of this type such that the leftover~$H$ is quasirandom is the main
result of~\cite{DegPack}.
Here, we need to establish additional properties for completing this to a
perfect packing.

\begin{lemma}[almost perfect packing lemma]\label{lem:appl}
  Assume $0\ll c\ll\xi\ll\delta\ll\gamma\ll\gamma'\ll\nu\ll\mu,\hat p,\frac1D$.
  Let $\widehat{H}$ be a $(\xi,2D+3)$-quasirandom graph with $n$ vertices and density $\hat{p}$.
  Let $s^*\le\frac74n$, let $(G_s)_{s\in[s^*]}$ be a $D$-degenerate $(\mu,n)$-graph sequence with
  maximum degree $\frac{cn}{\log n}$ and $\sum_{s\in[s^*]}e(G_s)=e(\widehat{H})$, and let $(G'_s)_{s\in[s^*]}$ be a
  corresponding subgraph sequence omitting $\lfloor \nu n \rfloor$ leaves.
  Then \PackingProcess\ (applied with constants~$\gamma$ and~$\delta$ to the graphs $(G''_s)_{s\in[s^*]}$
  obtained in \PerfectPacking{} from $(G'_s)_{s\in[s^*]}$ by adding isolated
  vertices)
  a.a.s.\ provides a packing $(\phi'_s)_{s\in[s^*]}$ of
  $(G'_s)_{s\in[s^*]}$ into~$\widehat{H}$ with leftover~$H$ such that for
  $p=\lfloor\mu n\rfloor \lfloor\nu n\rfloor\binom{n}{2}^{-1}$ we have
  \begin{enumerate}[label=\itmarab{P}]
  \item\label{appl:quasi} $H$ is $({\gamma'}^3,2D+3)$-quasirandom and has density~$p$,
  \end{enumerate}
  for all $v\in V(H)$ and $s^*-\lfloor\mu n\rfloor<s,s'\le s^*$ we have
  \begin{enumerate}[label=\itmarab{P},start=2]
    \item\label{appl:weight} $w(v)=(1\pm{\gamma'}^3)\frac{pn}{2}$,
    \item\label{appl:3} $\big|N_H(v)\setminus \im\phi'_s\big|=(1\pm{\gamma'}^3)\mu pn$,
    \item\label{appl:4} $\big|N_H(v)\setminus
      (\im\phi'_s\cup\im\phi'_{s'})\big|=(1\pm{\gamma'}^3)\mu^2pn$ if $s\neq s'$,
  \end{enumerate}
  for all $u,v\in V(H)$ with $u\neq v$ we have,
  \begin{enumerate}[label=\itmarab{P},start=5] 
    \item\label{appl:degN} $\sum_s w_s(v)\ONE_{u\not\in\im\phi'_s}=(1\pm{\gamma'}^3)\mu\frac{pn}{2}$,
  \end{enumerate}
  and for all $u\in V(H)$ and $s^*-\lfloor\mu n\rfloor<s\le s^*$ we have
  \begin{enumerate}[label=\itmarab{P},start=6]
    \item\label{appl:sumweight} If $u\not\in\im\phi'_s$ then $\sum_{v\colon vu\in E(H)} w_s(v)<\frac{10p^2 n}{\mu}$.
  \end{enumerate}
\end{lemma}

Our second lemma states that there is an orientation of~$H$ suitable for
completing the perfect packing by embedding the leaves with the help of the
algorithm \MatchLeaves.
A \emph{random orientation} of a graph $H=(V,E)$ is an orientation of~$H$ in which the
orientation of each edge $\{ u,v\}\in
E$ is chosen independently and uniformly at random. We prove this lemma in Section~\ref{sec:orient}.

\begin{lemma}[orientation lemma]\label{lem:orient}
  Let~$H$ be a $({\gamma'}^3,2)$-quasirandom graph of density~$p$ with vertex weights $w\colon
  V(H)\to \mathbb{N}_0$ such that $w(v)=(1\pm{\gamma'}^3)\frac{pn}{2}$ for all
  $v\in V(H)$ and such that $\sum_{v\in V} w(v)=e(H)$. If $\vec{H}_0$ is a random orientation of~$H$, then a.a.s.\
    there is an orientation~$\vec{H}$ of~$H$ such that for all $v\in V(H)$
  \begin{enumerate}[label=\itmarab{O}]
    \item\label{orient:deg} $\deg^+_{\vec{H}}(v)=w(v)$, and
    \item\label{orient:diff} $\big|\{uv\in E(H)\colon uv \text{ is oriented differently in }
      \vec{H} \text{ and } \vec{H}_0\}\big|\le{\gamma'}^2 n$.
  \end{enumerate}
\end{lemma}

Our last lemma states that if in a graph~$F$
satisfying a certain degree-codegree condition, we remove a few edges and
then choose a perfect matching uniformly at random, then each edge is
roughly equally likely to appear in the matching. In the proof of our main
theorem, we shall show that the leaf matching graphs $F_v$ satisfy these
conditions, and hence \MatchLeaves{} can find a perfect matching in~$F_v$,
using edges almost uniformly.

\begin{lemma}[matching lemma]
\label{lem:match}
  Assume $0\ll\frac{1}{m}\ll p\ll \mu\ll 1$. Let $F=F[U,W]$ be a bipartite
  graph with $|U|=|W|=(1\pm p) m$ such that
  \begin{enumerate}[label=\itmarab{M}]
    \item\label{lem:match:1} $\deg_{F}(x)=(1\pm p)\mu m$ for all $x\in U\cup W$, and
    \item\label{lem:match:2} $\deg_{F}(u,u')=(1\pm p)\mu^2 m$ for all but at most
      $\frac{m^2}{\log m}$ pairs $\{u,u'\}\in\binom{U}{2}$,
  \end{enumerate}
  and let $F'=F'[U,W]$ be a spanning subgraph of $F[U,W]$ such that
  \begin{enumerate}[label=\itmarab{M},start=3]
    \item\label{lem:match:3} $\deg_{F}(x)-\deg_{F'}(x)<\frac{100pm}{\mu^2}$ for all $x\in U\cup W$.
  \end{enumerate}
  Then $F'$ has a perfect matching and for a perfect matching $\sigma$ chosen uniformly at random
  among all perfect matchings in $F'$ and for all $uw\in E(F')$ we have
  \[\Prob[\sigma(u)=w]\le\frac{2}{\mu m}\,.\]
\end{lemma}

This lemma is a straightforward consequence of a lemma (Lemma~\ref{lem:KKOT}) on random matchings
in super-regular pairs by Felix Joos (see~\cite{KKOT}) and the degree-codegree
characterisation of super-regular pairs (Lemma~\ref{lem:DLR}) provided by 
Duke, Lefmann, and R\"odl in~\cite{DukLefRod}. For
completeness, we provide the deduction in Section~\ref{sec:match}.

\section{Proof of the main technical theorem}
\label{sec:maintech}

To prove Theorem~\ref{thm:maintech} we shall run the algorithm
\PerfectPacking{} (Algorithm~\ref{alg:perfect}), which uses \PackingProcess{} to pack the $G'_s$. 
The resulting graph $H$ of unused edges
is likely to satisfy the conclusions of the
almost perfect packing lemma (Lemma~\ref{lem:appl}).
\PerfectPacking{} then chooses a random orientation $\oH_0$ of $H$ and modifies this orientation slightly to obtain $\oH$,
which satisfies the conclusions of the orientation lemma (Lemma~\ref{lem:orient}) and also
oriented versions of properties~\ref{appl:3} and~\ref{appl:4} of Lemma~\ref{lem:appl}. Finally, \PerfectPacking{} runs
\MatchLeaves{} to complete the packing. To show that \MatchLeaves{} succeeds,
we will verify that with high probability for each $r$ the graphs $F_{r}$ and
$F_{r}^{(r-1)}$ satisfy the conditions of the matching lemma (Lemma~\ref{lem:match}). For this we use
Corollary~\ref{cor:freedm} and the union bound.

\begin{proof}[Proof of Theorem~\ref{thm:maintech}]
 We use constants with relations as given in Section~\ref{sec:const}, that
 is
 \[0\ll c\ll\xi\ll\delta\ll\gamma\ll\gamma'\ll\nu\ll\mu,\hat p_0,\frac1D\,,\]
 and $\hat p\ge\hat p_0$.
 Suppose that~$\widehat H$ is an $(\xi,2D+3)$-quasirandom graph with $n$
 vertices and density~$\hat{p}$. Suppose that $s^{*}\le \frac{7}{4}n$ and that the
 graph sequence $(G_{s})_{s\in[s^*]}$ is a $D$-degenerate $(\mu,n)$-graph
 sequence, with maximum degree $\Delta\le\tfrac{cn}{\log n}$, such that the
 last $\lceil(D+1)^{-3}n\rceil$ vertices in the degeneracy order form an independent
 set in $G_{s}$, and all have the same degree $d_s$ in $G_s$. Suppose
 further that $\sum_{s\in[s^*]}e(G_s)=e(\widehat{H})$. We use \PerfectPacking{}
 (Algorithm~\ref{alg:perfect}) for packing $(G_{s})_{s\in[s^*]}$
 into~$\widehat H$ and argue in the following that it succeeds a.a.s.
 
 As $\lfloor\nu n\rfloor<\mu n$, \PerfectPacking{} can choose a corresponding subgraph
 sequence $(G'_s)_{s\in[s^*]}$ omitting $\lfloor\nu n\rfloor$ leaves. Next it
 creates for each $s\in [s^*]$ a graph $G''_s$. For the non-special graphs
 ($s\le s^*-\lfloor\mu n\rfloor$) it sets $G''_s:=G'_s$. For the special graphs
 ($s>s^*-\lfloor\mu n\rfloor$) it obtain~$G''_s$
 by adding the set $I'_s$ of
 $n-v(G'_s)$ isolated vertices, which we place at the end of the
 $D$-degeneracy order. Note that for each~$G''_s$ the last $\delta n$ vertices of $G''_s$ in the degeneracy order
 are an independent set all of whose vertices have degree $d_s$. Indeed,
 if $s\le s^*-\lfloor\mu n\rfloor$ then this holds by
 assumption on $G_s$ and because $\delta<(D+1)^{-3}$, and if $s>s^*-\lfloor\mu
 n\rfloor$ then this holds because $n-v(G'_s)=n-\lfloor \mu n \rfloor$ and
 $\delta<\mu$ (and in this case $d_s=0$).

 \PerfectPacking{} next runs \PackingProcess{} with input $(G''_s)_{s\in[s^*]}$ and $\widehat{H}$.
 By the almost perfect packing lemma (Lemma~\ref{lem:appl}), \PackingProcess{} a.a.s.\ returns a packing
 $(\phi^*_s)_{s\in[s^*]}$ of $(G''_s)_{s\in[s^*]}$ into $\widehat{H}$, and a
 graph $H$ consisting of all the edges not used in the packing, which satisfies
 the conclusions \ref{appl:quasi}--\ref{appl:sumweight} of Lemma~\ref{lem:appl}.
 As described in \PerfectPacking{}, we let for each $s\in[s^*]$ the map
 $\phi'_s$ be the embedding of $G'_s$ into $\widehat{H}$ induced by $\phi^*_s$.
 By construction of the $(G''_s)_{s\in[s^*]}$, the $(\phi'_s)_{s\in[s^*]}$ form
 a packing of the $(G'_s)_{s\in[s^*]}$ into $\widehat{H}$, with $H$ being the
 graph formed by the unused edges. The total number of unused edges is by
 construction $\lfloor\mu n\rfloor\lfloor\nu n\rfloor=p\binom{n}{2}$, so $H$ has
 density $p$.
 
 \PerfectPacking{} next chooses a random-like orientation of~$H$. More precisely,
 we want to use an orientation $\oH$ of $H$ such that
 $w(v)=\deg^+_{\oH}(v)$ for each $v\in V(H)$, which in addition inherits
 oriented versions of~\ref{appl:3} and~\ref{appl:4}. The next claim states that
 such an orientation exists.

 \begin{claim} For all sufficiently large $n$ there exists an orientation $\oH$
   of $H$ such that $w(v)=\deg^+_{\oH}(v)$ for each $v\in V(H)$, and in addition
   for each $s^*-\lfloor\mu n\rfloor<s,s'\le s^*$ we have
  \begin{enumerate}[label=\itmarab{P'},start=3]
   \item\label{appl:3p} $\big|N^+_{\oH}(v)\setminus \im\phi'_s\big|=(1\pm\gamma')\tfrac{\mu pn}{2}$, and
   \item\label{appl:4p} $\big|N^+_{\oH}(v)\setminus
      (\im\phi'_s\cup\im\phi'_{s'})\big|=(1\pm\gamma')\tfrac{\mu^2pn}{2}$ if $s\neq s'$.
  \end{enumerate}
 \end{claim}
 \begin{claimproof}
  By~\ref{appl:quasi}, in particular $H$ is $({\gamma'}^3,2)$-quasirandom and of
  density $p$, and by~\ref{appl:weight} we have
  $w(v)=(1\pm{\gamma'}^3)\tfrac{pn}{2}$ for all $v\in V(H)$. This verifies that
  $H$ satisfies the conditions of the orientation lemma (Lemma~\ref{lem:orient}).
  
  Let $\oH_0$ be a random orientation of $H$. Given $v\in V(H)$ and $s^*-\lfloor\mu n\rfloor<s\le s^*$, by~\ref{appl:3} and Theorem~\ref{thm:chernoff}, with probability at least $1-\exp\big(-\tfrac{{\gamma'}^6\mu pn}{12}\big)$ we have
  \[\big|N^+_{\oH_0}(v)\setminus \im\phi'_s\big|=(1\pm3{\gamma'}^3)\tfrac{\mu pn}{2}\,.\]
  Similarly, given $v\in V(H)$ and $s^*-\lfloor\mu n\rfloor<s<s'\le s^*$, by~\ref{appl:4} and Theorem~\ref{thm:chernoff}, with probability at least $1-\exp\big(-\tfrac{{\gamma'}^6\mu^2 pn}{12}\big)$ we have
  \[\big|N^+_{\oH_0}(v)\setminus (\im\phi'_s\cup\im\phi'_{s'})\big|=(1\pm3{\gamma'}^3)\tfrac{\mu^2pn}{2}\,.\]
  Taking the union bound, and by Lemma~\ref{lem:orient}, with probability at least $1-2n^3\exp\big(-\tfrac{{\gamma'}^6\mu^2 pn}{12}\big)-o(1)$ each of the above good events holds for each $v\in V(H)$ and each $s^*-\lfloor\mu n\rfloor<s,s'\le s^*$, and in addition there is an orientation $\oH$ of $H$ satisfying conclusions~\ref{orient:deg} and~\ref{orient:diff} of Lemma~\ref{lem:orient}.
  
  For sufficiently large $n$ we have $1-2n^3\exp\big(-\tfrac{{\gamma'}^6\mu^2 pn}{12}\big)-o(1)>0$, so we fix $\oH_0$ and $\oH$ satisfying all these properties. By~\ref{orient:deg} the orientation $\oH$ satisfies $\deg_{\oH}^+(v)=w(v)$ for each $v\in V(H)$, as desired. Given $v\in V(H)$ and $s^*-\lfloor\mu n\rfloor<s\le s^*$, by~\ref{orient:diff} we have
  \[\big|N^+_{\oH}(v)\setminus \im\phi'_s\big|=\big|N^+_{\oH_0}(v)\setminus \im\phi'_s\big|\pm{\gamma'}^2 n=(1\pm3{\gamma'}^3)\tfrac{\mu pn}{2}\pm{\gamma'}^2 n=(1\pm\gamma')\tfrac{\mu pn}{2}\,,\]
  where the final inequality is by choice of $\gamma'$. This verifies~\ref{appl:3p}. Similarly, given $v\in V(H)$ and $s^*-\lfloor\mu n\rfloor<s<s'\le s^*$, we have
  \begin{equation*}\begin{split}
      \big|N^+_{\oH}(v)\setminus
      (\im\phi'_s\cup\im\phi'_{s'})\big| &=\big|N^+_{\oH_0}(v)\setminus
      (\im\phi'_s\cup\im\phi'_{s'})\big|\pm{\gamma'}^2
      n=(1\pm3{\gamma'}^3)\tfrac{\mu^2pn}{2}\pm{\gamma'}^2
      n \\
      & =(1\pm\gamma')\tfrac{\mu^2pn}{2}\,,
  \end{split}\end{equation*}
  giving~\ref{appl:4p}.
 \end{claimproof}

 This orientation is now used to embed the remaining dangling leaves. \PerfectPacking{} runs
 \MatchLeaves{} (Algorithm~\ref{alg:MatchLeaves}) for this purpose. Recall that,
 for a vertex $v\in V(\oH)$, the leaf matching graph $F_v$ (see
 Definition~\ref{def:leafmatch}) is a bipartite graph with parts consisting of
 the leaves $L_v$ which we need to embed at $v$ (which will be in many different
 $G_s$) and the out-neighbours $N_{\oH}^+(v)$ of $v$ in $\oH$ to which we will
 embed these leaves, with an edge from a leaf in some $G_s$ to an out-neighbour
 $u$ of $v$ if $u\not\in\im\phi'_s$. Recall that for convenience we assume
 $V(\oH)=[n]$. \MatchLeaves{} starts with $F_v^{(0)}:=F_v$ for each $v\in[n]$,
 and then for each $r\in[n]$ in succession takes a random perfect matching
 $\sigma_r$ in $F_r^{(r-1)}$ and for each $k>r$ removes some edges from
 $F_k^{(r-1)}$ to form $F_k^{(r)}$. As explained in Section~\ref{sec:alg}, it is
 enough to show that with positive probability \MatchLeaves{}
 does not halt with failure. To analyse the running of
 \MatchLeaves{}, we aim to show that for each $r$ the graphs
 $F_r^{(0)}$ and $F_r^{(r-1)}$ satisfy the conditions of Lemma~\ref{lem:match}
 with $m:=\tfrac{pn}{2}$, with $F=F_r^{(0)}$ and $F'=F_r^{(r-1)}$, and with $U=L_r$ and $W=N^+_{\oH}(r)$. We shall then
 use Lemma~\ref{lem:match} to conclude that the
 matching $\sigma_r$ we choose in $F_r^{(r-1)}$ does not use any given edge with
 exceptionally high probability, which in turn will allow us to show that
 \MatchLeaves{} is successful.
 
 \underline{Property~\ref{lem:match:1}:} Given $x\in V(F_r^{(0)})$, we separate
 two cases. If $x\in L_r$ is in the graph $G_s$, then by~\ref{appl:3p} we have
 $\deg_{F_r^{(0)}}(x)=\big|N_{\oH}^+(r)\setminus\im\phi'_s\big|=(1\pm\gamma')\tfrac{\mu
   p n}{2}$. If $x\in N^+_{\oH}(r)$, then by~\ref{appl:degN} we have
 $\deg_{F_r^{(0)}}(x)=\sum_s
 w_s(r)\ONE_{x\not\in\im\phi'_s}=(1\pm{\gamma'}^3)\tfrac{\mu p n}{2}$. In either
 case, since $p>\gamma'$ this verifies~\ref{lem:match:1} for $F=F_r^{(0)}$,
 $F'=F_r^{(r-1)}$ and every $r\in [n]$.
 
 \underline{Property~\ref{lem:match:2}:} Given $u,u'\in L_r$, if $u\in V(G_s)$
 and $u'\in V(G_{s'})$, where $s\neq s'$, then by~\ref{appl:4p} we have
 $\deg_{F_r^{(0)}}(u,u')=\big|N_{\oH}^+(r)\setminus(\im\phi'_s\cup\im\phi'_{s'})\big|=(1\pm\gamma')\tfrac{\mu^2
   p n}{2}$. Again since $\gamma'<p$ this is as required by~\ref{lem:match:2},
 and we only need to show that the number of $u,u'\in L_r$ which are both in
 $G_s$ for some $s\in[s^*]$ is at most $\tfrac{p^2n^2}{4\log(pn/2)}$. But any
 given $G_s$ has at most $w_s(r)\le\Delta=\tfrac{cn}{\log n}$ vertices in $L_r$,
 so that for a given $u$ there are at most $\tfrac{cn}{\log n}$ choices of $u'$
 with $u,u'\in V(G_s)$ for some $s\in[s^*]$. Since $|L_r|\le n$ we conclude that
 there are at most $\tfrac{cn^2}{\log n}<\tfrac{p^2n^2}{4\log(pn/2)}$ pairs
 $u,u'\in L_r$ such that $u,u'\in V(G_s)$ for some $s\in[s^*]$. This completes
 the verification of~\ref{lem:match:2} for $F=F_r^{(0)}$, $F'=F_r^{(r-1)}$ and
 every $r\in [n]$.
 
\underline{Property~\ref{lem:match:3}:} 
This property does not hold deterministically, but we shall show that it
holds for all~$r$ with high probability. For this purpose we define the
following events.
For each $r\in[n]$ let $\mathcal{E}_r$ be the event that for each $y\in V(F_r^{(0)})$ we have
\begin{equation}\label{eq:defE}
  \deg_{F_r^{(0)}}(y)-\deg_{F_r^{(r-1)}}(y)\le
  50 p^2n\mu^{-2}\,,
\end{equation}
 that is, $\mathcal{E}_r$ is the event that~\ref{lem:match:3} holds for
 $F=F_r^{(0)}$ and $F'=F_r^{(r-1)}$. We shall prove the following claim below,
 but first show how it implies the theorem.

 \begin{claim}\label{cl:probE}
   With probability at least $1-n^{-1}$ for every $r\in [n]$ the event
   $\mathcal{E}_r$ holds.
 \end{claim}

 If $\mathcal{E}_r$ holds
 then all conditions of Lemma~\ref{lem:match} are satisfied 
 for $F=F_r^{(0)}$, $F'=F_r^{(r-1)}$. In this case, we can apply the lemma, and
 obtain a perfect matching $\sigma_r$ in $F_r^{(r-1)}$ with the following property.
 Let $\hist_{r-1}$ consist of the collection of matchings
 $\sigma_1,\dots,\sigma_{r-1}$ obtained in earlier rounds.

 \begin{claim}\label{cl:nicematch}
  For $r\in[n]$, for $u\in N^+_{\oH}(r)$ and for $s^*-\lfloor\mu n\rfloor<s\le s^*$ the following holds.
  Either $\cE_r$ does not occur or a random perfect matching~$\sigma_r$ in
  $F_r^{(r-1)}$ satisfies
  \begin{align*}
   \Prob\Big[\sigma_r^{-1}(u)\in V(G_s)|\hist_{r-1}\Big]\le\tfrac{4w_s(r)}{\mu pn}\,.
  \end{align*}
 \end{claim}
 \begin{claimproof}
   If $\cE_r$ occurs, then all properties \ref{lem:match:1}--\ref{lem:match:3} from Lemma~\ref{lem:match} are
   satisfied with $F=F_r^{(0)}$ and $F'=F_r^{(r-1)}$, 
   and thus, a random matching~$\sigma_r$ in~$F'$ satisfies
   for any given edge $xu\in E\big(F_r^{(r-1)}\big)$ 
   \[
     \Prob\big[xu\in\sigma_r\big|\hist_{r-1}\big]\le\tfrac{4}{\mu p n}.
   \] 
   Taking the union bound over the $w_s(r)$ choices of $x\in L_r$ which are in $G_s$, the claim follows.
 \end{claimproof}

 Hence, assuming Claim~\ref{cl:probE}, we get that
 Algorithm~\ref{alg:MatchLeaves} does not halt with failure in any round with probability at
 least $1-n^{-1}$ and provides matchings
 $\sigma_1,\dots,\sigma_r$. \PerfectPacking{} uses these matchings to define
 for each $s\in[s^*]$ the map $\phi_s:V(G_s)\to
 V(\widehat{H})$ by setting
 \[\phi_s(x)=\begin{cases} \phi'_s(x) & \text{ if }x\in V(G_s)\cap\dom(\phi_s)\\
                             \sigma_{r}(x) & \text{ if }x\in L_{r}\,.
                           \end{cases}
 \]
 Recall that for each $s$, the map $\phi'_s$ is an embedding
 of $G'_s$ into $\widehat{H}$. All the edges of $G_s$ which are not in $G'_s$
 have one end in the removed leaves $V_s$ and the other end in $V(G'_s)$.
 Consider those leaves of $G_s$ which are adjacent to $x\in V(G'_s)$. By
 definition, these are in $L_{\phi_s(x)}$ and by construction of
 $F_{\phi_s(x)}^{(\phi_s(x)-1)}$, they are embedded to distinct vertices of
 $\widehat{H}$ which are adjacent in $H$ to $\phi_s(x)$ and which are neither in
 $\im\phi_s$, nor are of the form $\sigma_i(y)$ for some $i<\phi_s(x)$ and $y\in
 V_s$. It follows that $\phi_s$ is indeed an embedding of $G_s$ into $\widehat{H}$ for
 each $s\in[s^*]$. 
   
 We now check that these embeddings together form a packing. The maps
 $(\phi'_s)_{s\in[s^*]}$ pack the graphs $(G'_s)_{s\in[s^*]}$ into
 $\widehat{H}$, leaving exactly the edges of $H$ unused. By construction $\oH$
 is an orientation of $H$, so for $\vec{vu}\in E(\oH)$, the edge $uv\in E(H)$ is
 used in the embedding of $G_s$, where $\sigma_v^{-1}(u)\in V(G_s)$. It follows
 that each edge of $\widehat{H}$ is used in the maps $(\phi_s)_{s\in[s^*]}$ at
 least once, and since $\sum_{s\in[s^*]}e(G_s)=e(\widehat{H})$ each edge must be
 used exactly once. This justifies that the maps $(\phi_s)_{s\in[s^*]}$
 perfectly pack the graphs $(G_s)_{s\in[s^*]}$ into $\widehat{H}$, as desired.
 
 \medskip

 So it remains to verify Claim~\ref{cl:probE}. We shall first argue that the claimed
 probability bound follows from a probability bound, given in~\eqref{eq:goody},
 which is of the right form to use Corrollary~\ref{cor:freedm}.
 Indeed,
 let $\mathcal{A}_r$ be the event that
 $\mathcal{E}_i$ holds for each $1\le i<r$ but $\mathcal{E}_r$ does not hold.
 Observe that if for each $r$ the event $\mathcal{A}_r$ does not hold, then
 $\mathcal{E}_r$ holds for each $r\in[n]$. In particular, by the union bound
 over $r\in[n]$ it
 suffices to show that for each fixed $r\in[n]$ we have $\Prob[\mathcal{A}_r]\le
 n^{-2}$. Further, by another union bound over the at most
 $v(F_r^{(0)})=2w(r)\le 2n$ different $y\in V(F_r^{(0)})$ and since $\mathcal{A}_r\subset \bigcap_{1\le i\le
   r-1}\mathcal{E}_i$ it is enough to show that for a fixed $y\in V(F_r^{(0)})$
 \begin{equation}\label{eq:goody}
  \Prob\Big[\bigcap_{1\le i\le r-1}\mathcal{E}_i\quad\text{and}\quad\deg_{F_r^{(0)}}(y)-\deg_{F_r^{(r-1)}}(y)>50 p^2n\mu^{-2}\Big]\le\tfrac12n^{-3}\,,
 \end{equation}
 where we used the definition of~$\mathcal{E}_i$ (see~\eqref{eq:defE}). The
 remainder of this proof is devoted to establishing this bound. We will use
 Corollary~\ref{cor:freedm} for this purpose, with the good event $\bigcap_{1\le
   i\le r-1}\mathcal{E}_i$. To that end, define for each $1\le i\le r-1$ the
 random variable
 \[Y_i:=\deg_{F_r^{(i-1)}}(y)-\deg_{F_r^{(i)}}(y)\]
 and observe that
 \[
 \deg_{F_r^{(0)}}(y)-\deg_{F_r^{(r-1)}}(y)
 = \sum_{i=1}^{r-1} Y_i\,. 
 \]
 
To apply Corollary~\ref{cor:freedm} we need to find the range of each $Y_i$ and the expectation of each $Y_i$, conditioned on the history $\hist_{i-1}$ which consists of the collection of matchings $\sigma_1,\dots,\sigma_{i-1}$. This is encapsulated in Claim~\ref{cl:niceY}. 
 
  \begin{claim}\label{cl:niceY} For each $1\le i\le r-1$, we have $0\le Y_i\le\Delta$. Furthermore, either some $\mathcal{E}_i$ with $1\le i\le r-1$ does not occur, or we have $\sum_{i=1}^{r-1}\Exp[Y_i|\hist_{i-1}]\le 25p^2n\mu^{-2}$.
 \end{claim}
 \begin{claimproof}
  We first show $0\le Y_i\le\Delta$. There are two cases to consider. First, if $y\in L_r$, then $y$ is in $G_s$ for some $s\in[s^*]$. An edge $yu$ of $F_r^{(i-1)}$ is removed to form $F_r^{(i)}$ only if $u$ is assigned a leaf of $G_s$ in $\sigma_i$. Since there are at most $w_s(i)\le\Delta$ such leaves, we have $Y_i\le\Delta$ in this case. Second, if $y\in N^+_{\oH}(r)$, and $y$ is assigned a leaf of $G_s$ in $\sigma_i$, then we remove all edges of $F_r^{(i-1)}$ from $y$ to leaves of $G_s$ to form $F_r^{(i)}$.  Since $\sigma_i$ is a matching, this happens for at most one $s\in[s^*]$. There are at most $w_s(r)\le\Delta$ such leaves of $G_s$, so also in this case we have $Y_i\le\Delta$.
  
  We now bound above the sum of conditional expectations. Again, there are two cases to consider. First, if $y\in L_r$, then let $s$ be such that $y\in V(G_s)$. Suppose that $\hist_{i-1}$ is a history up to and including $\sigma_{i-1}$ such that $\cE_i$ holds. By linearity of expectation, and because $\sigma_i^{-1}(u)\in V(G_s)$ means that some leaf of $G_s$ is matched to $u$ in $\sigma_i$, we have
  \begin{equation*}\begin{split}
   \Exp\big[Y_i\big|\hist_{i-1}\big]&=\sum_{\substack{u\in N_{F_r^{(i-1)}}(y)\\ \vec{iu}\in E(\oH)}}\Prob\big[\sigma_i^{-1}(u)\in V(G_s)\big|\hist_{i-1}\big]\\
   &\le\sum_{\substack{u\in N_{F_r^{(i-1)}}(y)\\ \vec{iu}\in E(\oH)}}w_s(i)\tfrac{4}{\mu p n}\le\sum_{u\in N_H(r,i)}w_s(i)\tfrac{4}{\mu p n}\,,
   \end{split}\end{equation*}
  where the first inequality is by Claim~\ref{cl:nicematch} and the second holds since $\vec{iu}\in E(\oH)$ implies $iu\in E(H)$ and since $u\in N_{F_r^{(i-1)}}(y)$ implies $ru\in E(H)$. Summing over $i$, either some $\cE_i$ with $i\in[r-1]$ does not hold, or we have
  \[\sum_{i=1}^{r-1}\Exp\big[Y_i\big|\hist_{i-1}\big]\le \sum_{i=1}^{r-1}\sum_{u\in N_H(r,i)}w_s(i)\tfrac{4}{\mu p n}\le\sum_{i=1}^{n}\big|N_H(r,i)\big|\cdot w_s(i)\tfrac{4}{\mu p n}\le 2p^2n\cdot\tfrac{4}{\mu p n}\sum_{i=1}^nw_s(i)\,,\]
  where the final inequality is by~\ref{appl:quasi}. Recall that we defined $p=\lfloor\mu n\rfloor\lfloor\nu n\rfloor\binom{n}{2}^{-1}$, so in particular $\nu n\le \tfrac{pn}{\mu}$. Since $\sum_{i=1}^nw_s(i)=\lfloor\nu n\rfloor\le\tfrac{pn}{\mu}$ counts the number of leaves removed from $G_s$ to form $G'_s$, we obtain that either some $\cE_i$ with $i\in[r-1]$ does not hold, or
  \[\sum_{i=1}^{r-1}\Exp\big[Y_i\big|\hist_{i-1}\big]\le\tfrac{8p^2n}{\mu^2}\,,\]
  as desired.
  
  Finally, we consider the case $y\in N^+_{\oH}(r)$. If a leaf of $G_s$ is assigned to $y$ by $\sigma_i$, it follows that $y$ is adjacent to $w_s(r)$ leaves of $G_s$ in $F_r^{(i-1)}$ and the edges to these leaves are exactly the edges at $y$ removed from $F_r^{(i-1)}$ to form $F_r^{(i)}$. Suppose that $\hist_{i-1}$ is a history up to and including $\sigma_{i-1}$ such that $\cE_i$ holds. Since a leaf of $G_s$ can only be assigned to $y$ by $\sigma_i$ if $\vec{iy}\in E(\oH)$, and by linearity of expectation, we have
  \begin{equation*}\begin{split}
   \Exp\big[Y_i\big|\hist_{i-1}\big]&=\sum_{s^*-\lfloor\mu n\rfloor<s\le s^*}\ONE_{y\in N^+_{\oH}(i)}\Prob\big[\sigma_i^{-1}(y)\in V(G_s)\big|\hist_{i-1}\big]\cdot w_s(r)\\
   &\le \sum_{s^*-\lfloor\mu n\rfloor<s\le s^*}\ONE_{y\in N^+_{\oH}(i)}w_s(r)w_s(i)\tfrac{4}{\mu p n}\,,
  \end{split}\end{equation*}
  where the second line follows by Claim~\ref{cl:nicematch}. Summing over $i$, either some $\cE_i$ with $i\in[r-1]$ does not hold, or we have
  \begin{equation*}\begin{split}
   \sum_{i=1}^{r-1} \Exp\big[Y_i\big|\hist_{i-1}\big]&\le \sum_{i=1}^{r-1}\sum_{s^*-\lfloor\mu n\rfloor<s\le s^*}\ONE_{y\in N^+_{\oH}(i)}w_s(r)w_s(i)\tfrac{4}{\mu p n}\\
   &\le\sum_{s^*-\lfloor\mu n\rfloor<s\le s^*} \sum_{i=1}^{n}\ONE_{y\in N^+_{\oH}(i)}w_s(r)w_s(i)\tfrac{4}{\mu p n}\\
   &=\sum_{s^*-\lfloor\mu n\rfloor<s\le s^*} \sum_{v:\vec{vy}\in E(\oH)}^{n}w_s(r)w_s(v)\tfrac{4}{\mu p n}\\
   &\le \sum_{s^*-\lfloor\mu n\rfloor<s\le s^*} \tfrac{4w_s(r)}{\mu p n}\sum_{v:vy\in E(H)}w_s(v)\\
   &\le \sum_{s^*-\lfloor\mu n\rfloor<s\le s^*}\tfrac{4w_s(r)}{\mu p n}\cdot \tfrac{10p^2n}{\mu}=\tfrac{40p}{\mu^2}\sum_{s^*-\lfloor\mu n\rfloor<s\le s^*}w_s(r)\,,
  \end{split}\end{equation*}
  where the last inequality is by~\ref{appl:sumweight}. By definition of $w(r)$, by~\ref{appl:weight} and by choice of $\gamma'$ we have $\sum_{s^*-\lfloor\mu n\rfloor<s\le s^*}w_s(r)=w(r)\le \frac{5}{8}pn$, so we conclude that either some $\cE_i$ with $i\in[r-1]$ does not hold, or we have
  \[\sum_{i=1}^{r-1} \Exp\big[Y_i\big|\hist_{i-1}\big]\le \tfrac{40p}{\mu^2}\cdot \frac{5}{8}pn=\tfrac{25p^2n}{\mu^2}\,,\]
  as desired.
\end{claimproof}

  Using Claim~\ref{cl:niceY}, We are now in a position to apply 
 Corollary~\ref{cor:freedm}, with $R=\Delta=\tfrac{cn}{\log n}$, with
 $\tmu=25p^2n\mu^{-2}$, and with the event
 $\cE=\bigcap_{i=1}^{r-1}\mathcal{E}_i$, which gives
 \[\Prob\Big[\bigcap_{1\le i\le r-1}\mathcal{E}_i\quad\text{and}\quad\sum_{i=1}^{r-1}Y_i>50p^2n\mu^{-2}\Big]\le\exp\big(-\tfrac{\tmu}{4R}\big)=\exp(-6.25c^{-1}p^2\mu^{-2}\log n)<\tfrac12n^{-3}\,,\]
 where the final inequality is by choice of $c$. This
 establishes~\eqref{eq:goody}.
\end{proof}

\section{Proof of the orientation lemma}
\label{sec:orient}

In this section we prove Lemma~\ref{lem:orient}.

\begin{proof}[Proof of Lemma~\ref{lem:orient}]
By the given quasirandomness of $H$ we know that
$\deg_{H}(v)=(1\pm \gamma'^3)pn$ and
$|N_{H}(v)\cap N_{H}(w)|=(1\pm \gamma'^3)p^2n$ 
for every $v\neq w\in V(H)$.
Applying a standard Chernoff argument, i.e. using Theorem~\ref{thm:chernoff}, we
obtain that a.a.s.\ 
for every $v\neq w\in V(H)$ we have
\[
\deg_{\vec{H}_0}^+(v)=(1\pm 2\gamma'^3)\frac{pn}{2}
\text{ and }
|N_{\vec{H}_0}^+(v)\cap N_{\vec{H}_0}^-(w)|=(1\pm 2\gamma'^3)\frac{p^2n}{4}\, .
\] 
From now on fix an arbitrary orientation $\vec{H}_0$
satisfying these two properties.
Starting with $\vec{H}_0$ we aim to switch the orientations
of some edges until we find an oriented graph $\vec{H}$
as desired. In order to do so, we will successively
switch the orientations of pairs of edges, 
thus producing a sequence
of oriented graphs $(\vec{H}_i)_{0\leq i\leq t}$ 
that eventually end up with $\vec{H}_t=\vec{H}$.
For any such oriented graph $\vec{H}_i$
and every vertex $v\in V(H)$ we
define the potential $\phi_i(v):=\deg_{\vec{H}_i}^+(v)-w(v)$ and
\[
\phi(\vec{H}_i) := \sum_{v\in V(H)} |\phi_i(v)|~ .
\]
Initially we have $|\phi_0(v)|\leq 3\gamma'^3 pn$ for every $v\in V(H)$. 

The algorithm \OrientationSwitch{} describes how orientations are switched. In
every iteration of this algorithm, the central idea is to change the potential
of two vertices $x,y\in V(H)$ with $\phi_i(x)>0$ and $\phi_i(y)<0$ in the
following way: We choose a vertex $m\in N_{\vec{H}_i}^+(x)\cap
N_{\vec{H}_i}^-(y)$ uniformly at random. We then \emph{switch} (the orientation
of) the directed edge $xm$, that is we replace $xm$ with $mx$, and we also
switch the edge $my$. Switching these two edges creates a new orientation
$\vec{H}_{i+1}$ of $H$. The vertex $m$ will be called the \emph{middle vertex}
of the switching, while $x$ and $y$ are called the \emph{end vertices}. In case
that
\[\big|\{uv\in E(H)\colon uv \text{ is oriented differently in } \vec{H}_i
  \text{ and } \vec{H}_0\}\big|\]
gets too large in some round $i$ and for some
vertex $v$, we let the algorithm halt with failure. However, we will see in the
following that this happens with probability tending to 0.

 \begin{algorithm}[ht]
   \caption{\OrientationSwitch{}}\label{alg:OrientationSwicth}
   let $t:=\phi(\vec{H}_0)/2$\;
    \For{$i=0$ \KwTo $t-1$}{
     \lIf{$\exists~v$ with 
     $\big|\{uv\in E(H)\colon uv \text{ is oriented differently in }
     \vec{H}_i \text{ and } \vec{H}_0\}\big|>100{\gamma'}^3 n$ \\\quad\mbox{}}{
        halt with failure}
      choose vertices $x,y\in V(H)$ with $\phi_i(x)>0$
      and $\phi_i(y)<0$\;
      choose a vertex $m\in N_{\vec{H}_i}^+(x)\cap N_{\vec{H}_i}^-(y)$
      uniformly at random\;
      create the new oriented graph $\vec{H}_{i+1}$
      by switching the orientations of $xm$ and $my$\;
    }
   \Return $H_t$ \;
  \end{algorithm}
 
We start with some easy observations.
 
\begin{observation}\label{obs:switch1}
As long as the algorithm does not halt with failure
we have 
\[
\phi(\vec{H}_{i+1})=\phi(\vec{H}_i)-2\, .
\]
\end{observation}

\begin{observation}\label{obs:switch2}
For every vertex $v\in V(H)$ with $\phi_0(v)>0$ 
(or $\phi_0(v)<0$) 
it holds that $\phi_{i-1}(v)\geq \phi_i(v)\geq 0$ (or $\phi_{i-1}(v)\leq \phi_i(v)\leq 0$) for all $i\in [t]$.
\end{observation} 
 
Indeed, both observations
hold since the switching of the orientations of $xm$
and $my$ ensures that $\phi_{i+1}(x)=\phi_i(x)-1$
and $\phi_{i+1}(y)=\phi_i(y)+1$, while
the potentials of all the other vertices do not change.

\begin{claim}
A.a.s.\ throughout the algorithm every vertex $v\in V(H)$ is chosen at most $40\gamma'^3 n$ times as the middle vertex of a switching.
\end{claim} 

\begin{proof}
Every vertex $v\in V(H)$ can become a middle vertex only
if $v\in N_{\vec{H}_i}^+(x)\cap N_{\vec{H}_i}^-(y)$
for some $0\leq i\leq t-1$ and $x,y\in V(H)$
with $\phi_i(x)>0$ and $\phi_i(y)<0$. Now, $v$
has at most $(1+\gamma'^3)pn<2pn$ neighbours $x\in V(H)$ and every such vertex with positive potential 
participates in a switching as an end vertex 
in at most $|\phi_0(x)|\leq 3\gamma'^3 pn$
rounds. Thus, there are at most $6\gamma'^3 p^2n^2$
rounds which may consider $v$ as a suitable middle vertex.
In each such round, the middle vertex is chosen
uniformly at random from a set $N_{\vec{H}_i}^+(x)\cap N_{\vec{H}_i}^-(y)$.
As long as the algorithm does not halt with failure
we have 
\[
|N_{\vec{H}_i}^+(v)\cap N_{\vec{H}_i}^-(w)| = 
|N_{\vec{H}_0}^+(v)\cap N_{\vec{H}_0}^-(w)| \pm 2\cdot 100\gamma'^3 n
=(1\pm \gamma'^2)\frac{p^2n}{4}~ .
\]
Thus, when $v$ is suitable for being a middle vertex, the probability that $v$ is chosen is bounded from above by
$\frac{5}{p^2n}$. Now, applying a Chernoff-type argument
the claim follows.
\end{proof}
 
With the above statements in hand, we can show that 
a.a.s.\ \OrientationSwitch{} does not halt with failure 
and that the resulting oriented graph
$\vec{H}=\vec{H}_t$ satisfies the properties~\ref{orient:deg} and~\ref{orient:diff}. 
Indeed, let $v\in V(H)$ be any vertex.
In some round, we change the orientation of exactly one edge incident
with $v$ 
if and only if $v$ is an end vertex of the switching in this round.
As such a switching decreases $|\phi_i(v)|$ by 1
and since $|\phi_i(v)|$ never increases
according to Observation~\ref{obs:switch2}, this happens
at most $|\phi_0(v)|\leq 3\gamma' pn$ times.
Moreover, we change the orientation of exactly two 
edges incident with $v$ if and only if $v$ is a middle vertex
of a switching. By the above claim a.a.s.\ this happens
at most $40\gamma'^3 n$ times. 
Thus, as long as the algorithm runs, we a.a.s.\ switch the orientations of at most
\[
3\gamma' pn + 2\cdot 40\gamma'^3 n < 100\gamma'^3 n < \gamma'^2 n 
\]
edges incident with $v$.
It follows that the algorithm runs without failures, and also that
property~\ref{orient:diff} holds. 
By Observation~\ref{obs:switch1} and since $t=\phi(\vec{H}_0)/2$ we obtain that
$\phi(\vec{H}_t)=0$, meaning that~\ref{orient:deg} holds for $\vec{H}=\vec{H}_t$.
\end{proof}

\section{Proof of the matching lemma}
\label{sec:match}

In this section we provide the proof of Lemma~\ref{lem:match}. The proof of
this lemma is the only place in this paper where we use the concept of a
regular pair.

\begin{definition}[density, $(\eps,d)$-regular, $(\eps,d)$-super-regular]
  Let~$G$ be a graph and $U,W\subset V(G)$ be disjoint vertex sets. The
  \emph{density} of the bipartite graph $G[U,W]$ is 
  \[d_G(U,W)=\frac{e(G[U,W])}{|U||W|}\,.\]
  We say that $G[U,W]$ is \emph{$(\eps,d)$-regular} if for all $U'\subset U$ and
  $W'\subset W$ with $|U'|\ge\eps|U|$ and $|W'|\ge\eps|W|$ we have
  \[d_G(U',W')=d\pm\eps\,.\]
  The graph $G[U,W]$ is \emph{$(\eps,d)$-super-regular} if it is
  $(\eps,d)$-regular and for all $u\in U$ and for all $w\in W$ we have
  \[\deg_{G[U,W]}(u)=(d\pm\eps)|W|, \qquad\text{and}\qquad
    \deg_{G[U,W]}(w)=(d\pm\eps)|U|\,.\]
\end{definition}

It is well-known that regular pairs are forced by a degree-codegree condition;
we use the following formulation due to
Duke, Lefmann, and R\"odl in~\cite{DukLefRod}.

\begin{lemma}[degree-codegree condition~\cite{DukLefRod}]
  \label{lem:DLR}
  Assume $0<\eps<2^{-200}$ and let $G[U,W]$ be a bipartite graph with
  parts~$U$ and~$W$ of size $|U|=|W|=n$ and density $d=d_{G[U,W]}(U,W)$.
  If
  \begin{enumerate}[label=\rom]
    \item $\deg_{G[U,W]}(u)>(d-\eps)|W|$ for all $u\in U$, and
    \item $\deg_{G[U,W]}(u,u')<(d+\eps)^2|W|$ for all but at most
      $2\eps|U|^2$ pairs $\{u,u'\}\in\binom{U}{2}$,
  \end{enumerate}
  then $G[U,W]$ is $(\eps^{\frac16},d)$-regular.
  \qed
\end{lemma}

If we choose a perfect matching uniformly at random in a super-regular pair
then each edge is roughly equally likely to appear in the matching, as was shown
by Joos (see~\cite{KKOT}).

\begin{lemma}[perfect matchings in super-regular pairs
  {\cite[Theorem~4.3]{KKOT}}]
  \label{lem:KKOT}
  Assume $0\ll\frac{1}{m'}\ll\eps'\ll d\ll 1$. Let $G[U,W]$ be an
  $(\eps',d)$-super-regular graph with $|U|=|W|=m'$.
  Then $G[U,W]$ contains a perfect matching.
  Moreover, for a perfect matching $\sigma$ chosen uniformly at random
  among all perfect matchings in $G[U,W]$ and for all $uw\in E(G)$ we have
  \begin{flalign*}
   &&\Prob[\sigma(u)=w]=(1\pm(\eps')^{\frac1{20}})\frac{1}{dm'}~ .&&\qed
  \end{flalign*}
\end{lemma}

The proof of the matching lemma simply combines these two lemmas.

\begin{proof}[Proof of Lemma~\ref{lem:match}]
  By~\ref{lem:match:1} and~\ref{lem:match:3}, for all $x\in U\cup W$ we
  have
  \begin{equation}
    \label{eq:match:deg}
    \deg_{F'}(x)=\big(\mu\pm \tfrac{200p}{\mu^2}\big)m\,
  \end{equation}
  By~\ref{lem:match:2} and~\ref{lem:match:3}, for all but at most
  $\frac{m^2}{\log m}$ pairs $\{u,u'\}\in\binom{U}{2}$ we have
  \begin{equation}
    \label{eq:match:codeg}
    \deg_{F'}(u,u')=\big(\mu^2\pm \tfrac{300p}{\mu^2}\big)m\,.
  \end{equation}
  We want to apply Lemma~\ref{lem:DLR} with $d=\mu$ and $\eps=400p/\mu^3$ to
  conclude that $F'[U,W]$ is super-regular, and now check the
  conditions of this lemma.
  By~\eqref{eq:match:deg}, for $u\in U$ we have
  \begin{equation*}
      \deg_{F'}(u)=\big(d\pm \tfrac{200p}{\mu^2}\big)m 
      = \big(d\pm \tfrac{200p}{\mu^2}\big)\frac{|W|}{1\pm p}
      = \big(d\pm \tfrac{400 p}{\mu^2}\big)|W|
      >(d-\eps)|W|\,, 
  \end{equation*}
  and similarly for $w\in W$ we have $\deg_{F'}(w)=\big(d\pm \tfrac{400p}{\mu^2}\big)|U|>(d-\eps)|U|$.
  By~\eqref{eq:match:codeg}, for all but at most $\frac{m^2}{\log m}$ pairs
  $\{u,u'\}\in\binom{U}{2}$ we have
  \begin{equation*}\begin{split}
      \deg_{F'}(u,u')&\le\big(d^2+\tfrac{300p}{\mu^2}\big)m\le\big(d^2+\tfrac{300p}{\mu^2}\big)\frac{|W|}{{1-p}}
      \le\big(d^2+\tfrac{400p}{\mu^2}\big)|W|<(d^2+2\eps d+\eps^2)|W| \\
      &=(d+\eps)^2|W|\,,
  \end{split}\end{equation*}
  where the last inequality uses $d=\mu$ and $\eps=400p/\mu^3$.
  We conclude that, if 
  $\frac{m^2}{\log m}\leq 2\eps |U|^2$ which holds for  
  $\log m>1/\eps$, then $F'$ is
  $\big((400\frac{p}{\mu^3})^{\frac16}, \mu\big)$-regular by
  Lemma~\ref{lem:DLR}.
  Since $\deg_{F'}(x)=\big(d\pm \tfrac{400p}{\mu^2}\big)|U|$ for all $x\in U\cup W$, it follows
  that~$F'$ is  $\big((400\frac{p}{\mu^3})^{\frac16},\mu\big)$-super-regular.
  
  Hence we can apply Lemma~\ref{lem:KKOT} to~$F'$ with 
  \[m'=|U|=(1\pm p)m\,, \quad
    \eps'=\Big(400\frac{p}{\mu^3}\Big)^{\frac16}\,, \quad \text{and} \quad
    d=\mu\,,
  \]
  and conclude that~$F'$ has a perfect matching and that for a perfect
  matching~$\sigma$ chosen uniformly at random among all perfect matchings
  of~$F'$ and for all $uw\in E(F')$ we have
  \[\Prob[\sigma(u)=w]=\Big(1\pm\Big(400\frac{p}{\mu^3}\Big)^{\frac1{120}}\Big)\frac{1}{\mu(1\pm
      p)m} \le \frac{2}{\mu m}\,,\]
  where the inequality holds if~$p$ is small and $\big(400\frac{p}{\mu^3}\big)^{\frac1{120}}\le\frac{1}{100}$.
\end{proof}

\section{Proof of the almost perfect packing lemma}
\label{sec:appl}

In this section we prove Lemma~\ref{lem:appl}. This is the technical part
of this paper, which requires some stamina. 

We start this section by explaining the setup which we use throughout. Then, in
Section~\ref{subsec:inv} we define some auxiliary properties that our random
packing process preserves.
In Section~\ref{subsec:RandomEmbedding} we analyse the behaviour of the algorithm
\RandomEmbedding{},
and in Section~\ref{subsec:PackingProcess} the behaviour of \PackingProcess{}.
In Section~\ref{subsec:appl_proof}, finally, we use the obtained results to show
Lemma~\ref{lem:appl}. 

In the results in this section we shall use the following setup.

\begin{setting}\label{set:graphs}
  We use the constants defined in Setting~\ref{set:const}.

  Let $(G''_s)_{s\in[s^*]}$ (for some $s^*\le \frac{7}{4}n$) be graphs on $[n]$, such
  that for each $s$ and $x\in V(G''_s)$ we have $\LEFTDEG_{G''_s}(x)\le D$, such
  that $\Delta(G''_s)\le cn/\log n$, and such that the final $\delta n$
  vertices of $G''_s$ all have degree $d_s$ and form an independent set.

  Let $\widehat H$ be a $(\xi,2D+3)$-quasirandom graph with $n$ vertices and density $\hat{p}$.
  Recall that \PackingProcess{} chooses $H_0^*$ as a subgraph of~$\widehat H$ by picking edges
  of~$\widehat H$ independently with probability~$\gamma\binom{n}{2}/e(\widehat
  H)$. We will assume that $e(H_0^*)\leq 1.1\gamma \binom{n}{2}$.
\end{setting}

We note at this point that we assume $e(H_0^*)\leq 1.1\gamma \binom{n}{2}$ in
order to make use of (\ref{eq:ps}). This inequality holds with probability at
least $1-e^{-n}$ and hence this assumption does not affect the proof of
Lemma~\ref{lem:appl}, since we will see that, if this inequality holds, each of
the properties \ref{appl:quasi} -- \ref{appl:sumweight} occurs with probability
at least $1-n^{-4}$.

\subsection{Coquasirandomness, diet, codiet, and cover}
\label{subsec:inv}

The following properties coined in~\cite{DegPack} are preserved throughout the run of our
random packing process. Firstly, for our analysis of \PackingProcess{}, we need
the concept of coquasirandomness. This controls the intersections of vertex
neighbourhoods in two edge-disjoint graphs on the same vertex set.

\begin{definition}[coquasirandom]
    For $\alpha>0$ and $L\in\mathbb N$, we say that a pair of graphs $(F, F^*)$, both on the same vertex set~$V$ of order~$n$
    and with densities $p$ and $p^*$, respectively, is \emph{$(\alpha,L)$-coquasirandom} if for
    every set $S\subset V$ of at most $L$ vertices and every subset $R\subseteq S$ we have
    \[|N_{F}(R)\cap N_{F^*}(S\setminus R)|
    =(1\pm\alpha)p^{|R|}(p^*)^{|S\setminus R|}n\,.\]
\end{definition}

For the analysis of one run of \RandomEmbedding{} we further need the following concepts.

\begin{definition}[diet condition, codiet condition, cover condition]
  \label{def:dietcover} 
  Let~$H$ be a graph with~$n$ vertices and $p\binom{n}{2}$ edges, and let
  $X\subseteq V(H)$ be any vertex set. We say that the pair $(H,X)$
  satisfies the \emph{$(\beta,L)$-diet condition} if for every set
  $S\subset V(H)$ of at most~$L$ vertices we have 
  \[|N_{H}(S)\setminus
  X|=(1\pm\beta)p^{|S|}(n-|X|)\,.\]
  Given further $H^*$ on the same vertex set as $H$, which has no edges in common with $H$ and which has $p^*\binom{n}{2}$ edges, we say that the triple $(H,H^*,X)$ satisfies the \emph{$(\beta,L)$-codiet condition} if for every set $S\subset V(H)$ of at most $L$ vertices, and for every $R\subset S$, we have
  \[\big|\big(N_{H}(R)\cap N_{H^*}(S\setminus R)\big)\setminus
  X\big|=(1\pm\beta)p^{|R|}(p^*)^{|S\setminus R|}(n-|X|)\,.\]
  Further, let~$G$ be a graph with vertex set~$[n]$.
  Given $\eps>0$,  $i\in [n-\epsilon n]$, and $d\in\mathbb N$, we define
  \[X_{i,d}:=\{x\in V(G)\colon i\le x<i+\varepsilon n,|\LNBH_G(x)|=d\}\,.\] 
  We say that a partial embedding $\psi$ of $G$ into $H$, which embeds
  $\LNBH_G(x)$ for each $i\le x<i+\varepsilon n$, satisfies the
  \emph{$(\eps,\beta,i)$-cover condition} if for each $v\in V(H)$, and for each $d\in\mathbb N$,
  we have
  \[
  \big|\big\{ x\in X_{i,d}:v\in N_{H}\big(\psi(\LNBH_G(x))\big)\big\}\big|=(1\pm\beta)p^{d}|X_{i,d}|\pm\varepsilon^{2}n\,.
  \]
\end{definition}

Following~\cite{DegPack}, we use Definition~\ref{def:dietcover} to define key events $\DietEvent(\cdot;\cdot)$, $\CoverEvent(\cdot;\cdot)$, $\CodietEvent(\cdot)$ on the probability space $\Omega^{G\AlgMap H}$ underlying the run of \RandomEmbedding\ which attempts to embed $G$ into $H$. (For a formal definition of this probability space, see~\cite[Section~4.1]{DegPack}.)

 Suppose that $D$, $\delta$ and $\eps$ are as in Setting~\ref{set:const}.
 Suppose that $\lambda>0$. Suppose that we have graphs $G$ and $H$ as in
 Algorithm~\ref{alg:embed}. Suppose that we run \RandomEmbedding{} to partially
 embed $G$ into $H$. Let $(\psi_{i})_{i\in[t_*]}$ be the partial embeddings of
 $G\big[[i]\big]$ into $H$, where $t_*=n-\delta n$ if \RandomEmbedding{} succeeded, and otherwise
 $t_*+1$ is the step in which \RandomEmbedding{} halted with failure.
\begin{itemize}[leftmargin=*]
	\item For each $t\in[n-\delta n]$, let $\DietEvent(\lambda;t)\subset\Omega^{G\AlgMap H}$ correspond to executions of \RandomEmbedding{} for which $t_*\ge t$ and the pair $(H,\im\psi_{t})$
	satisfies the $(\lambda,2D+3)$-diet condition.
	\item For each $t\in[n-\delta n]$, let $\CoverEvent(\lambda;t)\subset\Omega^{G\AlgMap H}$ correspond to executions of \RandomEmbedding{} for which $t_*\ge t+\eps n$ and the embedding $\psi_{t^*}$ of~$G$ into~$H$
	satisfies the $(\varepsilon,\lambda,t)$-cover condition.
	\item Suppose further that we have a graph $H^*_0$ with $V(H)=V(H^*_0)$. For each $t\in[n-\delta n]$, let $\CodietEvent(t)\subset\Omega^{G\AlgMap H}$ correspond to executions of \RandomEmbedding{} for which $t_*\ge t$ and the triple $(H,H^*_0,\im\psi_t)$ satisfies the $(2\eta,2D+3)$-codiet condition.
\end{itemize}

\subsection{Properties of \RandomEmbedding{}}
\label{subsec:RandomEmbedding}

In this section we collect properties that are preserved during a run of
\RandomEmbedding{}. The constants we use are as in
Setting~\ref{set:const}. However, since we are only concerned with a
single run of \RandomEmbedding{} here, we only consider a single guest
graph~$G$, and a single host graph~$H$ with the following properties.

\begin{setting}\label{set:RandEmb}
  Let~$G$ be a graph on vertex set $[n]$ such that $\LEFTDEG_G(x)\le D$ for each
  $x\in V(G)$ and $\Delta(G)\le cn/\log n$.
  Let~$H$ be an $(\alpha,2D+3)$-quasirandom graph with $n$ vertices and $p\binom{n}{2}$ edges, with $p\ge\gamma$, and suppose that $H^*_0$ is a graph on $V(H)$ such that $(H,H^*_0)$ is $(\eta,2D+3)$-coquasirandom.
\end{setting}

The following lemma comes from~\cite[Lemma~24]{DegPack} and the deduction of~\cite[Lemma~18]{DegPack} which comes immediately after.
Specifically,~\ref{22:a} is the deduction of \cite[Lemma~18]{DegPack} and~\ref{22:d} is explicitly in
\cite[Lemma~24]{DegPack}, while~\ref{22:b} and~\ref{22:c} differ only from
the statements of~\cite[Lemma~24]{DegPack} in that the error bound we
give here is in terms of $\beta_t$ whereas in~\cite[Lemma~24]{DegPack} a
(larger) error bound $C\alpha$ is given. In the proof
of~\cite[Lemma~24]{DegPack}, the stronger error bounds we claim here are
explicitly obtained. The cover conditions asserted are otherwise identical,
despite being written slightly differently.

\begin{lemma}\label{lem:22}
  Given~$D\in \mathbb{N}$ and $\gamma>0$, let
  $\delta,\alpha_0,\alpha_{2n},C,\eps$ be as in
  Setting~\ref{set:const}. Then the following holds for any
  $\alpha_0\le\alpha\le\alpha_{2n}$ and all sufficiently large $n$.
  Let~$G$,~$H$ and $H^*_0$ be as in Setting~\ref{set:RandEmb}.
  Let $\beta_t=\beta_t(\alpha)$ be as in Setting~\ref{set:const}.
  If we run \RandomEmbedding{} to embed~$G[{\scriptstyle [n-\delta n]}]$ into~$H$,
  then with probability at least $1-2n^{-9}$ 
  \begin{enumerate}[label=\abc]
    \item\label{22:a} \RandomEmbedding{} succeeds in constructing
      partial embeddings $(\psi_i)_{i\in[n-\delta n]}$,
    \item\label{22:b} $(H,\im\psi_t)$ satisfies the $(\beta_t,2D+3)$-diet condition (i.e.\ $\DietEvent(\beta_t;t)$ occurs) for
      each $t\in[n-\delta n]$,
    \item\label{22:c} $\psi_t$ has the $(\eps,20D\beta_{t-\eps n+2},t-\eps n+2)$-cover
      condition (i.e.\ $\CoverEvent(20D\beta_{t-\eps n+2},t-\eps n+2)$ occurs) for each $t\in[\eps n-1,n-\delta n]$.
    \item\label{22:d} $(H,H^*_0,\im\psi_t)$ satisfies the $(2\eta,2D+3)$-codiet condition (i.e.\ $\CodietEvent(t)$ occurs) for
      each $t\in[n-\delta n]$. \qed
  \end{enumerate}
\end{lemma}

The next lemma is proven as part of~\cite[Lemma~26]{DegPack} (it can
be found in~\cite[Claim~26.1]{DegPack}).

\begin{lemma}\label{lem:24}
  Given~$D\in \mathbb{N}$ and $\gamma>0$, let
  $\delta,\alpha_0,\alpha_{2n},C,\eps$ be as in
  Setting~\ref{set:const}. Then the following holds for any
  $\alpha_0\le\alpha\le\alpha_{2n}$ and all sufficiently large $n$.
  Let~$G$ and~$H$ be as in Setting~\ref{set:RandEmb} and let~$1\leq j\leq t+1-\eps n$ for $t\leq (1-\delta)n$.
  Let $\beta_j=\beta_j(\alpha)$ be as in Setting~\ref{set:const}.
  Assume we run \RandomEmbedding{} to embed~$G[{\scriptstyle [n-\delta n]}]$ into~$H$, that it
  produces a partial embedding~$\psi_j$ such that $(H,\im\psi_j)$ has the
  $(\beta_j,2D+3)$-diet condition, and let $T\subset V(H)\setminus\im\psi_j$
  with $|T|\ge\frac12\gamma^{2D+3}\delta n$.
  Then with probability at least $1-2n^{-2D-19}$, one of the following
  occurs.
  \begin{enumerate}[label=\abc]
  \item\label{e:part1} $\psi_t$ does not have the $(\eps,20D\beta_j,j)$-cover condition (i.e.\ $\CoverEvent(20D\beta_j,j)$ does not occur) , or
  \item\label{e:part2} $\big|\{x\colon j\le x<j+\eps n, \psi_{t-1}(x)\in T\}\big|=(1\pm40D\beta_j)\frac{|T|\eps n}{n-j}$.
  \end{enumerate}
\end{lemma}

In~\cite[Lemma~28]{DegPack} we estimated the probability that, when running
\RandomEmbedding{}, a given vertex~$x\in V(H)$ is not used in the embedding
of the first~$t_1$ vertices of~$G$.

\begin{lemma}[Lemma~28 in~\cite{DegPack}]\label{lem:probvtx} 
  Given~$D\in \mathbb{N}$ and $\gamma>0$, let
  $\delta,\alpha_0,\alpha_{2n},C,\eps$ be as in
  Setting~\ref{set:const}. Then the following holds for any
  $\alpha_0\le\alpha\le\alpha_{2n}$ and all sufficiently large $n$.
  Let~$G$ and~$H$ be as in Setting~\ref{set:RandEmb}.  Let
  $0\le t_0<t_1\le n-\delta n$. Let $\histens$ be a history ensemble of
  \RandomEmbedding{} up to time $t_0$, and suppose that
  $\Prob[\histens]\ge n^{-4}$. Then the following hold for any distinct
  vertices $u,v\in V(H)$.
  \begin{enumerate}[label=\abc]
   \item\label{lem:probvtx:a} If $v\not\in\im\psi_{t_0}$ then we have
     \begin{flalign*}
       &&\Prob\big[v\not\in\im\psi_{t_1}|\histens\big]=(1\pm
       100C\alpha\delta^{-1})\tfrac{n-1-t_1}{n-t_0}\,.
       &&
     \end{flalign*}
   \item\label{lem:probvtx:b} If $u,v\not\in\im\psi_{t_0}$ then we have
     \begin{flalign*}
       &&\Prob\big[u,v\not\in\im\psi_{t_1}|\histens\big]=(1\pm
       100C\alpha\delta^{-1})(\tfrac{n-1-t_1}{n-t_0})^2\,.
       &&\qed
     \end{flalign*}
  \end{enumerate}
\end{lemma}

In addition we estimated the probability that a given edge of~$G$ is
embedded to a given edge of~$H$. The following lemma is~\cite[Lemma~29]{DegPack}, together with equation~(6.10) of that paper which is established in the proof.

\begin{lemma}[Lemma~29 in~\cite{DegPack}]\label{lem:probedge}
 Given~$D\in \mathbb{N}$, and $\gamma>0$, let constants
 $\delta,\eps,C,\alpha_0,\alpha_{2n}$ be as in
 Setting~\ref{set:const}. Then the following holds for any
 $\alpha_0\le\alpha\le\alpha_{2n}$ and all sufficiently large $n$. 
 Let~$G$ and~$H$ be as in Setting~\ref{set:RandEmb}. 
 Let $uv$ be an edge of $H$, and let $xy$ be an edge of $G$. When \RandomEmbedding{} is run to embed $G[{\scriptstyle [n-\delta n]}]$ into $H$, we have
 \[\Prob\big[x\AlgMap u,y\AlgMap v\big]=\big(1\pm500C\alpha\delta^{-1}\big)^{4D+2}\cdot p^{-1}n^{-2}\,,\]
 and furthermore the probability that some edge of $G$ is embedded to $uv$ is
\begin{flalign*}
   &&\big(1\pm500C\alpha\delta^{-1}\big)^{4D+2}p^{-1}n^{-2}\cdot 2e(G)\,.
   &&\qed
  \end{flalign*}   
\end{lemma}

We can use these two lemmas for estimating the probability that a given vertex of~$G$
is embedded on a given vertex of~$H$.

\begin{lemma}[embedding a vertex on a given vertex]\label{lem:vertex}
  Given $D\in\mathbb N$, $\gamma>0$, let $\delta$, $\eps$, $C$, $\alpha_0$, $\alpha_{2n}$
  be as in Setting~\ref{set:const} and let $p\ge\gamma$. Let $\alpha_0<\alpha\le\alpha_{2n}$
  and let~$n$ be sufficiently large. 
  Let~$G$ and~$H$ be as in Setting~\ref{set:RandEmb}. 
  Let $x\in V(G)$ with $x\leq (1-\delta)n$ and 
  $u\in V(H)$. When we run
  \RandomEmbedding{} to embed~$G[{\scriptstyle [n-\delta n]}]$ into~$H$, then
  \[\Prob\big[x\AlgMap u\big]=(1\pm10^4C\alpha D\delta^{-1})\frac1n\,.\]
\end{lemma}
\begin{proof}
  While we could prove this lemma directly following the methods of~\cite{DegPack}, it is convenient to deduce it from the results of~\cite{DegPack}. We separate two cases.
  
  If $x$ is an isolated vertex in $G$, then we embed $x$ to $u$ if and only if the first $x-1$ vertices of $G$ are not embedded to $u$, and then among the $n-x+1$ vertices of $H$ to which we could embed $x$, we choose $u$. Using Lemma~\ref{lem:probvtx}\ref{lem:probvtx:a} to estimate the probability of the first event occurring, with $t_0=0$ and $t_1=x-1$ (and so $\histens$ is trivial) we have
  \begin{equation*}\begin{split}
    \Prob\big[x\AlgMap u\big] &=
    \Prob\big[u\not\in\psi_{x-1}\big]
    \Prob\big[x\AlgMap u\big| u\not\in\psi_{x-1}\big]=
    (1\pm 100C\alpha\delta^{-1})\frac{n-1-x+1}{n}\cdot\frac{1}{n-x+1} \\
    & =(1\pm 200C\alpha\delta^{-1})\frac{1}{n}\,.
  \end{split}\end{equation*}

  If, on the other hand, there is~$y$ such that $xy\in E(H)$, then we embed $x$ to $u$ if and only if we embed $x$ to $u$ and $y$ to some neighbour $v$ of $u$ in $H$. Since these events are disjoint as $v$ ranges over the neighbours of $u$, the probability that one of them occurs is exactly the sum of their individual probabilities, and the latter are estimated by Lemma~\ref{lem:probedge}. Since by the $(\alpha,2D+3)$-quasirandomness of $H$, the vertex $u$ has $(1\pm\alpha)pn$ neighbours, we obtain
  \begin{equation*}\begin{split}
    \Prob\big[x\AlgMap u\big] &=
    \sum_{v\in N_H(u)}
    \Prob\big[x\AlgMap u,y\AlgMap v\big]=
    (1\pm \alpha)pn\cdot(1\pm 500C\alpha\delta^{-1})^{4D+2}\frac{1}{pn^2} \\
    & =(1\pm 10^4C\alpha D\delta^{-1})\frac{1}{n}\,.
  \end{split}\end{equation*}
  In either case, we conclude the desired bound.
\end{proof}

We further need the following lemma, estimating the probability that a given
vertex of~$G$ is embedded to a given vertex of~$H$ and another given vertex
of~$H$ is not used in the embedding of the first $n-\lfloor\mu n\rfloor$ vertices of~$G$. We will be interested in this when $G$ is a special graph; so the remaining vertices of $G$ (which \RandomEmbedding{} also embeds) are isolated vertices. The proof of this lemma is rather similar to the proof of~\cite[Lemma~29]{DegPack}.

\begin{lemma}[embedding a vertex on a given vertex and not using another vertex]
\label{lem:vertexnotimage}
  Given $D\in\mathbb N$, $\gamma>0$, let $\delta$, $\eps$, $C$, $\alpha_0$, $\alpha_{2n}$
  be as in Setting~\ref{set:const} and let $p\ge\gamma$. Let $\alpha_0<\alpha\le\alpha_{2n}$
  and let~$n$ be sufficiently large. 
  Let~$G$ and~$H$ be as in Setting~\ref{set:RandEmb}. 
  Let $x\in V(G[{\scriptstyle [n-\mu n]}])$ and $u,v\in V(H)$ with $u\neq v$. When we run
  \RandomEmbedding{} to construct an embedding~$\psi_{n-\lfloor\mu n\rfloor}$ of the first $n-\lfloor\mu n\rfloor$ vertices of~$G$ into~$H$, then
  \[\Prob\big[x\AlgMap v \text{ and } u\not\in\im\psi_{n-\lfloor\mu n\rfloor}\big]=(1\pm10^3C\alpha
    D\delta^{-1})\frac\mu n\,.\]  
\end{lemma}
\begin{proof}
 Let $y_1,\dots,y_d$ with $d\le D$ be the neighbours of~$x$ in $\LNBH_G(x)$
  in degeneracy order, and (for convenience) define $y_0=0$. We define a collection of events. Let $\histens'_{0}$ be the almost sure event. For each $1\le i\le d$, let $\histens_i$ be the intersection of $\histens'_{i-1}$ and the event that neither $u$ nor $v$ is in the image of $\psi_{y_i-1}$, and let $\histens'_i$ be the intersection of $\histens_i$ and the event that $y_i$ is embedded to a vertex of $N_H(v)\setminus\{u\}$. Let $\histens_{d+1}$ be the intersection of $\histens'_d$ and the event that neither $u$ nor $v$ is in the image of $\psi_{x-1}$. Let $\histens'_{d+1}$ be the intersection of $\histens_{d+1}$ and the event $x\AlgMap v$. And finally let $\histens_{d+2}$ be the intersection of $\histens'_{d+1}$ and the event that $u\not\in\im\psi_{n-\lfloor\mu n\rfloor}$. Note that all of these events are history ensembles up to some given time.
  
  Now what we want to do is estimate $\Prob[\histens_{d+2}]$, and the reason for giving this collection of events is that we can estimate each of the successive conditional probabilities. We can estimate $\Prob[\histens_i|\histens'_{i-1}]$ for each $1\le i\le d+2$ using Lemma~\ref{lem:probvtx} (using part~\ref{lem:probvtx:b} for $1\le i\le d+1$ and part~\ref{lem:probvtx:a} for the final part). And we can estimate $\Prob[\histens'_i|\histens_i]$ using the diet condition for each $1\le i\le d+1$; the probability that the diet condition fails is tiny. To justify both of these steps we need to know $\Prob[\histens_i],\Prob[\histens'_i]>n^{-4}$; this is (by induction) valid since the final $\histens_{d+2}$ is the smallest event and we will argue its probability satisfies this bound. Assuming this bound for a moment, by Lemma~\ref{lem:probvtx}, for each $1\le i\le d$ we have
  \begin{multline*}
   \Prob[\histens_i|\histens'_{i-1}]=(1\pm 100C\alpha\delta^{-1})\big(\tfrac{n-y_i}{n-y_{i-1}}\big)^2\,,\quad\Prob[\histens_{d+1}|\histens'_d]=(1\pm 100C\alpha\delta^{-1})\big(\tfrac{n-x}{n-y_{d}}\big)^2\\
   \text{and}\quad \Prob[\histens_{d+2}|\histens'_{d+1}]=(1\pm 100C\alpha\delta^{-1})\tfrac{\lfloor\mu n\rfloor-1}{n-x}\,.
  \end{multline*}
  
  For each $1\le i\le d$, we have
  \[\Prob[\histens'_i|\histens_i]=\frac{(1\pm C\alpha)p^{\LEFTDEG_G(y_i)+1}(n-y_i+1)\pm 1}{(1\pm C\alpha)p^{\LEFTDEG_G(y_i)}(n-y_i+1)}\pm 4n^{-5}=(1\pm 4C\alpha)p\,.\]
  The fraction in the first term assumes the $(C\alpha,2D+3)$-diet condition,
  for the vertices $\psi_{y_i-1}\big(\LNBH_G(y_i)\big)\cup\{v\}$ in the
  numerator and $\psi_{y_i-1}\big(\LNBH_G(y_i)\big)$ in the denominator, to
  estimate respectively the number of neighbours of $v$ in the candidate set of
  $y_i$ which are not in $\im\psi_{y_i-1}$ and the number of vertices in the
  candidate set of $y_i$ which are not covered by $\im\psi_{y_i-1}$. The $\pm1$
  term in the numerator covers the possibility $u\in N_H(v)$. The $4n^{-5}$
  error term covers  the possibility of failure of the
  diet condition: By Lemma~\ref{lem:22} the probability that the diet condition
  fails is at most $2n^{-9}$, hence since $\Prob[\histens_i]>n^{-4}$ the
  probability that the diet condition fails conditioned on $\histens_i$ is at
  most $2n^{-5}$. By similar logic, we have
  \[\Prob[\histens'_{d+1}|\histens_{d+1}]=\frac{1}{(1\pm C\alpha)p^{d}(n-x+1)}\pm 4n^{-5}=(1\pm 4C\alpha)\tfrac{1}{p^d(n-x+1)}\,.\]
  
  Multiplying together all these conditional probabilities, many terms cancel and we obtain
  \begin{align*}
   \Prob[\histens_{d+2}]&=(1\pm100C\alpha\delta^{-1})^{d+2}(1\pm4C\alpha)^{d+1}\frac{(n-x)(\lfloor\mu n\rfloor-1)}{n^2(n-x+1)}\\
   &=(1\pm 100C\alpha\delta^{-1})^{2D+4}\cdot\frac{\mu}{n}\,,
  \end{align*}
  which since $\alpha\le\alpha_{2n}$ and by choice of $\alpha_{2n}$ implies the desired bound.
\end{proof}

\subsection{Properties of \PackingProcess{}}
\label{subsec:PackingProcess}

The following lemma summarises some facts we obtain in the course of proving~\cite[Theorem~11]{DegPack}.

\begin{lemma}[\PackingProcess{} lemma]\label{lem:PackingProcess}
  Given~$D$, $\hat p$, $\gamma$, let $(\alpha_s)_{s\in[s^*]}$, $\eta$ and the graphs
  $(G''_s)_{s\in[s^*]}$, $\widehat H$ be as in Setting~\ref{set:graphs}.
  When \PackingProcess{} is run with input $(G''_s)_{s\in[s^*]}$ and
  $\widehat H$, with probability at least
  $1-2n^{-5}$, the following holds.
  \begin{enumerate}[label=\abc]
    \item\label{lem:PackingProcess:a} \PackingProcess{} succeeds in packing $(G''_s)_{s\in[s^*]}$
      into~$\widehat H$.
    \item\label{lem:PackingProcess:quasi} For each $s\in[s^*]$ the pair $(H_s,H^*_0)$ is $(\alpha_s,2D+3)$-coquasirandom.
    \item\label{lem:PackingProcess:c} The leftover graph~$H$ is $(\eta,2D+3)$-quasirandom.
    \item\label{lem:PackingProcess:d} $H_0^*$ has maximum degree at most $2\gamma n$.
  \end{enumerate}
\end{lemma}
\begin{proof}
 \ref{lem:PackingProcess:a} is obtained by summing the failure probabilities of all exceptional events in~\cite[Proof of Theorem~11]{DegPack}.

 \ref{lem:PackingProcess:quasi} holding is implied by the exceptional event (ii) of that proof not occurring.
 
 \ref{lem:PackingProcess:c} is implied by exceptional event (v) of~\cite[Proof of Theorem~11]{DegPack} not occurring. Again, event (v) not occurring states that $(H_{s^*},H^*_{s^*})$ is $(\eta,2D+3)$-quasirandom. We would like to know that this implies $H=H_{s^*}\cup H^*_{s^*}$ is $(\eta,2D+3)$-quasirandom. Since $H_{s^*}$ and $H^*_{s^*}$ are edge-disjoint, given any vertex set $S$ of size at most $2D+3$, the neighbours $N_H(S)$ are partitioned into parts indexed by the subsets $R$ of $S$, where a vertex $v$ is in the part indexed by $R$ if it is adjacent in $H_{s^*}$ to the vertices $R$ and in $H^*_{s^*}$ to the vertices $S\setminus R$. Now $(\eta,2D+3)$-coquasirandomness gives bounds on these part sizes with a $(1\pm\eta)$ relative error, and summing the bounds we obtain the desired $(\eta,2D+3)$-quasirandomness of $H$.
Indeed, by the argument above we obtain
\begin{align*}
N_H(S)
& =\sum_{R\subseteq S} (1\pm \eta)(p_{s*})^{|R|}
(p_{s*}^*)^{|S\setminus R|} n \\
& = (1\pm \eta) n \sum_{r=0}^{|S|} \binom{s}{r}(p_{s*})^{r}
(p_{s*}^*)^{|S|-r}
= (1\pm \eta)(p_{s*}+p_{s*}^*)^{|S|} n
\end{align*}
for every $S$ of size at most $2D+3$.
 
 \ref{lem:PackingProcess:d} is implied by exceptional event (i) not occurring: this event in particular implies that $H^*_0$ is $(\tfrac14\alpha_0,2D+3)$-quasirandom, which together with the fact $e(H^*_0)=(1\pm\alpha_0)\gamma\binom{n}{2}$ from~\cite[Lemma~16]{DegPack} implies the claimed maximum degree.
\end{proof}

We further need the following two lemmas. The first states that, while
running \PackingProcess{}, chosen subsets~$T$ of neighbourhoods of vertices shrink
roughly as expected. We will use this with $T$ being a vertex neighbourhood with the embedded image of one or two of the $G''_i$ removed. Recall that $p_s$ denotes
the density of $H_s$.

\begin{lemma}\label{lem:T}
  Assume Setting~\ref{set:graphs} and let $s^*-\lfloor\mu n\rfloor<s<s'\le s^*$. Consider the following experiment.
  Run \PackingProcess{} with input
  $(G''_{s''})_{{s''}\in[s^*]}$ and $\widehat H$ up to and including the
  embedding of~$G''_s$. Then fix $T\subset N_{H_s}(v)$ with
  $|T|\ge\frac12 p\mu^2 n$, and continue \PackingProcess{} to perform the embedding
  of $G''_{s+1},\dots,G''_{s'}$.
  
  The probability that \PackingProcess\ fails before embedding $G''_{s'}$, or $H_i$ fails to be $(\alpha_i,2D+3)$-quasirandom for some $1\le i\le s'$, or we have \[\big|T\cap N_{H_{s'}}(v)\big|=(1\pm \gamma^{-1}\alpha_{s'})\frac{p_{s'}}{p_s}|T|\,,\]
  is at least $1-n^{-C}$.
\end{lemma}
\begin{proof}
 For $s\le i\le s'$, we define the event~$\cE_i$ that \PackingProcess\ does not fail before embedding $G''_i$, and~$H_j$ is $(\alpha_j,2D+3)$-quasirandom for each $1\le j\le i$,
  and $|T\cap N_{H_j}(v)|=(1\pm \gamma^{-1}\alpha_j)\frac{p_j}{p_s}|T|$ for each $s\le j\le i$. If the event in the lemma statement fails to occur, then there must exist some $s\le i< s'$ such that $\cE_i$ occurs and
  \[|T\cap N_{H_{i+1}}(v)|\neq(1\pm \gamma^{-1}\alpha_{i+1})\frac{p_{i+1}}{p_s}|T|\,.\]
 It suffices to show that each of these bad events occurs with probability at most $n^{-C-1}$, since then the union bound over the at most $\mu n$ choices of $i$ gives the lemma statement. This is an estimate we can obtain using Corollary~\ref{cor:freedm}. We now fix $s\le i<s'$ and prove the desired estimate.
 
 Suppose $s\le j\le i$, and let $Y_j:=|N_{H_{j}}(v)\cap T\setminus N_{H_{j+1}}(v)|$ be the number of edges from $v$ to $T$ used for the embedding $G_{j+1}$. Then we have $\big|T\cap N_{H_{i+1}}(v)\big|=|T|-\sum_{j=s}^iY_j$, and what we want to do is argue that the sum of random variables is concentrated. To that end, suppose $\hist$ is a history of \PackingProcess{} up to time $j$ such that $H_j$ is $(\alpha_j,2D+3)$-quasirandom and $|T\cap N_{H_j}(v)|=(1\pm \gamma^{-1}\alpha_j)\frac{p_j}{p_s}|T|$. Then we have
 \[\Exp\big[Y_j\,\big|\,\hist\big]=\frac{2e(G''_{j+1})\cdot(1\pm 500C\alpha_j\delta^{-1})^{4D+2}}{p_jn^2}\cdot(1\pm \gamma^{-1}\alpha_j)\frac{p_j}{p_s}|T|\]
  where we use linearity of expectation: the first factor is by Lemma~\ref{lem:probedge} the probability that a given edge from $v$ to $T$ in $H_j$ is used in the embedding of $G''_j$, and the second factor is the number of such edges. Note that the $p_j$ terms cancel, so we obtain
  \begin{align*}
   \Exp\big[Y_j\,\big|\,\hist\big]&=\frac{2e(G''_{j+1})\cdot(1\pm 500C\alpha_j\delta^{-1})^{4D+2}}{p_sn^2}\cdot(1\pm \gamma^{-1}\alpha_j)|T|\\
   &=\frac{2e(G''_{j+1})|T|}{p_s n(n-1)}\pm\frac{10^5\delta^{-1}CD^2|T|}{p_s
      n}\alpha_{j}\,,
  \end{align*}
  where for the error term we use the upper bound $e(G''_{j+1})\le Dn$ and our choice $\delta^{-1}>\gamma^{-1}$. Let
  \[\tilde{\mu}:=\sum_{j=s}^i\frac{2e(G''_{j+1})|T|}{p_s n(n-1)}\quad\text{and}\quad\tilde\nu:=\sum_{j=s}^i\frac{10^5\delta^{-1}CD^2|T|}{p_s
      n}\alpha_{j}\]
and observe that $\tilde{\mu}\leq |T|\leq n$ and
$\tilde{\nu} \leq \frac{10^5\delta^{-1}CD^2|T|}{p_s}\alpha_{i} < \frac{|T|}{10^3}$ since $p_s\geq \gamma$ and
by the definition of $\alpha_j$. 
  
  We trivially have $0\le Y_j\le \Delta(G''_{j+1})\le  cn/\log n$. So what Corollary~\ref{cor:freedm}\ref{cor:freedm:tails}, with $\tilde{\rho}=\eps n$, gives us is that
  \[\Prob\left[\cE_i\text{ and }\sum_{j=s}^iY_i\neq \tilde{\mu}\pm(\tilde\nu+\eps n)\right]<2\exp\left(-\tfrac{\eps^2n^2}{4cn^2/\log n}\right)<n^{-C-1}\,,\]
  where we use the upper bound $\tilde{\mu}+\tilde{\nu}+\tilde{\rho}\le 2n$ for the first inequality and the choice of $c$ as well as $\eps < \frac{1}{C}$ for the second. This is the probability bound we wanted. We now simply need to show that if
  \[\sum_{j=s}^iY_i=\tilde{\mu}\pm(\tilde\nu+\eps n)\]
  then we have
  \[|T\cap N_{H_{i+1}}(v)|=(1\pm \gamma^{-1}\alpha_{i+1})\frac{p_{i+1}}{p_s}|T|\,.\] 
  Since
  \[|T|-\tilde{\mu}=|T|\Big(1-\tfrac{\sum_{j=s}^ie(G''_{j+1})}{p_s\binom{n}{2}}\Big)=|T|\Big(1-\tfrac{(p_s-p_{i+1})\binom{n}{2}}{p_s\binom{n}{2}}\Big)=\frac{p_{i+1}}{p_s}|T|\,,\]
  what remains is to argue $\tilde\nu+\eps n<\gamma^{-1}\alpha_{i+1}\tfrac{p_{i+1}}{p_s}|T|$. Since $\alpha_j=\frac{\delta}{10^8CD}\exp\left(\frac{10^8CD^3\delta^{-1}(j-2n)}{n}\right)$ is increasing in $j$, we have
  \begin{align}
  \begin{split}
  \label{eq:sum:alpha}
    \sum_{j=s}^i\alpha_j&\le \int_s^{i+1}\alpha_j\,\mathrm{d}j\le\int_{-\infty}^{i+1}\alpha_j\,\mathrm{d}j\\
    &=\Big[\frac{\delta}{10^8CD}\cdot\frac{n}{10^8CD^3\delta^{-1}}\cdot\exp\Big(\frac{10^8CD^3\delta^{-1}(j-2n)}{n}\Big)\Big]_{j=-\infty}^{i+1}
    =\frac{\delta n}{10^8CD^3}\alpha_{i+1}\,.
   \end{split}
  \end{align}
  It follows that
  \[\tilde\nu+\eps n\le \frac{10^5\delta^{-1}CD^2|T|}{p_s n}\cdot \frac{\delta n}{10^8CD^3}\alpha_{i+1}+\eps n\le\tfrac{\alpha_{i+1}}{1000D}\cdot\tfrac{1}{p_s}|T|+\eps n\,.\]
  Finally, since $p_{i+1},p\ge\gamma$, by choice of $\eps$, since $\delta\leq \mu$ and because $|T|\ge\tfrac12 p\mu^2n$, we conclude $\tilde\nu+\eps n\le\gamma^{-1}\alpha_{i+1}\tfrac{p_{i+1}}{p_s}|T|$ as desired.
\end{proof}

The second lemma states that for a set~$S$ of host graph vertices fixed
before the embedding of $G''_s$, it
is likely that the embedding of $G''_s$ (which has $n-\lfloor\mu n\rfloor$ vertices) uses about $(1-\mu)|S|$ vertices of $S$. To prove it, we repeatedly apply Lemma~\ref{lem:24}, which tells us that it is likely that each successive $\eps n$ vertices of $G''_s$ embedded cover about the expected fraction of $S$.

\begin{lemma}\label{lem:S}
  Assume Setting~\ref{set:graphs} and let $s^*-\lfloor\mu n\rfloor<s\le s^*$.  Run \PackingProcess{} with input
  $(G''_{s''})_{{s''}\in[s^*]}$ and $\widehat H$ up to just before the
  embedding of~$G''_s$.
  Then fix any $S\subset V(H_{s-1})$ with $|S|\ge\frac12p\mu^2 n$, and
  let \PackingProcess{} perform the embedding of~$G''_s$. With
  probability at least $1-3n^{-9}$ either $H_{s-1}$ is not $(\alpha_{s-1},2D+3)$-quasirandom or
  \[|S\setminus\im\phi'_{s}|=(1\pm C'\alpha_{s})\mu|S|\,.\]
\end{lemma}
\begin{proof}
  Fix~$s$ such that $s^*-\lfloor\mu n\rfloor<s\le s^*$, and condition on $H_{s-1}$. If $H_{s-1}$ is not $(\alpha_{s-1},2D+3)$-quasirandom, then the bad event of this lemma cannot occur. So it suffices to show that if $H_{s-1}$ is $(\alpha_{s-1},2D+3)$-quasirandom, then the probability of the event $|S\setminus\im\phi'_{s}|\neq(1\pm C'\alpha_{s})\mu|S|$, conditioned on $H_{s-1}$, is at most $3n^{-9}$. This is what we will now do, so we suppose that $H_{s-1}$ is $(\alpha_{s-1},2D+3)$-quasirandom. Consider the run of \RandomEmbedding{} which embeds $G''_s[{\scriptstyle [n-\delta n]}]$. 
  
  Recall that the embedding $\phi'_s$ of~$G'_s$ is given by letting \RandomEmbedding{} perform the embedding
  of~$G''_s[{\scriptstyle [n-\delta n]}]$, constructing the partial embeddings $\psi_t$ for
  $0\le t\le(1-\delta)n$. More precisely, $\phi'_s$ is given by ignoring the
  embedding of all vertices not in~$G'_s$, that is, by $\psi_{n-\lfloor \mu n
    \rfloor}$.

Define $S_0=S$, and for $i=1,\dots,\tau$ with
  $\tau=\lceil\frac{(1-\mu)}{\eps}\rceil$ set
  $S_i=S_{i-1}\setminus \im \psi_{i\eps n}$. Since $S_\tau\subseteq S\setminus\im\phi'_{s}\subseteq S_{\tau-1}$, it is enough to show both $|S_{\tau-1}|$ and $|S_\tau|$ are likely to be in the claimed range. Since the two quantities differ by at most $\eps n$, we will focus on estimating $|S_\tau|$.
  In this proof we will always use $\alpha=\alpha_s$, and hence will
  often omit the parameter~$\alpha$ in $\beta_t=\beta_t(\alpha)$. 
  By Lemma~\ref{lem:24} (applied with $t=j+\eps n +1$), with probability at
  least $1-n^{-2D-18}$, either for some $j\le n-\mu n-\eps n$
  \begin{enumerate}[label=\abc]
    \item \RandomEmbedding{} failed to
      construct $\psi_j$, or
    \item the partial embedding $\psi_{j+\eps n +1}$ 
    of $G_s''[{\scriptstyle [j+\eps n +1]}]$ into $H_{s-1}$
    does not have the
      $(\eps,20D\beta_j,j)$-cover condition,
  \end{enumerate}
  or we have that for every $j\le n-\mu n-\eps n$
 \begin{enumerate}[label=\abc,start=3]
    \item\label{lem:S:size} 
    $
    \left| \left\{x:~ j\leq x<j+\eps n:~
    		\psi_{j+\eps n}(x) \in S\setminus \im \psi_j 
    		\right\}\right| =
    		(1\pm 40D\beta_{j})\frac{|S\setminus \im \psi_j|\eps
      n}{n-j}~ .
    $
  \end{enumerate} 
  
  By Lemma~\ref{lem:22}, with probability at least $1-2n^{-9}$, the first
  two options do not hold, so with probability at least $1-3n^{-9}$ we have
  that~\ref{lem:S:size} holds for every $j\leq n-\mu n -\eps n$. Applying~\ref{lem:S:size} with $j=(i-1)\eps n$ we conclude
  $$|S_i|=|S_{i-1}|-(1\pm 40D\beta_{(i-1)\eps n})\frac{|S_{i-1}|\eps
      n}{n-(i-1)\eps n}$$ 
for all $i\ge 1$.

  Assuming this is the case, we get
  \[|S_i|=|S_{i-1}|\Big(1-\frac{(1\pm 40D\beta_{(i-1)\eps n})\eps
      }{1-(i-1)\eps}\Big)\,,\]
  and hence
  \[|S_\tau|=|S|\prod_{i=1}^{\tau}\Big(1-\frac{(1\pm 40D\beta_{(i-1)\eps n})\eps
      }{1-(i-1)\eps}\Big)\,.\]
  In order to evaluate this product, observe that
  \[1-\frac{(1\pm 40D\beta_{i\eps n})\eps}{1-i\eps}
    =\frac{1-(i+1)\eps}{1-i\eps}\pm\frac{40 D\beta_{i\eps n}\eps}{1-i\eps}
    = \frac{1-i\eps-\eps}{1-i\eps}\Big(1\pm\frac{40D\beta_{i\eps
        n}\eps }{1-(i+1)\eps}\Big)\,,
  \]
  and therefore
  \[|S_\tau|=|S|\prod_{i=0}^{\tau-1}\frac{1-i\eps-\eps}{1-i\eps}
    \cdot\Big(1\pm\frac{40D\beta_{i\eps n}\eps}{1-(i+1)\eps}\Big)
    =|S|(1-\tau\eps) \prod_{i=0}^{\tau-1}\Big(1\pm\frac{40D\beta_{i\eps n}\eps}{1-(i+1)\eps}\Big)\,.
  \]
  
By the definition of $\tau$ we have
  $\frac{(1-\mu)}{\eps}\le\tau\le\frac{(1-\mu)}{\eps}+1$ and hence $(1-\tau\eps)=\mu(1\pm\frac\eps\mu)$.
Moreover, we obtain that
\begin{align*}
	\sum_{i=0}^{\tau-1}\frac{40D\beta_{i\eps n}
		\eps}{1-(i+1)\eps}
	& \leq \frac{40D\eps }{1-\tau \eps} 
	\sum_{i=0}^{\tau -1} \beta_{i\eps n}
	\leq \frac{80D\eps }{\mu } 
	\sum_{i=0}^{\tau -1} \beta_{i\eps n} \\
	&\le\frac{80D}{\mu n}\int_{0}^{\tau}\eps n\beta_{i\eps n}\,\mathrm{d}i
    \le\frac{80D}{\mu n}\int_{0}^{\tau\eps n}\beta_{x}\,\mathrm{d}x \\
    &\leByRef{eq:betabound}
    \frac{80D}{\mu\cdot 1000
      D\delta^{-2}\gamma^{-2D-10}}\beta_{\tau \eps n}
    \le\beta_{(1-\mu + \eps )n}
    \le\frac12 C'\alpha=\frac12 C'\alpha_s
\end{align*}  
   since
  $\beta_{(1-\mu+\eps)n}=\beta_{(1-\mu+\eps)n}(\alpha)=2\alpha\exp(1000D\delta^{-2}\gamma^{-2D-10}(1-\mu+\eps))<2\alpha $
  and \[C'=10^4\cdot\frac{40D}{\delta}\exp(1000D\delta^{-2}\gamma^{-2D-10})\,.\]

So, since $\prod_i(1\pm x_i)=1\pm2\sum_i x_i$
  as long as $\sum_i x_i < \frac{1}{100}$
  and since $\frac{1}{2}C'\alpha_s < \frac{1}{100}$, 
  we get
  \begin{equation*}\begin{split}
      |S_\tau|&=|S|(1-\tau\eps)\Big(1\pm2\sum_{i=0}^{\tau-1}\frac{40D\beta_{i\eps
          n}\eps}{1-(i+1)\eps}\Big)
      =|S|(1-\tau\eps)\Big(1\pm\frac{80D\eps}{1-\tau\eps}\sum_{i=0}^{\tau-1}\beta_{i\eps
        n}\Big) \\
      &=|S|\Big(1-\tau\eps\pm 80D\eps\sum_{i=0}^{\tau-1}\beta_{i\eps
        n}\Big)
      =|S|\mu\Big(1\pm\frac{\eps}{\mu} \pm\frac{80D\eps}{\mu}\sum_{i=0}^{\tau-1}\beta_{i\eps
        n}\Big) \\
      &=|S|\mu\Big(1\pm\frac12\alpha_s \pm
      \frac{1}{2}C'\alpha_s\Big)\,,
  \end{split}\end{equation*}
  where for the last equation we use that $\eps\le\alpha_0\delta^2\gamma\le\frac12\alpha_s\mu$. It follows that  
  \[|S\setminus\im\phi'_s|=
  |S_{\tau}|\pm \eps n = 
  |S|\mu\big(1\pm\tfrac12\alpha_s\pm\tfrac12C'\alpha_s\big)\pm\eps n=\big(1\pm C'\alpha_s\big)\mu|S|\,,\]
  as desired.
\end{proof}

\subsection{Proof of Lemma~\ref{lem:appl}}
\label{subsec:appl_proof}

We now have all tools at hand to prove the almost perfect packing lemma.

\begin{proof}[Proof of Lemma~\ref{lem:appl}] 
  For $0\le s<s^*$ let~$\cE_{s}$ be the event that~$H_s$ is
  $(\alpha_{s},2D+3)$-quasirandom. By
  Lemma~\ref{lem:PackingProcess}\ref{lem:PackingProcess:quasi} we have
  \begin{equation}\label{eq:appl:bigcapE}
    \Prob\Big[\bigcap_s\cE_s\Big]\ge 1-2n^{-5}\,.
  \end{equation}
  Let $\hist_{s}$ be an embedding of $G''_1,\dots,G''_{s}$ by \PackingProcess{}
  such that $\cE_{s}$ holds.

  Recall that we may assume that $e(H_0^*)\leq 1.1\gamma \binom{n}{2}$ holds,
  which is fine as the probability of this inequality not being satisfied is at
  most $e^{-n}$. So, from now on we always condition on this assumption, and we
  shall show in the following that then each of the properties
  \ref{appl:quasi}--\ref{appl:sumweight} holds with probability at
  least~$1-n^{-4}$, which gives the lemma.

  \medskip
     
  \noindent
  \underline{\ref{appl:quasi}:
  $H$ is $({\gamma'}^3,2D+3)$-quasirandom and has density~$p$}.

  \medskip

  By Lemma~\ref{lem:PackingProcess}\ref{lem:PackingProcess:c}, the leftover
  graph~$H$ is $(\eta,2D+3)$-quasirandom with probability at least $1-2n^{-5}$.
  By~\eqref{eq:defconsts} and since $\gamma\ll \gamma'$ we
  have
  $\eta\le{\gamma'}^3$, which gives~\ref{appl:quasi}.
  
  \medskip

  \noindent
  \underline{\ref{appl:weight}:
  $w(v)=(1\pm{\gamma'}^3)\frac{pn}{2}$}.

  \medskip

  Fix $v\in V(H)$ and let $Y_s=w_s(x)\ONE_{x\AlgMap v}$. We have $Y_s\le\Delta$ and
  \[w(v)=\sum_s w_s(v)=\sum_{s,x\in V(G_s)} Y_s\,.\]
  We want to apply Corollary~\ref{cor:freedm}. By Lemma~\ref{lem:vertex} we have
  \begin{equation*}\begin{split}
      \sum_{s\in[s^*]} \Exp[Y_s | \hist_{s-1}]=
      \sum_{\substack{s\in[s^*] \\x\in V(G_s)}} w_s(x)\Prob[x\AlgMap v| \hist_{s-1}]=
      \sum_{\substack{s\in[s^*] \\x\in V(G_s)}} w_s(x) (1\pm10^4 C\alpha_s D\delta^{-1})\frac1n\,.
  \end{split}\end{equation*}
  It follows that
  \begin{equation*}\begin{split}
      \sum_{s\in[s^*]} \Exp[Y_s | \hist_{s-1}]=
      p\binom{n}{2}(1\pm10^4 C\alpha_{s^*} D\delta^{-1})\frac1n=
      \frac{pn}{2}(1\pm2\cdot 10^4 C\alpha_{s^*} D\delta^{-1})\,.
  \end{split}\end{equation*}
  By the second part of Corollary~\ref{cor:freedm}\ref{cor:freedm:tails}
  applied with $\cE=\bigcap_s\cE_s$, $\tmu=\frac{pn}{2}$, $\teta=2\cdot
  10^4 C\alpha_{s^*} D\delta^{-1}$ we obtain
  \begin{equation}\begin{split}
    \label{eq:appl:P2}
    \Prob\big[\cE\text{ and }\sum_s Y_s\neq \frac{pn}{2}\cdot(1\pm
    4\cdot 10^4C\alpha_{s^*} D\delta^{-1})\big]
    & \le 2\exp\Big(-\frac{\tmu \cdot 4\cdot 10^8C^2\alpha_{s^*}^2
      D^2\delta^{-2}}{4\Delta}\Big) \\
    & \le 2\exp(-10^{10}\log n)\,,
  \end{split}\end{equation}
  where the last inequality uses $\Delta\le cn/\log n$, $c\le
  10^{-10}\gamma^{10D}\alpha_0^4$, $\alpha_0\le\alpha_{s^*}$, $p\ge\mu\nu$ and $\gamma\ll\nu\le\mu$.

  We have $s^*\le\frac74 n$ and hence by the definition of $\alpha_x$ and of~$C$
  in~\eqref{eq:defconsts} we get
  \begin{equation}\label{eq:appl:alphadelta}
    \begin{split}
    4\cdot 10^4C\alpha_{s^*} D\delta^{-1}
  & \le 4\cdot 10^4C\alpha_{\frac74n} D\delta^{-1} \\
  & = 4\cdot 10^4C\cdot
  \frac{\delta}{10^8CD}\exp\big(-10^8CD^3\delta^{-1}\cdot\tfrac14 \big)
  \cdot D\delta^{-1} \\
  & \le \exp\big(-10^7CD^3\delta^{-1}\big) 
  \le \exp\big(-10^7\cdot 40D\exp(1000D\delta^{-2}\gamma^{-2D-10})\big) \\
  & \le \exp\big(-\exp(\gamma^{-2D-10})\big)
  \le\gamma^3\le{\gamma'}^3\,.
  \end{split}\end{equation}
  Combining this with~\eqref{eq:appl:P2} and~\eqref{eq:appl:bigcapE} and a union
  bound over~$v$, we
  conclude that~\ref{appl:weight} fails with probability at most
  $2n^{-5}+n\cdot n^{-10}\le n^{-4}$.

  \medskip

  \noindent
  \underline{\ref{appl:3}:
  $\big|N_H(v)\setminus \im\phi'_s\big|=(1\pm{\gamma'}^3)\mu pn$}
  and

  \noindent
  \underline{
    \ref{appl:4}:
    $\big|N_H(v)\setminus
    (\im\phi'_s\cup\im\phi'_{s'})\big|=(1\pm{\gamma'}^3)\mu^2pn$ if $s\neq s'$}.

  \medskip 

  We prove these together.
  Fix $v\in V(H)$ and $s,s'$ with $s^*-\lfloor\mu n\rfloor<s<s'\le s^*+1$.
  The artificial case $s'=s^*+1$ will be used to prove~\ref{appl:3}.

  We first run \PackingProcess{} up to time $s-1$ and consider the
  embedding of~$G''_s$. We want to apply Lemma~\ref{lem:S} to estimate what
  happens in this first stage.
  We set $S=N_{H_{s-1}}(v)$, so if $\cE_{s-1}$ holds then
  $|S|=(1\pm\alpha_{s-1})p_{s-1}n\ge\frac12pn\ge\frac12p\mu^2 n$.
  Hence we can apply Lemma~\ref{lem:S} with $S$ and
  conclude that with probability at least $1-3n^{-9}$ either $\cE_{s-1}$
  does not hold or
  \begin{equation}\label{eq:appl:Nv1}
    |N_{H_{s-1}}(v)\setminus\im \phi'_s|
    =(1\pm C'\alpha_s)\mu|S|
    =(1\pm 3C'\alpha_s)p_{s-1}\mu n\,.
  \end{equation}
  Further, we have $N_{H_{s}}(v)\setminus\im
  \phi'_s=N_{H_{s-1}}(v)\setminus\im \phi'_s$.

  Now let \PackingProcess{} perform the embeddings of $G''_{s+1},\dots,G''_{s'-1}$. We want to
  apply Lemma~\ref{lem:T} to estimate what happens in this second stage.
  Set $T=N_{H_s}(v)\setminus\im\phi'_s$ and observe that $T\cap
  N_{H_{s'-1}}(v)=N_{H_{s'-1}}(v)\setminus\im\phi'_s$.
  If \eqref{eq:appl:Nv1} holds, then $|T|\ge\frac12p\mu^2 n$ because
  by~\eqref{eq:defconsts} we have $C'\alpha_s\le 10^{-4}$. So
  by Lemma~\ref{lem:T} applied with~$T$ we get that with probability at
  least $1-n^{-C}$ either $\bigcap_{i} \cE_i$ fails or we have
  \begin{equation}\begin{split}\label{eq:appl:Nv2}
    |N_{H_{s'-1}}(v)\setminus\im\phi'_s| & =(1\pm \gamma^{-1}\alpha_{s'-1})\frac{p_{s'-1}}{p_s}|T|
    \eqByRef{eq:appl:Nv1} (1\pm \gamma^{-1}\alpha_{s'-1})\frac{p_{s'-1}}{p_{s}} (1\pm
    3C'\alpha_s)p_{s-1} \mu n \\
    & = (1\pm 5C'\alpha_{s'-1}) \mu p_{s'-1} n 
  \end{split}\end{equation}
  where the last equality follows from $\frac{p_{s-1}}{p_s}=1+o(1)$ and since $\gamma^{-1}<C'$.
  For the case $s'=s^*+1$ this immediately implies~\ref{appl:3}.
  Indeed, in this case~\eqref{eq:appl:Nv2} gets 
  \begin{equation*}\begin{split}
    |N_{H_{s^*}}(v)\setminus\im\phi'_s|= (1\pm 5C'\alpha_{s^*}) \mu p_{s^*} n~ . 
  \end{split}\end{equation*}
As long as $\Delta(H^*_{s^*})\leq \Delta(H^*_0)\leq 2\gamma n$, which holds with probability at least
 $1-2n^{-5}$ according to
 Lemma~\ref{lem:PackingProcess}\ref{lem:PackingProcess:d},
 we have that $|N_{H}(v)|-|N_{H_{s^*}}(v)|\leq 2\gamma n$ and $p_{s^*}=p\pm 2\gamma=\left(1\pm \frac{2\gamma}{p} \right)p$
 from which we conclude that
  \begin{align*}
    |N_{H}(v)\setminus\im\phi'_s|= (1\pm 5C'\alpha_{s^*}) \left(1\pm \frac{2\gamma}{p} \right)\mu p n \pm 2\gamma n 
    = (1\pm{\gamma'}^3)\mu pn
  \end{align*}
  since $C'\alpha_{s^*}<\frac{1}{100}\gamma'^3$, since
  $\gamma \ll \gamma'\ll \nu\ll \mu$ and $p\geq \mu\nu$.
  Hence, in total, taking a union bound over~$v$ and~$s$ and using~\eqref{eq:appl:bigcapE},
  the probability that~\ref{appl:3} fails is at most
  $4n^{-5}+n^2(3n^{-9}+n^{-C})\le n^{-4}$.

  For proving~\ref{appl:4}, assume that $s'\le s^*$ and consider next the
  embedding of~$G''_{s'}$ by \PackingProcess{}. We again want to apply Lemma~\ref{lem:S}, this
  time with $S=N_{H_{s'-1}}(v)\setminus\im\phi'_s$. If~\eqref{eq:appl:Nv2}
  holds, then
  $|S|=(1\pm 5C'\alpha_{s-1})\mu p_{s-1}n\ge\frac12p\mu^2 n$.
  Hence we can apply Lemma~\ref{lem:S} with $S$ and with~$s'$ in place
  of~$s$ to
  conclude that with probability at least $1-3n^{-9}$
  either~\eqref{eq:appl:Nv2} fails, or $\cE_{s'-1}$ fails or
  \begin{equation}\begin{split}\label{eq:appl:Nv3}
    |N_{H_{s'}}(v)\setminus(\im\phi'_s\cup\im\phi'_{s'})|& 
    =|N_{H_{s'-1}}(v)\setminus(\im\phi'_s\cup\im\phi'_{s'})| \\
    & =(1\pm C'\alpha_{s'})\mu|S|
    =(1\pm 7C'\alpha_{s'})p_{s'-1}\mu^2 n\,.
  \end{split}\end{equation}

  In a last stage, consider the embedding of $G''_{s'+1},\dots,G''_{s^*}$
  by \PackingProcess{}. We apply Lemma~\ref{lem:T} with 
  $T=N_{H_{s'}}(v)\setminus(\im\phi'_s\cup\im\phi'_{s'})\subset N_{H_{s'}}(v)$ and with~$s'$
  replaced by~$s^*$, which is possible if~\eqref{eq:appl:Nv3} holds since
  then $|T|\ge\frac12 p\mu^2n$. In this case, because $T\cap
  N_{H_{s^*}}(v)=N_{H_{s^*}}(v)\setminus(\im\phi'_s\cup\im \phi'_{s'})$,  we conclude that with
  probability at least $1-n^{-C}$ we have 
  \begin{equation*}\begin{split}
    |N_{H_{s*}}(v)\setminus(\im\phi'_s\cup \im \phi'_{s'})| & =(1\pm \gamma^{-1}\alpha_{s^*})\frac{p_{s^*}}{p_{s'}}|T| \\
 &   \eqByRef{eq:appl:Nv3} (1\pm \gamma^{-1}\alpha_{s^*})\frac{p_{s^*}}{p_{s'}} (1\pm
    7C'\alpha_{s'})p_{s'-1} \mu^2 n 
     = (1\pm 9C'\alpha_{s^*}) \mu^2 p_{s^*} n
  \end{split}\end{equation*}
  from which we obtain
  \begin{equation*}\begin{split}
    |N_{H}(v)\setminus(\im\phi'_s\cup \im \phi'_{s'})| &      = (1\pm \gamma'^3) \mu^2 p n
  \end{split}\end{equation*}
  analogously to the discussion of~\ref{appl:3} and
  as long as $\Delta(H_0)\leq 2\gamma n$.
  We conclude, using a union bound over~$v$,
  $s$ and~$s'$ and again~\eqref{eq:appl:bigcapE} and
  Lemma~\ref{lem:PackingProcess}\ref{lem:PackingProcess:d},
  that~\ref{appl:4} fails with probability at most
  $4n^{-5}+n^3(2\cdot 3n^{-9}+2\cdot n^{-C})\le n^{-4}$.

  \medskip

  \noindent 
  \underline{
  \ref{appl:degN}:
  $\sum_s w_s(v)\ONE_{u\not\in\im\phi'_s}=(1\pm{\gamma'}^3)\mu\frac{pn}{2}$}.

  \medskip

  Fix~$u$ and $v\neq u$ and define
  \[ Y_s=w_s(v)\ONE_{u\not\in\im\phi'_s}\,,\]
  and observe that $Y_s\le w_s(v)\le\Delta$. 
  Again, we want to apply Corollary~\ref{cor:freedm}.
  We have
  \[\Exp[Y_s|\hist_{s-1}]=\sum_{x\in V(G_s)}w_s(x)\cdot\Prob[x\AlgMap v,u\not\in\im\phi_s'|\hist_{s-1}]\,.\]
  By Lemma~\ref{lem:vertexnotimage} we obtain
  \[\Exp[Y_s|\hist_{s-1}]=
  \sum_{x\in V(G_s)} w_s(x)\cdot(1\pm 10^3C\alpha_{s-1}
    D\delta^{-1})\frac{\mu}{n}
    =\lfloor\nu n\rfloor \cdot(1\pm 10^3C\alpha_{s-1} D\delta^{-1})\frac{\mu}{n} \,. \]
  This implies
  \[\sum_s \Exp[Y_s|\hist_{s-1}]=\lfloor\mu n\rfloor\lfloor\nu n\rfloor
    \cdot(1\pm 10^3C\alpha_{s^*} D\delta^{-1})\frac{\mu}{n} 
  =\frac{\mu pn}{2}\cdot(1\pm 2\cdot 10^3C\alpha_{s^*} D\delta^{-1})\,. \]
  We apply the second part of Corollary~\ref{cor:freedm}\ref{cor:freedm:tails} with
  \[\cE=\bigcap_s\cE_s\,, \quad R=\Delta\,, \quad \tmu=\frac{\mu pn}{2}\,, \quad
  \teta=2\cdot 10^3C\alpha_{s^*} D\delta^{-1}\] and use
  $\teta\le\frac12$, which holds by definition of $\alpha_{s^*}$, to
  conclude that
  \begin{equation*}\begin{split}
    \Prob\big[\cE\text{ and }\sum_s Y_s\neq \frac{\mu pn}{2}\cdot(1\pm
    4\cdot 10^3C\alpha_{s^*} D\delta^{-1})\big]
    & \le 2\exp\Big(-\frac{\tmu \cdot4\cdot 10^6C^2\alpha_{s^*}^2
      D^2\delta^{-2}}{4\Delta}\Big) \\
    & \le 2\exp(-10^{10}\log n)\,,
  \end{split}\end{equation*}
  where the last inequality uses $\Delta\le cn/\log n$, $c\le
  10^{-10}\gamma^{10D}\alpha_0^4$, $\alpha_0\le\alpha_{s^*}$, $p\ge\mu\nu$ and $\gamma\ll\nu\ll\mu$.
  Combining this with~\eqref{eq:appl:alphadelta}
  and~\eqref{eq:appl:bigcapE} and using a union bound over all~$u$, $v$, we conclude that
  \ref{appl:degN} fails with probability at most $2n^{-5}+n^2\cdot
  n^{-10}\le n^{-4}$.

  \medskip 

  \noindent 
  \underline{
    \ref{appl:sumweight}: If $u\not\in\im\phi'_s$ then we have $\sum_{v\colon
      vu\in E(H)} w_s(v)<\frac{10p^2 n}{\mu}$}.

  \medskip

    The verification of this statement is the most complicated part of this
    proof. We fix $u\in V(H)$ and~$s$ with $s^*-\lfloor\mu n\rfloor<s\le s^*$.
    We shall show that either an unlikely event occurs, or the desired property
    holds when~$G'_s$ is embedded, and then continues to hold while the
    remaining guest graphs are embedded. The embeddings of these guest
    graphs~$G'_{s'}$
    is performed in the graphs $H_{s'}$ and we shall show that
    $\sum_{v\colon vu\in E(H_{s'})} w_s(v)$ stays concentrated.
    But since~\ref{appl:sumweight}
    concerns the whole graph~$H$, we additionally need to control the contribution of
    edges $vu$ in~$H_{s'}^*$, for which we can only provide an upper bound.
    More precisely, we shall establish the following claim. We will then, at the
    end of this proof, argue that this implies~\ref{appl:sumweight},

    \begin{claim}\label{cl:sumweight:main}
      Suppose $u\not\in\im\phi'_s$. Then with probability at least $1-4n^{-19}$
      either $(H_i,H^*_0)$ is not $(\alpha_i,2D+3)$-coquasirandom for some
      $i\in[s^*]$, or $\big(H_{i-1},\phi'_i\big([t]\big)\big)$ does not satisfy
      the $(C\alpha_{i-1},2D+3)$-diet condition for some $i\in[s^*]$ and
      $t\in[n-\delta n]$, or $\big(H_{i-1},H^*_0,\phi'_i\big([t]\big)\big)$ does
      not satisfy the $(2\eta,2D+3)$-codiet condition for some $i\in[s^*]$ and
      $t\in[n-\delta n]$, or for each $s\le s'\le s^*$ we have
     \begin{equation}\label{eq:sumweight:main}
      \sum_{v:vu\in E(H_{s'})}w_s(v)=\big(1\pm 10Cp^{-1}\alpha_{s'}\big)\tfrac{p_{s'}pn}{2\mu}\,,
     \end{equation}
     and
     \begin{equation}\label{eq:sumweight:Hssum}
      \sum_{v:vu\in E(H^*_{0})}w_s(v)\le \tfrac{\gamma pn}{\mu}\,.
     \end{equation}
    \end{claim}
    
    We will prove this claim in two steps. First (in
    Claim~\ref{cl:sumweight:start}), we establish that it is very likely
    that~\eqref{eq:sumweight:main} holds for $s'=s$ and
    that~\eqref{eq:sumweight:Hssum} holds. Then, based on
    Claims~\ref{cl:sumweight:Yexp} to~\ref{cl:sumweight:Zsum}, we show that it
    is unlikely that any given $s'>s$ is the smallest $s'$ for
    which~\eqref{eq:sumweight:main} fails. Taking the union bound over $s'$
    will complete the proof of the claim.

  Recall again that the embedding $\phi'_s$ of~$G'_s$ is given by letting \RandomEmbedding{} perform the embedding
  of~$G''_s[{\scriptstyle [n-\delta n]}]$,
  thus constructing partial embeddings $\psi_t$ for
  $0\le t\le(1-\delta)n$,
  and then ignoring the
  vertices that do not belong to~$G'_s$, i.e. the last 
  $\mu n - \delta n$ ones.
  
    \begin{claim}\label{cl:sumweight:start}
      Suppose $u\not\in \im \phi'_s$. Then with probability at least $1-n^{-20}$
      the pair $(H_{s-1},H^*_0)$ is not $(\alpha_{s-1},2D+3)$-coquasirandom, or
      $\big(H_{s-1},\phi'_s\big([t]\big)\big)$ does not satisfy the
      $(C\alpha_{s-1},2D+3)$-diet condition for some $t\in[n-\lfloor \mu n \rfloor]$,
      or
      $\big(H_{s-1},H^*_0,\phi'_s\big([t]\big)\big)$ does not satisfy the
      $(2\eta,2D+3)$-codiet condition for some $t\in[n-\lfloor \mu n \rfloor]$, or we have
     \[
      \sum_{v:vu\in E(H_{s})}w_s(v)=\big(1\pm 10Cp^{-1}\alpha_{s}\big)\tfrac{p_{s}pn}{2\mu}\quad\text{and}\quad\sum_{v:vu\in E(H^*_{0})}w_s(v)\le\tfrac{\gamma pn}{\mu}\,.
     \]
     
    \end{claim}
    \begin{claimproof} We begin by proving the concentration of $\sum_{v:vu\in E(H_{s})}w_s(v)$. For every $t\in
      [n]$, let $x_t$ be the $t$-th vertex of $G_s''$, let $\cE_t'$ be the event
      that $H_{s-1}$ is $(\alpha_{s-1},2D+3)$-quasirandom and
      $\big(H_{s-1},\phi'_s\big([t]\big)\big)$ satisfies the
      $(C\alpha_{s-1},2D+3)$-diet condition, and let $\hist'_t$ be a history up
      to and including the embedding of $x_t$ which satisfies $\cE'_{t}$. When
      \RandomEmbedding{} is run, for $t\in[n-\lfloor\mu n\rfloor]$ we obtain
      \[
        \Prob[x_t \AlgMap N_{H_{s-1}}(u) | \hist_{t-1} ]
        =
        \frac{(1\pm C\alpha_{s-1})p_{s-1}^{1+\deg_G^-(x_t)}(n-t+1)}{
          (1\pm C\alpha_{s-1})p_{s-1}^{\deg_G^-(x_t)}(n-t+1)}
        =
        (1\pm 3C\alpha_{s-1})p_{s-1}	
      \]
      where the first equality holds, since under assumption of the diet
      condition for $(H_{s-1},\phi_s'([t-1]))$ the candidate set for $x_t$ is of
      size $(1\pm C\alpha_{s-1})p_{s-1}^{\deg_G^-(x_t)}(n-t+1)$, and since
      $u\notin \im \phi_s'$ and thus there exist $(1\pm
      C\alpha_{s-1})p_{s-1}^{1+\deg_G^-(x_t)}(n-t+1)$ candidates among
      $N_{H_{s-1}}(u)$. Now, set
      \[
        X_t:=w_s(x_t) \cdot \ONE_{x_t \AlgMap N_{H_{s-1}}(u)}
      \]
      so that
      \[
        \sum_{v:vu\in E(H_{s})}w_s(v)=
        \sum_{v:vu\in E(H_{s-1})}w_s(v)=
        \sum_{t\in [n]} X_t~ ,
      \]
      where the first equation holds because of $u\notin \im \phi'_s$. In order
      to apply Corollary~\ref{cor:freedm}\ref{cor:freedm:tails} observe that
      $0\leq X_t\leq \Delta$. Moreover,
      \begin{equation*}\begin{split}
        \sum_{t\in [n]} \Exp[X_t|\hist'_{t-1}]
        &= \sum_{t\in [n-\lfloor\mu n\rfloor]} \Exp[X_t|\hist'_{t-1}]
         = \sum_{t\in [n-\lfloor\mu n\rfloor]} w_s(x_t) \Prob[x_t \AlgMap N_{H_{s-1}}(u) | \hist'_{t-1} ] \\
        & = (1\pm 3C\alpha_{s-1})p_{s-1} 
          \sum_{t\in [n-\lfloor\mu n\rfloor]} w_s(x_t) 
          = (1\pm 4C\alpha_{s-1}) \frac{p_spn}{2\mu}
      \end{split}\end{equation*}
      since $p_s=(1-o(1))p_{s-1}$, and $\sum_{t\in [n-\lfloor\mu n\rfloor]} w_s(x_t) = \lfloor
      \nu n \rfloor$, and by definition of $p$. So,
      Corollary~\ref{cor:freedm}\ref{cor:freedm:tails} with
      $\tilde{\mu}=\frac{p_spn}{2\mu}$, $\tilde{\nu} =
      4C\alpha_{s-1}\frac{p_spn}{2\mu}$ and $\tilde{\rho} =
      C\alpha_{s-1}\frac{p_spn}{2\mu}$ yields
      \begin{align*}
        \Prob\left[\cE'_t ~ \text{and} ~
        \sum_{t\in [n]}X_t 
        \neq (1\pm 5C\alpha_{s-1}) 		
        \frac{p_spn}{2\mu}
        \right]
        \leq
        2 \exp\left(
        - \frac{C^2\alpha_{s-1}^2p_spn}{4\Delta \mu (1+5C\alpha_{s-1})}
        \right) \leq n^{-21}
      \end{align*}
      where the last inequality holds, since $\Delta\leq \frac{cn}{\log n}$ and
      by choice of $c$. This gives the first part of the claim, as
      $5C\alpha_{s-1}\leq 10Cp^{-1}\alpha_s$.

      The second part of the claim, concerning $H^*_0$, is very similar, and we
      only sketch the proof. We define $\cE''_t$ to be the event that
      $(H_{s-1},H^*_0)$ is $(\alpha_{s-1},2D+3)$-coquasirandom and
      $\big(H_{s-1},H^*_0,\phi'_s\big([t]\big)\big)$ satisfies the
      $(2\eta,2D+3)$-codiet condition, and let $\hist''_t$ be a history up to
      and including the embedding of $x_t$ which satisfies $\cE''_{t''}$. By a
      similar calculation as before, with $p^*\leq 1.1\gamma$ being the density
      of $H_0^*$, we see that when \RandomEmbedding{} is run, we have
      \[\Prob[x_t\AlgMap N_{H^*_0}(u)|\hist_{t-1}]=\big(1\pm6\eta\big)p^*\le\tfrac 32 \gamma,\]
      as the codiet condition for $(H_{s-1},H_0^*)$ makes sure that the
      candidate set for $x_t$ is of size $(1\pm 2\eta
      )p_{s-1}^{\deg_G^-(x_t)}(n-t+1)$, while among these candidates $(1\pm
      2\eta)p_{s-1}^{\deg_G^-(x_t)}p^*(n-t+1)$ vertices belong to
      $N_{H^*_0}(u)$.
 
      Having that, we can again define $X'_t:=w_s(x_t)\cdot\ONE_{x_t\AlgMap
        N_{H^*_0}(u)}$, and as before we obtain
      \[\sum_{t\in[n]}\Exp[X'_t|\hist''_{t-1}]=\sum_{t\in[n-\mu
          n]}\Exp[X'_t|\hist''_{t-1}]\le\tfrac53\cdot\frac{\gamma p n}{2\mu}\,.\]

      Applying Corollary~\ref{cor:freedm}\ref{cor:freedm:uppertail}, we get
      \[\Prob\left[\cE''_t ~ \text{and} ~
          \sum_{t\in [n]}X'_t > \frac{\gamma pn}{\mu} \right] \leq n^{-21}\,.\]
      This is the second part of the claim; the total failure probability is at
      most $2n^{-21}<n^{-20}$.
    \end{claimproof}
    
    We now need to show that it is unlikely that a given $s'>s$ is the first
    $s'$ for which~\eqref{eq:sumweight:main} fails. 
    To that end, fix~$s'$ with
    $s^*-\lfloor\mu n\rfloor\le s<s'\le s^*$. For $s<i\le s^*$ we
    define \[Y_{i}:=\sum_{v\in N_{H_{i-1}}(u)\setminus N_{H_{i}}(u)} w_s(v)\,.\]
    We have
    \[\sum_{v\colon vu\in E(H_{s'})} w_s(v)=\sum_{v\colon vu\in E(H_{s})}
      w_s(v)-\sum_{i=s+1}^{s'} Y_{i}\,,\] and so the missing piece to establishing
    Claim~\ref{cl:sumweight:main} is to show that the sum of the~$Y_i$ is
    likely to stay close to its expectation. We start by determining this expectation.

    \begin{claim}\label{cl:sumweight:Yexp}
      Suppose that $H_{i-1}$ is $(\alpha_{i-1},2D+3)$-quasirandom, and suppose
      that~\eqref{eq:sumweight:main} holds for $s'=i-1$. Then when
      \RandomEmbedding{} is run to embed $G''_i[{\scriptstyle [n-\delta n]}]$
      into $H_{i-1}$, we have
      \[\Exp[Y_i|H_{i-1}]=\big(1\pm 10^4CD\alpha_{i-1}\delta^{-1}\big)\cdot\tfrac{pe(G''_i)}{\mu n}\,.\]
    \end{claim}
    \begin{claimproof}
      By definition of $Y_i$ we obtain
      \[
        Y_i=\sum_{v\in N_{H_{i-1}}(u)}
        w_s(v)\cdot \ONE_{\text{uv is used when embedding }G_i''}
      \]
      and therefore
      \[
        \Exp[Y_i|H_{i-1}]
        = 
        \sum_{v\in N_{H_{i-1}}(u)}
        w_s(v)\cdot \Prob[\text{uv is used when embedding }G_i''|H_{i-1}]~ .
      \]
      Under assumption that $H_{i-1}$ is $(\alpha_{i-1},2D+3)$-quasirandom,
      Lemma~\ref{lem:probedge} yields
      \[
        \Exp[Y_i|H_{i-1}]
        = 
        \sum_{v\in N_{H_{i-1}}(u)}
        w_s(v)\cdot 
        (1\pm 500C\alpha_{i-1}\delta^{-1})^{4D+2} \frac{2e(G_i'')}{p_{i-1}n^2}~ .
      \]
      Applying~\eqref{eq:sumweight:main} for $s'=i-1$ finally leads to
      \begin{align*}
        \Exp[Y_i|H_{i-1}]
        & = (1\pm 10Cp^{-1}\alpha_{i-1})
          \frac{p_{i-1}pn}{2\mu}\cdot 
          (1\pm 500C\alpha_{i-1}\delta^{-1})^{4D+2} \frac{2e(G_i'')}{p_{i-1}n^2}\\
        &  = (1 \pm 10^4C\alpha_{i-1}D\delta^{-1} )\frac{pe(G_i'')}{\mu n}
      \end{align*}
      where last equality holds since $p\geq \nu \mu$ and $\delta \ll \nu\ll
      \mu$.
    \end{claimproof}
  
    What we would like to do now is apply Corollary~\ref{cor:freedm} (or
    Lemma~\ref{lem:freedman}) to show that the sum of the $Y_i$ is likely to be
    close to the sum of the observed expectations we just calculated. But
    unfortunately this approach fails, because the range of the $Y_i$ is too
    large; it is possible that there are as few as $O(\log n)$ vertices which
    contain all the weight of $w_s$ in $N_{H_s}(v)$, and we might use all the
    edges to these vertices in embedding a single $G''_i$. This is the reason
    for defining the random variables 
    \[
      Z_{i}:=\max\{Y_{i}-K'\Delta,0\} \quad \text{with} \quad K'=10^{10}CD^3\delta^{-1}\,.
    \]
    Trivially the `capped' random variable \[Y'_i:=Y_i-Z_i\] does not have an
    excessively large range (it cannot exceed $K'\Delta$), and we shall see (in
    the proof of Claim~\ref{cl:sumweight:main}) that we can apply
    Corollary~\ref{cor:freedm} to argue that the sum of the~$Y'_i$ is concentrated. In order to
    show that this implies that also the sum of the~$Y_i$ is concentrated, we need
    to argue that the `error' caused by the $Z_i$ is not too large, which we
    establish in Claim~\ref{cl:sumweight:Zsum}. As preparation for this, we will
    analyze the behaviour of the variables~$Z_i$ more in detail (in
    Claim~\ref{cl:probcap}) and bound their expectation (in
    Claim~\ref{cl:sumweight:Zexp}; we will need this bound when we show that the
    sum of the~$Y'_i$ is concentrated).

    Let us now try to understand the behaviour of~$Z_i$.
    Consider the embedding of $G''_i[{\scriptstyle [n-\delta n]}]$ into
    $H_{i-1}$ by \RandomEmbedding{}. Observe that~$Z_i$ is determined by the
    vertex~$x_t$ that is embedded to~$u$ and by the embedding of neighbours
    of~$x_t$. Until we embed~$x_t$ to $u$ at
    time $t$, we have used no edges of~$H_{i-1}$ leaving $u$. On embedding a vertex to $u$,
    we have
    \[\sum_{y\in\LNBH_{G''_i}(x_t)}w_s\big(\phi'_i(y))\le D\Delta\,,\]
    because $x_t$ has at most $D$ neighbours preceding it in the degeneracy
    order. Consider now the successive embedding of the forward neighbours
    $y_1,\dots,y_\ell$ of $x_t$ by \RandomEmbedding{}. In order for $Z_i>0$ to
    occur, we have to embed the next $j$ forwards neighbours of $x_t$ (for some
    $j$) to vertices such that $\sum_{k=1}^j w_s\big(\phi'_i(y_k)\big)\ge
    (K'-D-1)\Delta$. We say that the embedding of $G''_i$ \emph{goes near the
      cap} at the first time when we embed a $y_j$ such that this inequality
    holds. We write $\CapEvent(i,y)$ for the event that the embedding of $G''_i$ goes near the cap at the time when we embed $y$ (note that these events are pairwise disjoint as $y$ ranges over $V(G''_i)$), and we write $\CapEvent(i)$ for their union, i.e.\ the event that the embedding of $G''_i$ goes near the cap at some time. If $\CapEvent(i,y_j)$ occurs, we have the inequality $Z_i\le\sum_{k=j+1}^\ell
    w_s\big(\phi'_i(y_k)\big)$; it is important to note that the right hand side
    depends only on embeddings after the event of going near the cap is decided.
    Our next aim is to show that, conditioned on the embedding up to the time
    when $x_t$ is embedded to $u$, it is unlikely that the embedding goes near
    the cap.
  
    \begin{claim}\label{cl:probcap}
      Suppose that $H_{i-1}$ is $(\alpha_{i-1},2D+3)$-quasirandom. Suppose
      furthermore that $\psi_t$ is a partial embedding of $G''_i$ to $H_{i-1}$
      generated by \RandomEmbedding{} which embeds $x_t$ to $u$ (and embeds no
      vertices after $x_t$). Suppose that~$\psi_t$ is such that the probability, conditioned on
      $H_{i-1}$ and $\psi_t$, of $(H_{i-1},\im\phi'_i)$ failing to have the
      $(C\alpha_{i-1},2D+3)$-diet condition is at most $n^{-5}$. Then we have
     \[\Prob\big[\CapEvent(i)\,\big|\,H_{i-1},\psi_t\big]\le 3e^{-K'/8}\,.\]
    \end{claim}
  \begin{claimproof}
    With the notation from above, set
    \[
      X_k=w_s(\phi_i'(y_k))
    \]
    for every forward neighbour~$y_k$ of $x_t$, and observe that $0\leq X_k\leq
    \Delta$. Let $\hist'_{k-1}$ be a history up to and including the embedding
    $\psi_r$ of the vertex $x_r$ which comes immediately before $y_k$ in the
    ordering of $G_i''$. Let $\tilde\cE_r$ be the event that $(H_{i-1},\im\psi_r)$
    satisfies the $(C\alpha_{i-1},2D+3)$-diet condition.

    Then, if $\tilde\cE_r$ holds, we have
    \[
      \Exp\left[
        X_k | \hist'_{k-1} 
      \right]
      \leq 
      \frac{\lfloor \nu n \rfloor}{
        \frac{1}{2}p^D\mu n}
      \leq 2p^{-D}\mu^{-1} \nu
    \]
    since the sum over all weights from $G_i''$ is $\lfloor \nu n\rfloor$, while
    the diet-condition ensures that the candidate set for $y_k$ is of size at
    least
    \[
      (1-C\alpha_{i-1})p^{D}
      \lfloor \mu n\rfloor 
      \geq \frac{1}{2}p^D\mu n~ .
    \]

    In particular,
    \[
      \sum_{k=1}^\ell
      \Exp[X_k | \hist'_{k-1}]
      \leq 2p^{-D}\mu^{-1} \nu \Delta~ .
    \]
    Applying the first part of Corollary~\ref{cor:freedm}\ref{cor:freedm:tails}
    with
    $\cE = \bigcup_r \tilde\cE_r$,
    $\tilde{\mu}=\tilde{\nu}
    = p^{-D}\mu^{-1}\nu \Delta$,
    $\tilde{\rho}=(K'-D-1-2p^{-D}\mu^{-1}\nu)\Delta$
    and $R=\Delta$
    we then obtain that 
    \begin{align*}
      \Prob\left[
      \cE \text{ and } ~
      \sum_{k=1}^\ell X_k \geq (K'-D-1)
      \Delta
      \right]
      & \leq 2\exp
        \left(
        - \frac{(K'-D-1-2p^{-D}\mu^{-1}\nu)^2}{2(K'-D-1)}
        \right)\\
      & \leq 2\exp\left(
        -\frac{\left(\frac{1}{2}K'\right)^2}{2K'} 
        \right)
        = 2\exp\left(-\frac{K'}{8}\right)~ .
    \end{align*}
    Since by assumption the probability of $\cE$ not occurring is at most $n^{-5}$, the claim follows.
  \end{claimproof}
  
  Now we can use this, and Lemma~\ref{lem:vertex}, to estimate the expectation of $Z_i$ conditioned on $H_{i-1}$ which is quasirandom.
  
  \begin{claim}\label{cl:sumweight:Zexp}
   Suppose that $H_{i-1}$ is $(\alpha_{i-1},2D+3)$-quasirandom.
   Then we have
   \[\Exp\big[Z_i\big|H_{i-1}\big]\le 13\tfrac{e(G''_i)}{n}e^{-K'/8}\cdot 2\nu\mu^{-1}p_{s^*}^{-D}\,.\]
  \end{claim}
  \begin{claimproof}
    We have
    \begin{equation}\label{eq:Zexp:a}
      \Exp[Z_i|H_{i-1}]
      = \sum_{x\in V(G_i'')}
      \Prob[x \AlgMap u | H_{i-1}] \cdot
      \Exp[Z_i | x \AlgMap u, H_{i-1}]\,.
    \end{equation}
    Assuming that $H_{i-1}$ is $(\alpha_{i-1},2D+3)$-quasirandom, we know by Lemma~\ref{lem:vertex} that
    \begin{equation}\label{eq:Zexp:b}
      \Prob[x \AlgMap u | H_{i-1} ]
      = (1+10^4C\alpha_{s-1}D\delta^{-1})\frac{1}{n} =\big(1\pm\tfrac12\big) \frac{1}{n}\,.
    \end{equation}
    For estimating $\Exp[Z_i | x \AlgMap u, H_{i-1}]$, we let $\cE$ be the event that $(H_{i-1},\im\phi'_i)$ satisfies the
    $(C\alpha_{i-1},2D+3)$-diet condition. Then by linearity of expectation we have
    \begin{equation}\begin{split}\label{eq:Zexp:c}
     \Exp[Z_i | x \AlgMap u, H_{i-1}]&=\Exp[Z_i\ONE_{\cE} | x \AlgMap u, H_{i-1}]+\Exp[Z_i\ONE_{\bar{\cE}} | x \AlgMap u, H_{i-1}]\\
     &\le \Exp[Z_i\ONE_{\cE} | x \AlgMap u, H_{i-1}]+n\cdot \tfrac{2n^{-9}}{2/n}=\Exp[Z_i\ONE_{\cE} | x \AlgMap u, H_{i-1}]+4n^{-7}\,,
    \end{split}\end{equation}
     where the estimate for the second term is from Lemma~\ref{lem:22} bounding
     the probability of $\bar{\cE}$ and~\eqref{eq:Zexp:b} lower bounding the
     probability of $x\AlgMap u$, and since trivially $Z_i\le Y_i\le n$. To
     estimate the first term, we observe that since outside $\CapEvent(i)$ we
     have $Z_i=0$, it follows that
     \begin{multline}\label{eq:Zexp:twirl}
       \Exp[Z_i\ONE_{\cE}|x\AlgMap u,H_{i-1}] \\
       =\sum_{z\in V(G''_i)}\Prob[\CapEvent(i,z)|x\AlgMap
      u,H_{i-1}]\cdot\Exp[Z_i\ONE_{\cE}|x\AlgMap u,H_{i-1},\CapEvent(i,z)]\,.
    \end{multline}
    Note that the only terms of the sum in which the probability is positive are those with $z$ a forwards neighbour of $x$, so fix such a $z$. Recall that if $\CapEvent(i,z)$ occurs then we have
    \[Z_i\le\sum_{\substack{y\in N_{G''_i}(x)\\y\text{ comes after }z}}w_s(\phi'_i(y))\,,\quad\text{and so}\quad
    Z_i\ONE_{\cE}\le\sum_{\substack{y\in N_{G''_i}(x)\\y\text{ comes after }z}}w_s(\phi'_i(y))\ONE_{\cE}\,.\]
    For bounding $\Exp[Z_i\ONE_{\cE}|x\AlgMap u,H_{i-1},\CapEvent(i,z)]$,
    for any forwards neighbour~$y$ of~$x$ which comes after $z$ in the degeneracy order, let $\cH'_{<y}$ denote
    any history up to and including the embedding of the vertex which comes
    immediately before~$y$ that is consistent with $x\AlgMap u$ and is contained in $\CapEvent(i,z)$. Then we have
    \begin{multline}\label{eq:Zexp:twirl2}
      \Exp[Z_i\ONE_{\cE}|x\AlgMap u,H_{i-1},\CapEvent(i,z)] \\
      \le\sum_{\substack{y\in N_{G''_i}(x)\\y\text{ comes after }z}}~\sum_{\cH'_{<y}}
      \Exp\Big[w_s\big(\phi'_i(y)\big)\ONE_{\cE}\Big|\cH'_{<y},H_{i-1}\Big]
      \cdot \Prob[\cH'_{<y}|x\AlgMap u,H_{i-1},\CapEvent(i,z)]\,.
    \end{multline}
    Let $y$ be a forwards neighbour of $x$ which comes after $z$. Then $y$ is
    not isolated, so it is in the first $n-\mu n$ vertices of $G''_i$.
    We want to calculate
    $\Exp\Big[w_s\big(\phi'_i(y)\big)\ONE_{\cE}\Big|\cH'_{<y},H_{i-1}\Big]$.
    There are two cases to consider.
    First, if $\cE$ occurs, then since~$y$ is in the first $n-\mu n$ vertices
    of~$G''_i$, it has a candidate set of size at
    least
    \[
      (1-C\alpha_{i-1})p_{i-1}^{D} \lfloor \mu n\rfloor \geq
      \frac{1}{2}p_{s^*}^D\mu n\,.\]
    Hence we embed $y$ uniformly to a set of
    size at least $\tfrac12p_{s^*}^D\mu n$, so (because the total weight of all
    vertices in $G''_s$ is $\lfloor\nu n\rfloor$) the expectation of
    $w_s\big(\phi'_i(y)\big)$ conditioned on $\cH'_{<y}$ and $H_{i-1}$ is at
    most $\tfrac{\nu n}{p_{s^*}^D\mu n/2}$. Second, if $y$ is chosen from a
    candidate set of size less than $\tfrac12p_{s^*}^D\mu n$, then the event
    $\cE$ does not occur, and so the conditional expectation we
    want to calculate is zero. In either case, we obtain
    \[\Exp\Big[w_s\big(\phi'_i(y)\big)\ONE_{\cE}\Big|\cH'_{<y},H_{i-1}\Big]\le 2\nu
    \mu^{-1}p_{s^*}^{-D}\,.\]
    Plugging this into~\eqref{eq:Zexp:twirl2} gives 
    \[\Exp[Z_i\ONE_{\cE}|x\AlgMap u,H_{i-1},\CapEvent(i,z)]\le d_{G''_i}(x)\cdot 2\nu
    \mu^{-1}p_{s^*}^{-D}\,,\]
    since the sum over $\cH'_{<y}$ of $\Prob[\cH'_{<y}|x\AlgMap u,H_{i-1},\CapEvent(i,z)]$ is trivially $1$, and $d_{G''_i}(x)$ is at least as big as the number of forward neighbours of $x$ which come after $y$. Now putting this into~\eqref{eq:Zexp:twirl} we obtain
    \begin{align*}
     \Exp[Z_i\ONE_{\cE}|x\AlgMap u,H_{i-1}]&\le\sum_{z\in V(G''_i)}\Prob[\CapEvent(i,z)|x\AlgMap
      u,H_{i-1}]\cdot d_{G''_i}(x)\cdot 2\nu \mu^{-1}p_{s^*}^{-D}\\
    &=\Prob[\CapEvent(i)|x\AlgMap
      u,H_{i-1}]\cdot d_{G''_i}(x)\cdot 2\nu \mu^{-1}p_{s^*}^{-D}\,.
    \end{align*}
      
    We finally use Claim~\ref{cl:probcap} to estimate $\Prob[\CapEvent(i)|x\AlgMap
    u,H_{i-1}]$. By~\eqref{eq:Zexp:b}, we have $\Prob[x\AlgMap
    u|H_{i-1}]\ge\tfrac{1}{2n}$. By Lemma~\ref{lem:22}, the probability that
    $(H_{i-1},\phi'_i)$ fails to have the $(C\alpha_{i-1},2D+3)$-diet condition,
    conditioned on $H_{i-1}$, is at most $2n^{-9}$. Consequently, summing up
    $\Prob[\psi_x|x\AlgMap u,H_{i-1}]$ over partial embeddings $\psi_x$ which
    embed the vertices up to and including $x$ of $G''_i$, and embed $x$ to $u$,
    but which fail the condition of Claim~\ref{cl:probcap} (i.e.\ the
    probability that $(H_{i-1},\phi'_i)$ fails to have the
    $(C\alpha_{i-1},2D+3)$-diet condition, conditioned on $H_{i-1}$ and
    $\psi_x$, exceeds $n^{-5}$), we obtain at most $4n^{-3}$. For any $\psi_x$
    which does satisfy the condition of Claim~\ref{cl:probcap}, we have
    $\Prob[\CapEvent(i)|H_{i-1},\psi_x]\le 3e^{-K'/8}$. Putting these together, we have
    \begin{align*}
     \Prob[\CapEvent(i)|x\AlgMap u,H_{i-1}]&=\sum_{\psi_x}\Prob[\psi_x|x\AlgMap u,H_{i-1}]\cdot\Prob[\CapEvent(i)|\psi_x,x\AlgMap u,H_{i-1}]\\
     &\le 4n^{-3}\cdot 1+1\cdot 3e^{-K'/8}= 3e^{-K'/8}+4n^{-3}\,.
    \end{align*}
    At last, we obtain
    \[
      \Exp[Z_i\ONE_{\cE}|x \AlgMap u,H_{i-1}]\leq (3e^{-\frac{K'}{8}}+4n^{-3})\cdot d_{G_i''}(x)\cdot 2\nu\mu^{-1}p_{s^*}^{-D}\,.\]
    Thus, using~\eqref{eq:Zexp:c}, we have
    \begin{align*}
     \Exp[Z_i|x \AlgMap u,H_{i-1}]&\le (3e^{-\frac{K'}{8}}+4n^{-3})\cdot d_{G_i''}(x)\cdot 2\nu\mu^{-1}p_{s^*}^{-D}+4n^{-7}\\
     &\le 3e^{-\frac{K'}{8}}\cdot d_{G_i''}(x)\cdot 2\nu\mu^{-1}p_{s^*}^{-D}+n^{-2}\,,
    \end{align*}
    and so by~\eqref{eq:Zexp:a} and~\eqref{eq:Zexp:b} we get
    \begin{align*}
      \Exp[Z_i|H_{i-1}]
      & \leq 
        \sum_{x\in V(G_i'')} \frac{2}{n}\cdot\Big( 3e^{-\frac{K'}{8}}\cdot d_{G_i''}(x)\cdot 2\nu\mu^{-1}p_{s^*}^{-D}+n^{-2} \Big)\\
      & = \frac{2 \sum_{x\in V(G_i'')} d_{G_i''}(x)}{n}\cdot 3e^{-\frac{K'}{8}}\cdot 2\nu\mu^{-1}p_{s^*}^{-D} + 2n^{-2}
        < 13 \tfrac{e(G''_i)}{n} \cdot e^{-\frac{K'}{8}} \cdot 2\nu\mu^{-1}p_{s^*}^{-D}\,.
    \end{align*}
\end{claimproof}
  
This expectation is tiny, because the term $e^{-K'/8}$ is very small. Thus we see that the expectations of $Y'_i$ and
$Y_i$ (conditioning on any $H_{i-1}$ which is quasirandom) are very close. The final
thing we have to do before we complete the proof of
Claim~\ref{cl:sumweight:main} is to show that the sum of the $Z_i$ is likely to
be very small.

  \begin{claim}\label{cl:sumweight:Zsum}
    With probability at least $1-2n^{-20}$, the following event occurs when \PackingProcess{} is run. Either $H_i$ is not $(\alpha_i,2D+3)$-quasirandom for some $i\in[s^*]$, or $(H_{i-1},\phi'_i([n-\mu n]))$ does not satisfy the $(C\alpha_{i-1},2D+3)$-diet condition for some $i\in[s^*]$, or we have
    \[\sum_{i=s+1}^{s'}Z_i\le \tfrac{\alpha_s p n}{1000\mu}\,.\]
  \end{claim}
  \begin{claimproof} Let $\cE$ denote the event that $H_i$ is $(\alpha_i,2D+3)$-quasirandom for each $i\in[s^*]$, and $(H_{i-1},\phi'_i([n-\mu n]))$ satisfies the $(C\alpha_{i-1},2D+3)$-diet condition for each $i\in[s^*]$. So we want to show that it is likely that either $\cE$ fails or $\sum_{i=s+1}^{s'}Z_i<\tfrac{\alpha_spn}{1000\mu}$.
  
   In order to prove this claim, we need to reinterpret $\sum_{i=s+1}^{s'}Z_i$. The random variables $Z_i$ can be very large, so that Corollary~\ref{cor:freedm} does not help us.
   
   What we do is to use our earlier observation that we can understand $Z_i$ as
   follows. We watch \RandomEmbedding{} as it embeds $G''_i[{\scriptstyle
     [n-\delta n]}]$, until it embeds some $x_t$ to $u$, and then embeds
   the forwards neighbours of $x_t$ until it goes near the cap (if one or the
   other event does not occur, then $Z_i=0$). Then $Z_i$ is at most the sum of
   $w_s\big(\phi'_i(y)\big)$ taken over forwards neighbours $y$ of $x_t$
   which are embedded after reaching the cap. We refer to these vertices~$y$
   as \emph{after-cap} vertices. We then use the inequality
   \[\sum_{i=s+1}^{s'}Z_i\le\sum_{i=s+1}^{s'}~\sum_{\substack{v=\phi'_i(y)\\ \text{for $y$ after-cap in }G''_i}}w_s(v)\,,\]
   where the right hand side sum runs over all after-cap vertices in all graphs
   $G''_{s+1},\dots,G''_{s'}$. For a given after-cap vertex $y\in V(G''_i)$ we
   know $y$ is an after-cap vertex before we embed it. Now when we embed $y$,
   provided the $(C\alpha_{i-1},2D+3)$-diet condition holds for
   $(H_{i-1},\phi([n-\mu n]))$, we embed it uniformly into a set~$S$ of size at
   least $\tfrac12 p_{i-1}^D\mu n$ (because $y$, since it is not isolated, must
   be one of the first $n-\mu n$ vertices of $G''_i$). The sum of $w_s(z)$ over
   the vertices $z$ of~$S$ is at most $\lfloor\nu n\rfloor$. So the expected
   value of $w_s\big(\phi'_i(y)\big)$, conditioned on the history up to the time
   $y-1$ immediately before embedding $y$ and on the $(C\alpha_{i-1},2D+3)$-diet
   condition holding for $(H_{i-1},\phi'_i([y-1]))$, is at most
   $2\lfloor\nu n\rfloor p_{i-1}^{-D}\mu^{-1} n^{-1}\le 2p
   p_{s^*}^{-D}\mu^{-2}$, where the inequality uses the equation $\lfloor\nu
   n\rfloor\lfloor\mu n\rfloor=p\binom{n}{2}$.
   
   Let \[L:=40Dne^{-K'/8}\,,\] and define a random variable $X_j$ for $1\le j\le
   L$ by $X_j=w_s\big(\phi'_i(y)\big)$, where the $j$th after-cap vertex in a
   run of \PackingProcess{} is $y\in V(G''_i)$ ($y$ and hence~$i$ depend on the
   run of \PackingProcess{}). If there is no such after-cap vertex, we let
   $X_j:=0$. Observe that we have $0\le X_j\le\Delta$ for each $j$, and what we
   just calculated is that, letting $\hist_j$ denote the history of
   \PackingProcess{} up to immediately before embedding the $j$th after-cap
   vertex $y\in V(G''_i)$, if the $(C\alpha_{i-1},2D+3)$-diet condition holds
   for $(H_{i-1},\phi'_i([y-1]))$, then $\Exp[X_j|\hist_j]\le2p
   p_{s^*}^{-D}\mu^{-2}$. So we can apply
   Corollary~\ref{cor:freedm}\ref{cor:freedm:uppertail}, with $\tilde{\mu}=2p
   p_{s^*}^{-D}\mu^{-2}L$, to obtain
   \[\Prob\Big[\cE\text{ and }\sum_{j=1}^LX_j>4p p_{s^*}^{-D}\mu^{-2}L\Big]<\exp\Big(-\frac{2p p_{s^*}^{-D}\mu^{-2}L}{4\Delta}\Big)\le n^{-20}\,,\]
   where the final inequality follows from $\Delta=\tfrac{cn}{\log n}$ and by choice of $c$ and $K'$.
   
   Let $T$ denote the total number of after-cap vertices encountered during the
   entire run of \PackingProcess{}. Since by choice of $K'$ and~\eqref{eq:defconsts} we have $4p
   p_{s^*}^{-D}\mu^{-2}L<\tfrac{\alpha_s p n}{1000\mu}$, what we have just
   argued is that
   \[\Prob\Big[\cE \text{ and } T\le L \text{ and }
     \sum_{i=s+1}^{s'}Z_i>\tfrac{\alpha_s p n}{1000\mu}\Big]
     \le n^{-20}\,.
   \]
   What we now want to do is estimate the probability of the event
   that $\cE$ occurs and that $T>L$.
  
   To that end, for each $s+1\le i\le s'$, we define $X'_i$ to be the number of after-cap vertices embedded 
   from $G''_i$ in a given run of \PackingProcess{}. By definition we have $T=\sum_{i=s+1}^{s'}X'_i$. Now, if $H_{i-1}$ is $(\alpha_{i-1},2D+3)$-quasirandom, we can estimate $\Exp[X'_i|H_{i-1}]$ as follows. First, observe $X'_i$ can only be positive if some $x_t\in V(G''_i)$ is embedded to $u$, and then $G''_i$ goes near the cap, and then the remaining neighbours of $x_t$ will be the after-cap vertices counted by $X'_i$. So we have
   \[X'_i\le\sum_{x_t\in V(G''_i)}\ONE_{x_t\AlgMap u}\ONE_{\CapEvent(i)}\cdot d_{G''_i}(x_t)\,.\]
   It follows that
   \[\Exp\big[X'_i\big|H_{i-1}\big]\le\sum_{x_t\in
       V(G''_i)}d_{G''_i}(x_t)\Prob\big[x_t\AlgMap u\mid
     H_{i-1}\big]\cdot\Prob\big[\CapEvent(i)\big|H_{i-1},\psi_t\big]\,,\]
   where~$\psi_t$ is a partial embedding of the first $t$ vertices of~$G''_i$ into~$H_{i-1}$ generated by
   \RandomEmbedding{} which embeds~$x_t$ to~$u$.
   By respectively Lemma~\ref{lem:vertex} and Claim~\ref{cl:probcap}, we have
   \[\Exp\big[X'_i\big|H_{i-1}\big]\le\sum_{x_t\in V(G''_i)}d_{G''_i}(x_t)\cdot \big(\tfrac{2}{n}\cdot 3e^{-K'/8}+2n^{-4}\big)\le 20De^{-K'/8}\,, \]
   where the first inequality uses the observation that, by Lemma~\ref{lem:22},
   there is at most $2n^{-4}$ probability of generating $\psi_t$ such that the
   $(C\alpha_{i-1},2D+3)$-diet condition holding for $(H_{i-1},\phi([n-\mu n]))$
   has more than $n^{-5}$ chance of failing (when embedding the remaining
   vertices). The second inequality uses the fact that $G''_i$ has at most $Dn$
   edges and so the sum of its degrees is at most $2Dn$.
   
   Since $0\le X'_i\le\Delta$ for each $i$, we can apply Corollary~\ref{cor:freedm}\ref{cor:freedm:uppertail}, with $\tilde\mu=20Dne^{-K'/8}$, to obtain
   \[\Prob\big[\cE\text{ and }\sum_{i=s+1}^{s'}X'_i>40Dne^{-K'/8}\big]<\exp\big(-\tfrac{20Dne^{-K'/8}}{4\Delta}\big)\le n^{-20}\,,\]
   where the second inequality comes from $\Delta=\tfrac{cn}{\log n}$ and choice of $c$ and $K'$. Since $\sum_{i=s+1}^{s'} X_i'= T$, this proves as desired that it is unlikely that $\cE$ occurs and $T>L$.
   
   Putting these two pieces together, we conclude that with probability at most $2n^{-20}$, the event $\cE$ occurs and we have $\sum_{i=s+1}^{s'}Z_i>\tfrac{\alpha_spn}{1000\mu}$. This completes the proof of the claim.
  \end{claimproof}
  
  The reader might at this point wonder why we cannot simply estimate the sum of the $Y_i$ by modifying the above method. The point is that it is not easy to obtain an accurate estimate of the quantity $\Exp[X_j|\hist_j]$ in the above proof (the upper bound we obtain above is off from the truth by a rather large factor, compensated for by the unlikeliness of going near the cap), and we would need such an accurate estimate for Claim~\ref{cl:sumweight:main}. 
  
  Finally, we are in a position to prove Claim~\ref{cl:sumweight:main}.
  
  \begin{claimproof}[Proof of Claim~\ref{cl:sumweight:main}]
    Firstly, by Claim~\ref{cl:sumweight:start} we have that either $(H_i,H^*_0)$ is
    not $(\alpha_i,2D+3)$-co\-qua\-si\-ran\-dom for some $i\in[s^*]$, or
    $\big(H_{i-1},\phi'_i\big([t]\big)\big)$ does not satisfy the
    $(C\alpha_{i-1},2D+3)$-diet condition for some $i\in[s^*]$ and
    $t\in[n-\delta n]$, or $\big(H_{i-1},H^*_0,\phi'_i\big([t]\big)\big)$ does
    not satisfy the $(2\eta,2D+3)$-codiet condition for some $i\in[s^*]$ and
    $t\in[n-\delta n]$, or that~\eqref{eq:sumweight:Hssum} holds
    and~\eqref{eq:sumweight:main} holds for the case $s'=s$ with probability at
    least $1-n^{-20}$. Now, let $s<s'\le s^*$. We aim to show that with
    probability at most $3n^{-20}$ we have that~\eqref{eq:sumweight:main}
    continues to hold for~$s'$. Taking a union bound over the choices of $s'$
    then completes the proof of Claim~\ref{cl:sumweight:main}.
    
    More precisely, let $\cE$ denote the event that $u\not\in\im\phi'_s$, and
    $(H_i,H^*_0)$ is $(\alpha_i,2D+3)$-coquasirandom for each $i\in[s^*]$, and
    $\big(H_{i-1},\phi'_i\big([t]\big)\big)$ satisfies the
    $(C\alpha_{i-1},2D+3)$-diet condition for each $i\in[s^*]$ and
    $t\in[n-\delta n]$, and $\big(H_{i-1},H^*_0,\phi'_i\big([t]\big)\big)$
    satisfies the $(2\eta,2D+3)$-codiet condition for each $i\in[s^*]$ and
    $t\in[n-\delta n]$, and~\eqref{eq:sumweight:main} holds for each $s\le
    i<s'$.  Our goal is to show that $\cE$ occurs and~\eqref{eq:sumweight:main}
    fails for $s'$ with probability at most $3n^{-20}$. 

    By Claim~\ref{cl:sumweight:start}, with probability at least $1-n^{-20}$,
    either we witness a failure of $\cE$ before beginning to embed $G''_{s+1}$,
    or we have $\sum_{v\colon vu\in E(H_s)}w_s(v)=\big(1\pm
    10Cp^{-1}\alpha_s\big)\tfrac{p_spn}{2\mu}$. Suppose that this likely event
    occurs, and that we do not witness a failure of $\cE$ before beginning to
    embed $G''_{s+1}$.
   
   Since we have
   \[\sum_{v\colon vu\in E(H_{s'})} w_s(v)=\sum_{v\colon vu\in E(H_{s})}
    w_s(v)-\sum_{i=s+1}^{s'} Y_{i}\,,\]
    and we want to conclude that it is unlikely that $\cE$ occurs and $\sum_{v\colon vu\in E(H_{s'})}w_s(v)\neq\big(1\pm 10Cp^{-1}\alpha_{s'}\big)\tfrac{p_{s'}pn}{2\mu}$, it is enough to estimate the probability, conditioned on $H_s$, that $\cE$ occurs and
    \begin{equation}\label{eq:sumweight:enough}
     \sum_{i=s+1}^{s'} Y_{i}\neq \big(1\pm 10Cp^{-1}\alpha_s\big)\tfrac{p_spn}{2\mu}-\big(1\pm 10Cp^{-1}\alpha_{s'}\big)\tfrac{p_{s'}pn}{2\mu}=\tfrac{(p_s-p_{s'})pn}{2\mu}\pm 20C\alpha_{s'}\tfrac{p_sn}{2\mu}\,.
    \end{equation}
    We have $Y_i=Y'_i+Z_i$ for each $i$, and so
    $\sum_{i=s+1}^{s'}Y_i=\sum_{i=s+1}^{s'}Y'_i+\sum_{i=s+1}^{s'}Z_i$.
    For showing that~\eqref{eq:sumweight:enough} is unlikely to occur, we will use
    Corollary~\ref{cor:freedm} to argue that $\sum Y'_i$ is concentrated and
    Claim~\ref{cl:sumweight:Zsum} to bound the contribution of $\sum Z_i$. Accordingly,
    we shall first calculate the expectation of $\sum Y'_i$.

    By Claim~\ref{cl:sumweight:Yexp}, provided $H_{i-1}$ does not
    witness that $\cE$ fails, we have $\Exp[Y_i|H_{i-1}]=\big(1\pm
    10^4CD\alpha_{i-1}\delta^{-1}\big)\cdot\tfrac{pe(G''_i)}{\mu n}$. By
    Claim~\ref{cl:sumweight:Zexp}, again provided $H_{i-1}$ does
    not witness that $\cE$ fails, we have $\Exp[Z_i|H_{i-1}]\le
    13\tfrac{e(G''_i)}{n}e^{-K'/8}\cdot 2\nu\mu^{-1}p_{s^*}^{-D}$. By linearity,
    we conclude
    \begin{align*}
     \Exp[Y'_i|H_{i-1}]&=\big(1\pm 10^4CD\alpha_{i-1}\delta^{-1}\big)\cdot\tfrac{pe(G''_i)}{\mu n}\pm 13\tfrac{e(G''_i)}{n}e^{-K'/8}\cdot 2\nu\mu^{-1}p_{s^*}^{-D}\\
     &=\big(1\pm 10^5 CD\alpha_{i-1}\delta^{-1}\big)\cdot\tfrac{pe(G''_i)}{\mu n} \,,
    \end{align*}
    where for the second inequality we use our choice of $K'$. Summing this up, we see that either $\cE$ fails or we have
    \begin{align*}
     \sum_{i=s+1}^{s'}\Exp[Y'_i|H_{i-1}]&=\sum_{i=s+1}^{s'}\big(1\pm 10^5 CD\alpha_{i-1}\delta^{-1}\big)\cdot\tfrac{pe(G''_i)}{\mu n}\\
     &=\sum_{i=s+1}^{s'}\tfrac{pe(G''_i)}{\mu n}\pm\sum_{i=s+1}^{s'}10^5 CD\alpha_{i-1}\delta^{-1}\cdot\tfrac{pDn}{\mu n}\\
     &=\tfrac{p}{\mu n}\big(p_s-p_{s'}\big)\binom{n}{2}\pm 10^5CD^2\delta^{-1}\mu^{-1}p\int_{i=-\infty}^{s'}\alpha_i\textrm{d}i\\
     &\eqByRef{eq:sum:alpha}
     \tfrac{pn}{2\mu}\big(p_s-p_{s'}\big)\pm \tfrac{1}{\mu}\pm 10^5CD^2\delta^{-1}\mu^{-1}p\cdot \frac{\delta n}{10^8CD^3}\alpha_{s'}\\
     &=\tfrac{pn}{2\mu}\big(p_s-p_{s'}\big)\pm \frac{p\alpha_{s'}n}{100\mu}=\tfrac{(p_s-p_{s'})pn}{2\mu}\pm C\alpha_{s'}\tfrac{p_sn}{2\mu}\,,
    \end{align*}
    where the final inequality is by choice of $C$
    and since $p\leq p_s+2\gamma$
    according to (\ref{eq:ps}). Now applying the first part of Corollary~\ref{cor:freedm}\ref{cor:freedm:tails}, with $\tilde\mu=\tfrac{(p_s-p_{s'})pn}{2\mu}$ and $\tilde\nu=\tilde\rho=C\alpha_{s'}\tfrac{p_sn}{2\mu}$, and using the fact $0\le Y'_i\le K'\Delta$, we obtain
    \[\Prob\Big[\cE\text{ and }\sum_{i=s+1}^{s'}Y'_i\neq \tfrac{(p_s-p_{s'})pn}{2\mu}\pm 2C\alpha_{s'}\tfrac{p_sn}{2\mu}\Big]<2\exp\big(-\tfrac{\tilde\rho^2}{2K'\Delta(\tilde\mu+\tilde\nu+\tilde\rho)}\big)\le n^{-20}\,,\]
    where the final inequality uses $\Delta=\tfrac{cn}{\log n}$ and the choice of $c$ and $K'$.
    
    Putting this estimate together with Claim~\ref{cl:sumweight:Zsum}, where we show that with probability at least $1-n^{-20}$ either $\cE$ does not occur, or we have $\sum_{i=s+1}^{s'}Z_i\le\tfrac{\alpha_spn}{1000\mu}$, we conclude the following. With probability at least $1-3n^{-20}$, either $\cE$ does not occur, or we have
    \[\sum_{i=s+1}^{s'}Y_i=\tfrac{(p_s-p_{s'})pn}{2\mu}\pm 2C\alpha_{s'}\tfrac{p_sn}{2\mu}\pm \tfrac{\alpha_spn}{1000\mu}=\tfrac{(p_s-p_{s'})pn}{2\mu}\pm 3C\alpha_{s'}\tfrac{p_sn}{2\mu}\,.\]
    If this holds~\eqref{eq:sumweight:enough} does not occur. With this we finally proved that with probability at most $3n^{-20}$ the event $\cE$ occurs and~\eqref{eq:sumweight:main} holds for each $s\le i<s'$ but fails for $s'$.
  \end{claimproof}
  
Finally, we argue that Claim~\ref{cl:sumweight:main} implies~\ref{appl:sumweight} holds with high probability. It is straightforward to check that $10Cp^{-1}\alpha_{s^*}<1$, and $p_{s^*}\le p$. Since $E(H)\subset E(H_{s^*})\cup E(H^*_0)$, provided~\eqref{eq:sumweight:main} with $s'=s^*$ and~\eqref{eq:sumweight:Hssum} hold, we have
 \begin{align*}
  \sum_{v:vu\in E(H)}w_s(v)&\le\sum_{v:vu\in E(H_{s^*})}w_s(v)+\sum_{v:vu\in E(H^*_0)}w_s(v)\\
  &\le 2\tfrac{p_{s^*}pn}{2\mu}+\tfrac{\gamma p n}{\mu}<\tfrac{10p^2n}{\mu}\,,
 \end{align*}
 where the last inequality follows since $p_{s^*},\gamma<p$. Thus by
 Claim~\ref{cl:sumweight:main}, with probability at least $1-4n^{-19}$, either
 the stated coquasirandomness, diet or codiet conditions fail,
 or~\ref{appl:sumweight} holds for fixed~$u$ and~$s$. So it is enough to check
 that it is unlikely that either the stated coquasirandomness, diet or codiet
 conditions fail. By respectively
 Lemma~\ref{lem:PackingProcess}\ref{lem:PackingProcess:quasi}, and
 Lemma~\ref{lem:22}\ref{22:b} and~\ref{22:d} (and the union bound over the at
 most $2n$ runs of \RandomEmbedding{}), the probability that either of these
 occur is at most $2n^{-5}+4n^{-8}$. For the latter, note that
 $\beta_t(\alpha_{i-1})\le C\alpha_{i-1}$ for each $i,t$. We finally conclude,
 using a union bound over~$u$ and~$s$, that~\ref{appl:sumweight} holds with
 probability at least $1-n^2\cdot 4n^{-19}-2n^{-5}-4n^{-8}>1-3n^{-5}$.
\end{proof}

\section{Concluding remarks}
\label{sec:concl}
 Once one knows that a given collection of graphs $\cG$ can be packed into a host graph $\widehat{H}$, it is natural to ask whether there is an efficient algorithm, randomised or not, which will exhibit such a packing. For $\cG$ as in Theorem~\ref{thm:main} (with the various parameters taken as fixed while $n$ is large) the obvious answer is simply to run our packing algorithm. Most of the steps in this algorithm simply consist of uniform random samples from sets which are of linear size and trivial to compute. In addition the completion step of \PackingProcess{} requires finding a perfect matching in a linear-sized and easily computed auxiliary bipartite graph; this is well known to be solvable in polynomial time using the augmenting paths algorithm. Finally, the completion step of \MatchLeaves{} requires sampling uniformly from the set of perfect matchings of a dense bipartite graph (which is linear-sized and easy to compute).
 
 If one assumes that it is possible to sample in polynomial time from these
 various distributions, then our algorithm clearly is polynomial time. However,
 if the source of randomness is an unbiased bit string (which is the natural and
 usual assumption) then one cannot sample exactly uniformly from arbitrary distributions. It is standard in the literature to ignore this problem (because sample approximately uniformly is possible and this suffices), but for completeness we give the details.
 
 For the random sampling in \PackingProcess{}, it is easy to sample approximately uniformly: using $k$ bits of randomness one can approximately sample any
 probability $p$ Bernoulli random variable up to an error $2^{-k}$ by viewing the bits as an integer in $[2^k]$ and returning $1$ if this integer is at most $2^kp$. One can similarly select uniformly from a set, by partitioning $[2^k]$ into intervals of approximately equal size corresponding to the set elements. For all the analysis here and in~\cite{DegPack}, it is easy to check that using $n$ random bits per sample, the sampling error is tiny compared to the probabilities we want to estimate and is absorbed by our error terms (in fact, $O(\log n)$ bits would suffice).
 
 However sampling a perfect matching approximately uniformly, even from a dense bipartite graph, is not so obviously possible. We actually do not need a uniform random perfect matching: what we need is any distribution on perfect matchings which satisfies the conclusion of Lemma~\ref{lem:match}, i.e.\ that any given edge is in the matching with probability not too much greater (by at most a factor $\tfrac32$ would suffice) than the average. So the question becomes whether one can sample in polynomial time from such a distribution. There is a Markov chain on perfect matchings due to Broder~\cite{Broder}, which Jerrum and Sinclair~\cite{JerSim} showed can be simulated and has polynomial mixing time. This means we can sample in polynomial time from a distribution on perfect matchings which is exponentially close to the uniform distribution, and in particular has the desired property.
 
 In conclusion, one can actually simulate the randomised algorithm of~\cite{DegPack} and this paper in polynomial time. Following the (somewhat) general belief that $\mathrm{RP}\neq\mathrm{NP}$, this suggests that the packing problem for the graphs we pack in this paper should not be NP-complete (in contrast to the general packing problem, which is known to be NP-complete~\cite{DorTar}). We suspect the problem is in P, but we do not know how to derandomise our algorithm, or otherwise provide a deterministic polynomial time algorithm for the packing.

\section{Acknowledgements}

Part of the work leading to this paper was done while PA and JB visited Hamburg.
PA and JB would like to thank TU Hamburg and the University of Hamburg for their
hospitality, the Suntory and Toyota International Centres for Economics and
Related Disciplines and TU Hamburg for financial support, and Heike B\"ottcher
for help with childcare arrangements.

\bibliographystyle{amsplain_yk}
\bibliography{packing}	 

\providecommand{\bysame}{\leavevmode\hbox to3em{\hrulefill}\thinspace}
\providecommand{\MR}{\relax\ifhmode\unskip\space\fi MR }
\providecommand{\MRhref}[2]{%
  \href{http://www.ams.org/mathscinet-getitem?mr=#1}{#2}
}
\providecommand{\href}[2]{#2}
\def\MR#1{\relax}
\begin{thebibliography}{10}

\bibitem{DegPack}
P.~Allen, J.~B{\"o}ttcher, J.~Hladk{\'y}, and D.~Piguet, \emph{Packing
  degenerate graphs}, arXiv:1711.04869.

\bibitem{BHPT}
J.~B{\"o}ttcher, J.~Hladk{\'y}, D.~Piguet, and A.~Taraz, \emph{An approximate
  version of the tree packing conjecture}, Israel J. Math. \textbf{211} (2016),
  no.~1, 391--446.

\bibitem{Broder}
A.~Z. Broder, \emph{How hard is it to marry at random? (on the approximation of
  the permanent)}, Proceedings of the Eighteenth Annual ACM Symposium on Theory
  of Computing (New York, NY, USA), STOC '86, ACM, 1986, pp.~50--58.

\bibitem{DorTar}
D.~Dor and M.~Tarsi, \emph{Graph decomposition is {NP}-complete: a complete
  proof of {H}olyer's conjecture}, SIAM J. Comput. \textbf{26} (1997), no.~4,
  1166--1187.

\bibitem{DukLefRod}
R.~A. Duke, H.~Lefmann, and V.~R\"odl, \emph{A fast approximation algorithm for
  computing the frequencies of subgraphs in a given graph}, SIAM J. Comput.
  \textbf{24} (1995), no.~3, 598--620.

\bibitem{FLM}
A.~Ferber, C.~Lee, and F.~Mousset, \emph{Packing spanning graphs from separable
  families}, Israel J. Math. \textbf{219} (2017), no.~2, 959--982.

\bibitem{FerSam}
A.~Ferber and W.~Samotij, \emph{Packing trees of unbounded degrees in random
  graphs}, arXiv:1607.07342.

\bibitem{Freedman}
D.~A. Freedman, \emph{On tail probabilities for martingales}, Ann. Probability
  \textbf{3} (1975), 100--118.

\bibitem{Gallian}
J.~A. Gallian, \emph{A dynamic survey of graph labeling}, Electron. J. Combin.
  \textbf{5} (1998), Dynamic Survey 6, 43.

\bibitem{GKLO:Designs}
S.~Glock, D.~K{\"u}hn, A.~Lo, and D.~Osthus, \emph{The existence of designs via
  iterative absorption}, arXiv:1611.06827.

\bibitem{GKLO:Fdesigns}
\bysame, \emph{{Hypergraph $F$-designs for arbitrary $F$}}, arXiv:1706.01800.

\bibitem{GyaLeh}
A.~Gy{\'a}rf{\'a}s and J.~Lehel, \emph{{Packing trees of different order into
  {$K_n$}}}, {Combinatorics (Proc. Fifth Hungarian Colloq., Keszthely, 1976)},
  {Colloq. Math. Soc. J{\'a}nos Bolyai}, vol.~18, North-Holland, Amsterdam,
  1978, pp.~463--469.

\bibitem{JLR}
S.~Janson, T.~{\L}uczak, and A.~Ruci{\'n}ski, \emph{Random graphs},
  Wiley-Interscience, 2000.

\bibitem{JerSim}
M.~Jerrum and A.~Sinclair, \emph{Approximating the permanent}, SIAM Journal on
  Computing \textbf{18} (1989), no.~6, 1149--1178.

\bibitem{JKKO}
F.~Joos, J.~Kim, D.~K\"uhn, and D.~Osthus, \emph{Optimal packings of bounded
  degree trees}, J. European Math. Soc., to appear.

\bibitem{Kee1}
P.~Keevash, \emph{The existence of designs}, arXiv:1401.3665.

\bibitem{Kee2}
\bysame, \emph{The existence of designs ii}, arXiv:1802.05900.

\bibitem{KKOT}
J.~Kim, D.~K{\"u}hn, D.~Osthus, and M.~Tyomkyn, \emph{A blow-up lemma for
  approximate decompositions}, arXiv:1604.07282.

\bibitem{Kirkman}
T.~P. Kirkman, \emph{On a problem in combinations}, Cambridge and Dublin Math.
  J. \textbf{2} (1847), 191--204.

\bibitem{KnoxKO}
F.~Knox, D.~K\"{u}hn, and D.~Osthus, \emph{Edge-disjoint {H}amilton cycles in
  random graphs}, Random Structures Algorithms \textbf{46} (2015), no.~3,
  397--445.

\bibitem{Lucas}
E.~Lucas, \emph{R\'{e}cr\'{e}ations math\'{e}matiques}, 2i\`eme \'{e}d.,
  nouveau tirage, Librairie Scientifique et Technique Albert Blanchard, Paris,
  1960.

\bibitem{MRS}
S.~Messuti, V.~R\"odl, and M.~Schacht, \emph{Packing minor-closed families of
  graphs into complete graphs}, J. Combin. Theory Ser. B \textbf{119} (2016),
  245--265.

\bibitem{MPS}
R.~Montgomery, A.~Pokrovskiy, and B.~Sudakov, \emph{Embedding rainbow trees
  with applications to graph labelling and decomposition}, arXiv:1803.03316.

\bibitem{Pluecker}
J.~Pl\"ucker, \emph{System der analytischen {G}eometrie, auf neue
  {B}etractungsweisen gegr\"undet, und insbesondere eine ausf\"uhrliche
  {T}heorie der {C}urven dritter {O}rdnung enthalend}, Duncker und Humboldt,
  Berlin, 1835.

\bibitem{RCW}
D.~K. Ray-Chaudhuri and R.~M. Wilson, \emph{Solution of {K}irkman's schoolgirl
  problem}, Combinatorics ({P}roc. {S}ympos. {P}ure {M}ath., {V}ol. {XIX},
  {U}niv. {C}alifornia, {L}os {A}ngeles, {C}alif., 1968), Amer. Math. Soc.,
  Providence, R.I., 1971, pp.~187--203.

\bibitem{Ringel}
G.~Ringel, \emph{{Problem 25}}, {Theory of Graphs and its Applications (Proc.
  Int. Symp. Smolenice 1963)}, Czech. Acad. Sci., Prague, 1963.

\bibitem{RodlNibble}
V.~R{\"o}dl, \emph{{On a packing and covering problem}}, European J. Combin.
  \textbf{6} (1985), no.~1, 69--78.

\bibitem{Steiner}
J.~Steiner, \emph{Combinatorische aufgabe}, Journal f{\"u}r die reine und
  angewandte Mathematik \textbf{45} (1853), 181--182.

\bibitem{Wilson75}
R.~M. Wilson, \emph{An existence theory for pairwise balanced designs. {III}.
  {P}roof of the existence conjectures}, J. Combinatorial Theory Ser. A
  \textbf{18} (1975), 71--79.

\end{thebibliography}

\end{document}